\documentclass[reqno,11pt]{amsart}
\usepackage{amsmath,amssymb,mathrsfs,amsthm,amsfonts}
\usepackage[inline]{enumitem}
\usepackage[usenames,dvipsnames]{xcolor}
\usepackage{hyperref}
\usepackage[percent]{overpic}
\usepackage{comment}
\usepackage{algorithm}
\usepackage{algorithmic}
\usepackage{mathtools}
\usepackage{subcaption}
\usepackage{tikz}
\usepackage{bm}
\usetikzlibrary{calc}
\usetikzlibrary{arrows.meta,backgrounds}
\usepackage{multirow,array,longtable,booktabs}
\usetikzlibrary{arrows}

\hypersetup{
	colorlinks=true, linkcolor=blue,
	citecolor=ForestGreen
}
\usepackage[paper=letterpaper,margin=1in]{geometry}

\newtheorem{theorem}{Theorem}[section]
\newtheorem{lemma}[theorem]{Lemma}
\newtheorem{proposition}[theorem]{Proposition}

\newtheorem{corollary}[theorem]{Corollary}
\theoremstyle{definition}
\newtheorem{definition}[theorem]{Definition}

\newtheorem{remark}[theorem]{Remark}
\numberwithin{equation}{section}
\allowdisplaybreaks
\usepackage{acronym}

\acrodef{LDP}{Large Deviation Principle}

\newcommand{\E}{\mathbf{E}}	

\newcommand{\be}{\begin{equation}}
\newcommand{\ee}{\end{equation}}
\newcommand{\bea} {\begin{array}{rl}}
\newcommand{\eea} {\end{array}}
\newcommand{\bepa}{\left\{ \begin{array}{l}}
\newcommand{\eepa} {\end{array}\right.}
\makeatletter

\newcommand\norm[1]{\left\Vert#1\right\Vert}

\newcommand{\R}{\mathbb{R}} 
\newcommand{\N}{\mathbb{N}}

\newcommand{\ds}{\displaystyle}

\newcommand{\T}{\mathbb{T}}

\renewcommand{\hat}{\widehat}

\renewcommand{\bar}{\overline}

\usepackage{graphicx}

\newcommand{\bd}{{\bf d}}
\newcommand{\ov}{\overline}
\newcommand{\bx}{\bm{x}}

\newcommand{\bX}{\bm{X}}

\newcommand{\cP}{\mathcal{P}}

\newcommand{\eps}{\epsilon}

\newcommand{\tr}{\text{tr}}

\newcommand{\wt}{\widetilde}

\newcommand{\by}{\bm{y}}

\newcommand{\balpha}{\bm{\alpha}}

\newcommand{\bzero}{\bm{0}}
\newcommand{\sub}{\cP_{\text{sub}}}
\newcommand{\cN}{\mathcal{N}}

\title[MFC with absorption]{Mean field control with absorption}

\author[P. Cardaliaguet]{Pierre Cardaliaguet
\address{(P. Cardaliaguet) Universit\'e Paris Dauphine, Place du Mar\'echal de Lattre de Tassigny, 75016 Paris, France
}\email{cardaliaguet@ceremade.dauphine.fr
}}

\author[J. Jackson]{Joe Jackson\address{(J. Jackson) The University of Chicago, Eckhart Hall, 5734 S University Ave, Chicago, IL 60637, USA}
\email{jsjackson@uchicago.edu}}
 
\author[P.E. Souganidis]{Panagiotis E. Souganidis\address{(P.E. Souganidis) The University of Chicago, Eckhart Hall, 5734 S University Ave, Chicago, IL 60637, USA}
\email{souganidis@uchicago.edu}}

\thanks{J. Jackson was supported by the NSF under Grant No. DMS2302703. P.C. was partially supported by P.S’s Air Force Office for Scientific Research grant FA9550-18-1-0494 and by the Agence Nationale de la Recherche (ANR), project ANR-22-CE40-0010 COSS. P. E. Souganidis was partially supported by the NSF grants  DMS-2452972 and DMS-2153822 , the Office for Naval Research grant N000141712095 and the Air Force Office for Scientific Research grant FA9550-18-1-0494.}

\begin{document}
	\begin{abstract}
            In this paper we study a mean field control problem in which particles are absorbed when they reach the boundary of a smooth domain. The value of the $N$-particle problem is described by a hierarchy of Hamilton-Jacobi equations which are coupled through their boundary conditions. The value function of the limiting problem; meanwhile, solves a Hamilton-Jacobi equation set on the space of sub-probability measures on the smooth domain, i.e. the space of non-negative measures with total mass at most one. Our main contributions are (i) to establish a comparison principle for this novel infinite-dimensional Hamilton-Jacobi equation and (ii) to prove that the value of the $N$-particle problem converges in a suitable sense towards the value of the limiting problem as $N$ tends to infinity. 
	\end{abstract}
	
	
	\maketitle
	{
		}
     \setcounter{tocdepth}{1}
\tableofcontents

\section{Introduction.}

The aim of this paper is to study a stochastic control problem in which a central planner controls a large number of particles, each of which is absorbed if it hits the boundary of  a smooth, bounded domain $\Omega \subset \R^d$. Similar models have been studied recently in the context of mean field games, but this work is the first paper to study the convergence as the (initial) number of particles increases. The main novelty of this problem compared to classical mean field control is that the total mass in the system can decrease over time. We note that absorption is not the only reasonable mechanism through which mass can decrease, and this paper is the first in a series devoted to a rigorous treatment of such problems via viscosity solutions methods. In particular, in the forthcoming work \cite{CJSstopping}, the authors will build on the methods introduced here to treat a problem in which particles are removed at the discretion of the central planner, through the choice of stopping times.

\subsection{Problem statement.}
%
%
 For each $N \in \N$ (the initial number of particles) and $K \in \N$ (the number of remaining particles), we are interested in the stochastic control problem whose value function $V^{N,K} : [0,T] \times \ov{\Omega}^K \to \R$ is given, for each $t_0 \in [0,T]$ and $\bx_0 = (x_0^1,...,x_0^K) \in \ov{\Omega}^K$, by the formula
\begin{align} \label{def.vnn}
    V^{N,K}(t_0,\bx_0)  = \inf_{\bm \alpha} \E\bigg[ \int_{t_0}^T \Big( \frac{1}{N} \sum_{i = 1}^K L\big(X_t^i, \alpha_t^i \big)1_{t < \tau^i} + F\big(m_t^{N,K}\big) \Big)dt  + G\big(m_T^{N,K}\big) \bigg]. 
\end{align}
In \eqref{def.vnn}, the $\ov{\Omega}^K$-valued state process $\bX = (X^1,...,X^K)$ evolves according to 
\begin{align} \label{x.dynamics}
    dX_t^i = \alpha_t^i dt + \sqrt{2} dW_t^i, \quad X^i_{t_0} = x_0^i, 
\end{align}
the stopping time $\tau^i$ is the first time the $i$-th component $X^i$ of the process $\bX$ hits $\partial \Omega$, that is, 
\begin{align*} 
\tau^i = \inf\{ t \geq t_0 : X_t^i \in \partial \Omega \},
\end{align*}
and 
$m_t^{N,K}$ is a random element of $\sub = \sub(\Omega)$,  the space of sub-probability measures on $\Omega$,  given by 
\begin{align*}
    m_t^{N,K} = \frac{1}{N} \sum_{i = 1}^K \delta_{X_t^i} 1_{t < \tau^i}. 
\end{align*}
The infimum in the definition of $V^{N,K}$ is taken over all $(\R^d)^K$-valued processes $\balpha = (\alpha^1,...,\alpha^K)$ which are square-integrable and progressively measurable with respect to the filtration generated by the independent Brownian motions $W^1,...,W^K$.

The data for this family of optimization problems consists of three functions 
\begin{align*}
    L : \R^d \times \R^d \to \R, \quad F, G : \sub \to \R.
\end{align*}
In addition, the Lagrangian $L$ determines a Hamiltonian $H : \R^d \times \R^d \to \R$ via the standard formula
\begin{align*}
    H(x,p) = \sup_{a \in \R^d} \Big( - L(x,a) - a \cdot p\Big).
\end{align*}
Precise assumptions on $L$, $F$, and $G$ will be made in Section \ref{sec.prelim} below.

 We emphasize that we are primarily interested in the optimization problem whose value function is $V^{N,N}$, but the entire collection $(V^{N,K})_{K = 1,...,N}$ are required in order to obtain an analytical description of the value. In particular, for each fixed $N \in \N$, the collection of value functions $(V^{N,K})_{K = 1,...,N}$ will solve a ``hierarchy" of finite-dimensional HJB equations of the form 
\begin{align} \label{hjbnk} \tag{$\text{HJB}_{N,K}$}
    \begin{cases}
      \ds   - \partial_t V^{N,K} - \sum_{i = 1}^K \Delta_{x^i} V^{N,K} + \frac{1}{N} \sum_{i = 1}^K H\big(x^i, ND_{x^i} V^{N,K} \big)  = F\big(m_{\bx}^{N,K}\big), \quad (t,\bx) \in [0,T) \times \Omega^K, \vspace{.2cm} \\ \ds 
        V^{N,K}(T,\bx) = G\big(m_{\bx}^{N,K}\big), \quad \bx \in \Omega^K, \vspace{.2cm} \\ \ds
        V^{N,K}(t,\bx) = V^{N,K-1}(t,\bx^{-i}), \quad \text{ if } x^i \in \partial \Omega, \quad K = 2,...,N,
       \vspace{.2cm} \\
       \ds  V^{N,1}(t,x) = V^{N,0}(t), \quad x \in \partial \Omega,  
    \end{cases}
\end{align}
where $V^{N,0} : [0,T] \to \R$ is defined by
\begin{align} \label{hjbnk0} \tag{$\text{HJB}_{N,0}$}
\ds  V^{N,0}(t) \coloneqq G(\bm 0) + (T-t) F(\bzero) \text{ for } t\in [0,T].
\end{align}
In \eqref{hjbnk} and throughout the paper, we use the following notational conventions: For $\bx \in \ov{\Omega}^K$ and $i \in \{1,...,K\}$, we write $\bx^{-i}$ for the element of $\ov{\Omega}^{K-1}$ obtained by removing the $i^{\text{th}}$ coordinate of $\bx$. That is, if $\bx = (x^1,...,x^K)$, then $\bx^{-i} = (x^1,...,x^{i-1},x^{i+1},...,x^K)$.  In addition, given $K \in \{1,...,N\}$ and $\bx \in \ov{\Omega}^K$, we write $m_{\bx}^{N,K} = \frac{1}{N} \sum_{i = 1}^K \delta_{x^i} \in \sub$.  Finally, $\bm{0}$ represents the zero measure in $\sub$. 

The lateral boundary conditions
\begin{align}\label{compatcond}
    V^{N,K}(t,\bx) = V^{N,K-1}(t, \bx^{-i}) \  \text{ if }  \  x^i \in \partial \Omega, \qquad  V^{N,1}(t,x) = V^{N,0}(t), \,\, \text{ if } x \in \partial \Omega
\end{align}
appearing in \eqref{hjbnk} provide the only coupling between the equations, and can be understood in terms of dynamic programming; when a particle reaches the boundary, the central planner faces a new optimization problem with one fewer particle.



Formally, we expect that, by sending $N \to \infty$, we should obtain a MFC problem with  value function $U : [0,T] \times \sub \to \R$ given by 
\begin{equation}\label{valuefunction}
    U(t_0,m_0) = \inf_{(m,\alpha)\in \mathcal A(t_0,m_0)} \bigg\{ \int_{t_0}^T \Big( \int_{\Omega} L\big( x, \alpha(t,x) \big)m_t(dx) + F(m_t) \Big) dt + G(m_T) \bigg\}, 
\end{equation}
where $\mathcal A(t_0,m_0)$ is the set of pairs $(m,\alpha)$ consisting of a curve $m: [t_0,T] \to  \sub$ and a measurable map $\alpha : [t_0,T] \times \Omega \to \R^d$ such that 
\begin{align*}
    \int_{t_0}^T \int_{\Omega} |\alpha(t,x)|^2 m_t(dx) dt < \infty, 
\end{align*}
and $m$ satisfies, in the weak sense, the initial-boundary value problem
\begin{align*}
    \partial_t m = \Delta m - \text{div}(m\alpha) \ \text{in} \  (t_0,T) \times \Omega,  \quad m = 0 \text{ on } (0,T) \times \partial \Omega, \quad m_{t_0} = m_0.
\end{align*}
The last point means that, for any $t_0 \leq t_1 < t_2 \leq T$ and any test function $\phi$ satisfying 
\begin{align*}
    \phi \in C^{1,2}([t_1,t_2] \times \ov{\Omega}), \quad \phi|_{[t_1,t_2] \times \partial \Omega} = 0, 
\end{align*}
we have 
\begin{align*}
    \int_{\Omega} \phi(t_2,x) m_{t_2}(dx) - \int_{\Omega} \phi(t_1,x) m_{t_1}(dx) = \int_{t_1}^{t_2} \int_{\Omega} \Big(\partial_t \phi+ \Delta \phi + \alpha \cdot D \phi\Big)  m_t(dx) dt.
\end{align*}
This is a standard MFC problem, except for the fact that, since $m$ solves  a Fokker-Planck with zero Dirichlet boundary conditions, its  total mass can decrease over time. 

We expect that the value function $U$ satisfies the infinite-dimensional Hamilton-Jacobi equation 
\begin{align} \label{hjbinf} \tag{$\text{HJB}_{\infty}$}
    \begin{cases}
        \ds - \partial_t U(t,m) - \int_{\Omega} \Delta_x \frac{\delta U}{\delta m}(t,m,x) m(dx) + \int_{\Omega} H\Big(x,  D_x \frac{\delta U}{\delta m}(t,m,x) \Big) m(dx) \vspace{.2cm} \\
       \hspace{6cm}  = F( m ) \ \ \text{in}  \ \  [0,T) \times \sub, 
         \\
       \ds  U(T,\cdot) = G  \ \text{in} \  \sub, \ \ \text{ and} \ \   \frac{\delta U}{\delta m}= 0 \ \text{in} \  [0,T) \times \sub \times \partial \Omega.   \end{cases}
\end{align}

We note that in \eqref{hjbinf}, the analogue of the boundary condition $V^{N,1}(t,x) = V^{N,0}(t) = G(\bzero) + (T-t) F(\bzero)$ for $x \in \partial \Omega$ is enforced by the equation. Indeed,  setting $m = \bzero$, we find $- \partial_t U(t,\bzero) = F(\bzero)$,
which, together with the terminal conditions, gives $U(t,\bzero) = G(\bzero) + (T-t) F(\bzero)$. We also claim that the condition
\begin{align} \label{boundary.cond}
    \frac{\delta U}{\delta m}(t,m,x) = 0 \text{ for } x \in \partial \Omega
\end{align}
is a natural analogue of the lateral boundary condition $V^{N,K} = V^{N,K-1}$ on $[0,T] \times \partial(\Omega^K)$. Indeed, if $U$ were smooth, then $\eqref{boundary.cond}$ would guarantee that $\lim_{x^i \to \partial \Omega} U(t,m_{\bx}^{N,K}) = U(t,m_{\bx^{-i}}^{N,K-1})$,
so that the maps $U^{N,K}(t,\bx) = U(t,m_{\bx}^{N,K})$
satisfy the same boundary conditions as $(V^{N,K})_{K = 1,...,N}$.

\subsection{The main results.}

We present next our main results, which, roughly speaking, show that $U$ is the unique viscosity solution of (a truncated version of) \eqref{hjbinf}, and $V^{N,N}$ converges to $U$.  The precise technical assumptions on $L$, $F$, and $G$ will be discussed in Section \ref{sec.prelim}, along with the definition of the metric $\bd$ which appears in the statement of Theorem \ref{thm.value}. Meanwhile, the definition of viscosity sub/supersolutions of \eqref{hjbinf} is introduced in Section \ref{sec.visc}. We begin by stating a comparison result.

\begin{theorem}[Comparison] \label{thm.comparison}
Suppose that \eqref{assump.main} and \eqref{assump.main1} hold, and that in addition $H$ is Lipschitz. If  $V^-$ is  a viscosity subsolution to \eqref{hjbinf} and  $V^+$ is a viscosity supersolution,  then $V^-(t,m) \leq V^+(t,m)$ for each $t \in [0,T]$ and $m \in \sub$. 
\end{theorem}

Notice that Theorem \ref{thm.comparison} requires an extra technical condition that $H$ is Lipschitz. Fortunately, we can prove a priori Lipschitz estimates which allow to truncate the Hamiltonian, so this assumption is not required when we prove the convergence of $V^{N,K}$ to $U$. To clarify this point, we set for each $R > 0$ the Hamiltonian 
\begin{align*}
    H^R(x,p) = \sup_{a \in B_R} \Big\{ - L(x,a) - a \cdot p \Big\},
\end{align*}
which is globally Lipschitz, and consider the truncated equation
\begin{align} \label{hjbinfR} \tag{$\text{HJB}_{\infty,R}$}
    \begin{cases}
        \ds - \partial_t U(t,m) - \int_{\Omega} \Delta_x \frac{\delta U}{\delta m}(t,m,x) m(dx) + \int_{\Omega} H^R\Big(x,  D_x \frac{\delta U}{\delta m}(t,m,x) \Big) m(dx) \vspace{.2cm}  \\
       \hspace{6cm}  = F( m ) \ \ \text{in}  \ \  [0,T) \times \sub, 
        \\
       \ds  U(T,\cdot) = G  \ \text{in} \  \sub, \ \ \text{ and} \ \   \frac{\delta U}{\delta m}= 0 \ \text{in} \  [0,T) \times \sub \times \partial \Omega.   \end{cases}
\end{align}

Our next result allows us to identify the value function $U$ as the unique solution of a PDE.

\begin{theorem}[PDE characterization of the value] \label{thm.value}
    Assume \eqref{assump.main} and \eqref{assump.main1}. Then there is a constant $C$ such that, for all $t,s \in [0,T]$, $m,n \in \sub$, 
    \begin{align*}
       | U(t,m) - U(s,n) | \leq C \Big(|t-s|^{1/2} + \bd(m,n) \Big), 
    \end{align*}
    and, moreover,  there is a constant $R_0 > 0$ such that, for all $R > R_0$, $U$ is the unique viscosity solution to \eqref{hjbinfR}. 
\end{theorem}

Our final contribution is to show that the value functions $(V^{N,K})_{K= 1,...,N}$ converge as $N \to \infty$ towards the value function $U$. 

\begin{theorem}[Convergence] \label{thm.conv}
      Assume \eqref{assump.main} and \eqref{assump.main1}.  Then the solutions  $(V^{N,K})_{K = 1,...,N}$ to \eqref{hjbnk} converge, as $N \to \infty$,  towards $U$, in the sense that 
    \begin{align*}
        \lim_{N \to \infty} \max_{K= 1,...,N} \sup_{(t,\bx) \in [0,T] \times \Omega^K} \Big| U(t,m_{\bx}^{N,K}) - V^{N,K}(t,\bx) \Big| = 0.
    \end{align*}
\end{theorem}

\subsection{Ideas of the proof.}
We discuss next the main arguments used to establish the comparison principle for \eqref{hjbinf} (Theorem \ref{thm.comparison}) and the convergence result (Theorem \ref{thm.conv}), with an emphasis on the challenges created by the boundary.
Unsurprisingly, our proof of Theorem \ref{thm.comparison} involves doubling variables. The challenge is to find penalizations which are compatible with the boundary. 
 
On the one hand, because we work with sub-probability measures, using the $2$-Wasserstein distance as a penalization (as in \cite{bertucci2023} and \cite{daudinseeger}) is not possible. On the other hand, working with smoother metrics, like those obtained by embedding $\sub$ into a sufficiently negative Sobolev space (as in \cite{SonerYanSICON, SonerYanAMO, BayraktarEkrenZhangCPDE, DJS2025}) presents more subtle challenges related to the boundary. 

In the end, we work with a much rougher metric.  Given a sub-solution $V^-$ and a super-solution $V^+$, we study an optimization problem of the form
\begin{align} \label{doubling.intro}
    \sup_{t,s \in [0,T], \, m,n \in \sub \cap L^2(\Omega)} \Big\{ V^-(t,m) - V^+(s,n) - \frac{1}{2\eps} \|m - n\|_{H^{-1}}^2 - \delta \|m\|_2^2 - \delta \|n\|_2^2 - \lambda(T-t) \Big\}, 
\end{align}
where $H^{-1} = (H_0^1(\Omega))^*$ and $\|m\|_2$ indicates the $L^2$-norm of the density of $m$ (see Section \ref{sec.prelim} for more details). The additional penalization by the squared $L^2$ norms of $m$ and $n$ brings important compactness properties, and, as discussed in Remark \ref{rmk.viscosityintuition}, also plays a role in the enforcement of the boundary condition \eqref{boundary.cond}. We follow a standard strategy, sending first $\eps$ and then $\delta$ to zero, in order to conclude that indeed $V^- \leq V^+$. For this strategy to succeed, a careful analysis of various error terms is needed, and here the choice of the norm $\| \cdot \|_{H^{-1}}$ is key.

 Formally, we have
\begin{align*}
    \frac{\delta}{\delta m} \Big[m \mapsto \frac{1}{2\eps} \|m - n_0\|_{H^{-1}}^2 \Big] = f, 
\end{align*}
where $f$ is the unique solution in $H_0^1(\Omega)$ of the PDE 
\begin{align*}
    f -  \Delta f = \frac{m - n_0}{\eps} \ \text{in} \ \Omega  \quad f=0 \ \text{on} \ {\partial\Omega}.
\end{align*}
In particular, the fact that the linear derivative vanishes on the boundary allows to justify an integration by parts in the analysis of the maximum point of \eqref{doubling.intro}. Similar ideas can be found in \cite{BertucciLionsSoug} and the lectures \cite{Lionsvideo}. 

 A key step in proving the convergence result (Theorem 1.3) is
to establish an equicontinuity estimate for the functions $(V^{N,K})_{K=1,...,N}$. Indeed, we will
prove in Theorem 4.1 that there is a constant $C$ which is independent of $N$ such that, for each $N \in \N$, $K,M \in \{1,...,N\}$, $t,s \in [0,T]$, $\bx \in \Omega^K$, $\by \in \Omega^M$, 
\begin{align} \label{vnk.lip.intro}
    |V^{N,K}(t,\bx) - V^{N,M}(s,\by)| \leq C \Big(|t-s|^{-1/2} + \bd(m_{\bx}^{N,K}, m_{\by}^{N,M}) \Big).
\end{align}
In the case without boundary, similar uniform in $N$ Lipschitz bounds can be obtained as a fairly straightforward consequence of the maximum principle, see, for example,  \cite[Lemma 3.1]{cdjs2023}. The argument here, carried out in subsection \ref{subsec.vnkLip}, is much more involved, and involves building barrier functions for $V^{N,K}$ which scale appropriately in $N$. In addition, after these barrier functions are used to obtain estimates on the boundary, some care is needed to propagate this bound to the interior in a way which again does not depend on $N$. This is done in Proposition \ref{prop.d1Lip} via a doubling of variables argument which treats all $N$ of the functions $(V^{N,K})_{K = 1,...,N}$ simultaneously.
 
Notice that in the doubling of variables argument outlined above, we need to work with test functions of the form 
\begin{align*}
    m \mapsto \Phi(t,m) + \delta \|m\|_2^2, 
\end{align*}
where $\Phi(t,m) = \frac{1}{2\eps} \|m - n_0\|_{-1}^2 + \lambda(T-t)$ for fixed $n_0 \in \sub \cap L^2(\Omega)$. This $\Phi$ is smooth as a function on $H^{-1}$, but to apply Definition \ref{defn.viscositysoln}, we need $\Phi$ to be a smooth test function in the sense of Definition \ref{def.testfunction}, and this is a much stronger condition. Proposition \ref{prop.equiv} explains how to circumvent this issue via a smoothing procedure, which is made more delicate by the presence of the boundary.

A similar regularization procedure is also required in the proof of the convergence result. In particular, the estimate \eqref{vnk.lip.intro} allows one to produce subsequential ``limit points" of the $V^{N,K}$ (see Definition \ref{def.limitpoint}). Together with the comparison principle from Theorem \ref{thm.comparison}, this compactness result reduces the convergence problem to showing that all such limit points are in fact viscosity solutions of \eqref{hjbinf} or, more precisely, of the truncated version \eqref{hjbinfR}. The natural route to obtain such a result involves projecting the relevant test functions, which take the form $\Psi(t,m) = \Phi(t,m) + \delta \|m\|_2^2$ down to finite dimensions, and looking at an optimization problem like 
\begin{align*}
    \max_{K = 0,...,N} \sup_{(t,\bx) \in [0,T] \times (\ov{\Omega})^K} \Big\{ V^{N,K}(t,\bx) - \Psi(t,m_{\bx}^{N,K}) \Big\}.
\end{align*}
But this does not make sense because we cannot compute the $L^2$ norm of an empirical measure. This necessitates another smoothing procedure, in which we regularize the $L^2$ norm, then send $N \to \infty$, and, finally, let the regularization parameter tend to zero. This is the subject of Proposition \ref{prop.subsol}.

\subsection{Related literature.}
Our results lie at the intersection of several active streams of literature; in particular, the study of mean field models with absorption, the theory of viscosity solutions for Hamilton-Jacobi equations on spaces of measures, and the convergence problem in MFC.

A number of papers in recent years have studied interacting particle systems with absorbing boundary conditions in the absence of any control. For example, several authors have used particle systems on $(0,\infty)$ with absorption at $0$ to model systemic risk in finance,  absorption representing the default of a financial institution; we refer to \cite{NadtochiyShkolnikov2019, NadtochiyShkolnikov2020, HamblyLedgerSojmark2019, LedgerSojmark2021} and the references therein for more details on such particle systems with and without common noise. These models typically feature singular interactions which we do not consider here, but the idea that agents or particles must leave the system when they reach some boundary is similar. The recent preprint \cite{GuoTomasevic} discusses a more standard weakly interacting particle system with an absorbing boundary, and \cite{HamblyLedger2017} treats a similar particle system with a common noise.

There have also been some recent efforts to study mean field games with absorption; see, for example, \cite{CampiFischer2018, CampiGhioLivieri2021, BurzoniCampi23} for probabilistic perspectives and \cite{GraberSircar2023} for the  study of the master equation for a particular model coming from economics. Our model can be viewed as a ``cooperative" analogue of the MFG model introduced in \cite{CampiGhioLivieri2021}. However, we emphasize that the existing works on MFGs with absorption focus on understanding the limiting model, and do not treat the convergence problem.

Hamilton-Jacobi equations on spaces of measures have received a huge amount of attention in the past few years, see e.g. \cite{BayraktarEkrenZhangCPDE, BayraktarEkrenHeZhang, BCEQTZ, touzizhangzhou, bertucci2023, confortihj, CKT23b, CKTT24, daudinseeger, DJS2025}. Typically, the main goal of these works is to obtain a comparison principle for viscosity (sub/super-)solutions of such equations. In the setting of mean field control; for instance, such comparison results allow one to characterize the value function of a MFC problem as the unique viscosity solution of a corresponding Hamilton-Jacobi equation. Compared to the equations studied in previous works, the novelties of \eqref{hjbinf} are (i) the fact that it is set on a space of sub-probability measures, rather than a space of probability measures, and (ii) the non-standard boundary conditions \eqref{boundary.cond}.

In the setting of standard mean field control, the hierarchy \eqref{hjbnk} is replaced by a single Hamilton-Jacobi-Bellman equation. The solution to this equation is the value function $V^N : [0,T] \times (\R^d)^N \to \R$ of an $N$-particle control problem. Meanwhile, the limiting value function $U : [0,T] \times \cP_2(\R^d) \to \R$ solves a Hamilton-Jacobi equation similar to \eqref{hjbinf}. In this setting, it is expected that $V^N$ converges to $U$, in the sense that, for $N$ large, $V^N(t,\bx) \approx U(t,m_{\bx}^N).$

For standard MFC problems, the convergence of $V^N$ to $U$ is very well understood. The first results of this type were obtained via probabilistic compactness arguments in \cite{budhiraja2012, Lacker2017}, and these techniques have since been extended in various directions in \cite{DjetePossamaiTan, djete2022extended}. More in the spirit of the present article, \cite{GangboMayorgaSwiech, mayorgaswiech} obtained the convergence of $V^N$ to $U$ for models with a purely common noise, by passing to the limit directly at the level of the PDEs satisfied by $V^N$ and $U$. The key point is to obtain appropriate equicontinuity estimates on the sequence $(V^N)_{N \in \N}$, and show that all of its ``limit points" solve the limiting Hamilton-Jacobi equation satisfied by $U$. Under appropriate technical conditions, it is also possible to quantify the convergence of $V^N$ to $U$ by using viscosity solutions techniques; see \cite{bayraktarcecchinchakrabory, cdjs2023, cjms2023, BayraktarEkrenZhangQuant, ddj2023, CDJM}.

\subsection*{Organization of the paper.} In Section \ref{sec.prelim} we introduce most of the  notation used in the paper and  present  some  preliminary facts. In Section \ref{sec.visc}  we give the definition for the viscosity solution to \eqref{hjbinf} and prove a technical fact that is then used in 
Section \ref{sec.comparison} to prove Theorem \ref{thm.comparison}, the comparison principle for \eqref{hjbinf}. Section \ref{sec.vnk} is devoted to the properties of the functions $V^{N,K}$ introduced above. The main results there are an equicontinuity estimate (Theorem \ref{thm.vnk.lip}) and a verification that every ``limit" point of the functions $(V^{N,K})_{K= 1,...,N}$ is a viscosity solution of \eqref{hjbinf} (Proposition \ref{prop.subsol}). Finally, in Section \ref{sec.u}, we verify that the value function $U$ is Lipschitz continuous and satisfies \eqref{hjbinf}, which allows us to complete the proof of the convergence result (Theorem \ref{thm.conv}).

\section{Notation, preliminaries, and assumptions.} \label{sec.prelim}
\subsection{Spaces of functions and measures on $\Omega$.}
We work throughout the paper with a bounded domain $\Omega \subset \R^d$  {with $C^3$ boundary} . We denote by $\ov{\Omega}$ the closure of $\Omega$, and by $\partial \Omega = \ov{\Omega} \setminus \Omega$ the boundary of $\Omega$.  We write $d_{\partial \Omega} : \R^d \to \R$ for the distance to the boundary $\partial \Omega$. For each $\eps > 0$, we define 
\begin{align*}
    \cN_{\eps} = \{x \in \Omega : d_{\partial \Omega}(x) < \eps\}. 
\end{align*}
Thus $\cN_{\eps}$ is an open ``collar" around $\partial \Omega$, contained in $\Omega$. It will also be useful at times to work with the signed distance function $d_{\partial \Omega}^{\text{s}} : \R^d \to \R$ given by 
\begin{align*}
    d_{\partial \Omega}^{\text{s}}(x) = \begin{cases}
        d_{\partial \Omega}(x) & x \in \ov {\Omega}, 
        \\
        -d_{\partial \Omega}(x) & x \in (\ov \Omega)^c.
    \end{cases}
\end{align*}
We recall that because $\partial \Omega$ is $C^3$, $d^{\text{s}}_{\partial \Omega}$ is $C^3$ in a neighborhood of $\partial \Omega$, and in particular for $\eps$ small enough $d_{\Omega}$ is $C^3$ in $\cN_{\eps}$, with bounded derivatives up to order $3$. We also recall that if $d_{\partial \Omega}$ is differentiable at a point $x \in \Omega$, then there is a unique point $y = \pi_{\partial \Omega}(x) \in \partial \Omega$ such that $d_{\partial \Omega}(x) = |x-y|$, and we have
\begin{align*}
    D d_{\partial \Omega}(x) = \frac{x - \pi_{\partial \Omega}(x)}{|x - \pi_{\partial \Omega}(x)|}. 
\end{align*}
In particular, $|D d_{\partial \Omega}(x)| = 1$ at any point $x$ of differentiability.  

We write  $\cP(\ov{\Omega})$ for the set of probability measures on $\ov{\Omega}$, and  recall that   $\sub$ stands for  the set of sub-probability measures on $\Omega$, that is,  the set of non-negative Borel measures $m$ on $\Omega$ with $0 \leq m(\Omega) \leq 1$. The zero measure on $\Omega$ is denoted by $\bzero$. We introduce the metric $\bd$ on $\sub$, which is inherited from duality with the set of Lipschitz function which vanish on $\partial \Omega$, that is,  we define a metric $\bd$ on $\sub$ by 
\begin{align*}
    \bd(m,n) = \sup_{\phi \in E} \int_{\Omega} \phi \, d(m - n), \ \text{where} \  E \coloneqq \big\{ \phi : \ov{\Omega} \to \R  :\phi \text{ is $1$-Lipschitz, } \phi = 0 \text{ on $\partial \Omega$} \big\}. 
\end{align*}
We note that $(\sub, \bd)$ is a compact metric space, which is a straightforward consequence of the compactness of $\cP(\ov \Omega)$ with respect to the standard 1-Wasserstein distance $\bd_1$. 
We can also view $\sub$ as a (non-compact) subset of the space $\sub(\ov{\Omega})$ of sub-probability measures on $\ov{\Omega}$, endowed with the topology of weak-$*$ convergence. We say that a sequence $m_n \in \sub$ converges weak-$*$ if it converges in weak-$*$ as a sequence in $\sub(\ov{\Omega})$, that is,  if for each continuous map $\phi : \ov{\Omega} \to \R$, we have $ \int \phi dm_n \to  \int \phi dm$ as $n\to \infty.$
 
We write  $L^2 = L^2(\Omega)$ for  the space of square integrable functions on $\Omega$ with inner-product 
   $ \langle f,g \rangle_2 = \int_{\Omega} f(x) g(x) dx$ and norm $\|  \|_2$. We denote by $H_0^1 = H_0^1(\Omega)$ the Hilbert space of functions $\phi \in L^2(\Omega)$ with distributional derivative $D\phi \in L^2(\Omega)$ and $\phi|_{\partial \Omega} = 0$ in the sense of trace. The inner product in $H_0^1$ is 
\begin{align*}
    \langle \phi, \psi \rangle_{H_0^1} = \int_{\Omega} \phi(x) \psi(x) dx + \int_{\Omega} D \phi(x) \cdot D\psi(x) dx, 
\end{align*}
and the corresponding norm is denoted by $\| \cdot \|_{H_0^1}$. The dual of  $H_0^1$, that is,  the set of bounded linear functionals $p : H_0^1 \to \R$, is the space $H^{-1} = H^{-1}(\Omega)$ which 
inherits from the duality with $H_0^1$ the  inner product  $\langle \cdot, \cdot \rangle_{H^{-1}}$  and the norm $\| \cdot \|_{-1} = \| \cdot \|_{H^{-1}}$.
Given $q \in H^{-1}$ and $\phi \in H_0^1$, we sometimes write $\langle q, \phi \rangle_{-1,1}$ or $\langle \phi, q \rangle_{1,-1}$ in place of $q(\phi)$.

We will often work with subsets of $\sub$ like $\sub \cap L^2$, that is,  the set of measures $m \in \sub$ which have a square-integrable density with respect to the Lebesgue measure $dx$. When $m$ is a measure which admits a density with respect to the Lebesgue measure on $\Omega$, we abuse notation, and use $m$ also to refer to this density. Similarly, we can view $\sub \cap H^{-1}$ as the set of measures $m \in \sub$ with the property that the map $ \phi \mapsto \int_{\Omega} \phi\, dm$
defines a bounded linear functional on $H_0^{1}$, and in this case we use $m$ to refer both to the measure and the linear functional it induces.
\subsection{Calculus on $\sub$ and on $H^{-1}$.}
We say that $\Phi : \sub \to \R$ has a (continuous) linear functional derivative if there is a continuous function $ \frac{\delta \Phi}{\delta m} : \sub \times \Omega \to \R$
with the property that, for any $m,m' \in \sub$, 
\begin{align*}
    \Phi(m') - \Phi(m) = \int_0^1 \int_{\Omega} \frac{\delta \Phi}{\delta m}\big( tm' + (1-t)m, x \big) (m'-m)(dx).
\end{align*}
Higher derivatives are defined in an analogous way.  For example,  the second derivative $\frac{\delta^2 \Phi}{\delta m}$, if it exists,  satisfies 
\begin{align*}
    \frac{\delta^2 \Phi}{\delta m^2}(m,x,y) = \frac{\delta}{\delta m}\Big[\frac{\delta \Phi}{\delta m}(m,x) \Big](y).
\end{align*}
We note that, if $\frac{\delta \Phi}{\delta m}(m,\cdot)$ is uniformly continuous, then it extends continuously to all of $\ov{\Omega}$, and we will often take this extension without comment. A similar comment holds for $\frac{\delta^2 \Phi}{\delta m^2}$.

Given $\Phi : H^{-1} \to \R$, we write  $D_{H^{-1}} \Phi$ for the Frechet derivative of $\Phi$, if it exists, that is,  for $q \in H^{-1}$, $D_{H^{-1}} \Phi(q)(\cdot) \in H_0^1$ satisfies 
\begin{align*}
    \Phi(q + r) = \Phi(q) + \langle D_{H^{-1}} \Phi(q), r \rangle_{-1,1} + o(\|r\|_{H^{-1}}).
\end{align*}
Since $D_{H^{-1}}\Phi(q)(\cdot) \in H_0^1(\Omega)$, it extends (by zero) to an element of $H_0^1(\R^d)$, and we will often make use of this extension without comment. We also find it convenient at times to use the notation $D_{H^{-1}} \Phi(q,x) = D_{H^{-1}} \Phi(q)(x)$, and we  write 
\begin{align*}
    D_x D_{H^{-1}} \Phi(q,x) = D_x \big[ D_{H^{-1}} \Phi(q)(\cdot)](x)
\end{align*}
for the gradient in $x$ of the Frechet derivative $D_{H^{-1}} \Phi$. Finally, we denote by $D^2_{H^{-1}} \Phi$ the second derivative of $\Phi : H^{-1} \to \R$, if it exists, viewed as a bilinear map $D_{H^{-1}}^2 \Phi : H^{-1} \times H^{-1} \to \R$.

We write $\bx = (x^1,...,x^K)$ for the general element of $\Omega^K$ or $\ov{\Omega}^K$. If it is necessary, we can further expand the coordinates of $\bx$ as $x^i = (x^i_1,...,x^i_d) \in \Omega \subset \R^d$. If $V : \Omega^K \to \R$ (or $\ov{\Omega}^K \to \R$) is differentiable, we write $D_{x^i} V = (D_{x^i_1}V,...,D_{x^i_d} V) \in \R^d$ for the gradient of $V$ in the direction $x^i$. Similarly, if $V$ is twice differentiable, we write $D_{x^ix^j} V$ for the $d \times d$ matrix $(D_{x^i_r x^j_q} V)_{r,q = 1,...,d}$. Finally, $\Delta_{x^i} V = \tr\big(D_{x^ix^i} V\big)$ is  the Laplacian in the variable $i$. 

\subsection{Assumptions.} \label{subsc.assump}

Throughout the paper we assume  that  the Lagrangian $L : \Omega \times \R^d \to \R$, and the cost functions $F,G : \sub \to \R$  are such that 
\begin{equation} \label{assump.main}
\begin{cases} \text{for each $R > 0$, there is a constant $C_R$ such that for each $x,x' \in \Omega$, $a,a' \in B_R$}\\[1.2mm]  
\hspace{.4cm} |L(x,a) - L(x',a')| \leq C_R \big(|x - x'| + |a - a'|\big),
\\[1.2mm] \text{ the map } a \mapsto L(x,a) \text{ is convex for each fixed } x \in \Omega, \text{ and }
\\[1.2mm]
\text{ $F$ and $G$ are $C^1$ and Lipschitz in $\sub$.}
\end{cases}
\end{equation}  
Here $B_R$ denotes the ball of radius $R$ centered at the origin in $\R^d$. In addition,  the Hamiltonian $H : \Omega \times \R^d \to \R$, which is  given  by 
\begin{align*}
   H(x,p) = \sup_{a \in \R^d} \Big\{ - a \cdot p - L(x,a) \Big\},
\end{align*}
is assumed to satisfy
\begin{equation} \label{assump.main1}
\begin{cases} \text{$H \in C^2(\Omega \times \R^d)$, and there is a constant $C$ such that, for each $x \in \Omega$, $p \in \R^d$,}\\[1.2mm]
\hspace{.4cm} |H(x,p)| \leq C(1 + |p|^2), \text{ and } 
\\[1.2mm]
\hspace{.4cm} |D_pH(x,p)| + |D_x H(x,p)| \leq C(1 + |p|).
 \end{cases}
 \end{equation}
 
 We emphasize that these assumptions are made for all the main results of the paper, a fact that will not be repeated  in each statement. All the properties in $\sub$ are assumed to be with respect to the metric $\bd$ which is introduced above.


\section{Viscosity solutions.}\label{sec.visc}

The first main contribution of the paper is the development of  a well-posedness  theory for  viscosity solutions to the limiting Hamilton-Jacobi equation \eqref{hjbinf}. For this, we need to introduce a notion of smooth test function. We refer to the previous section   for the notion of derivatives appearing in the following definition.

\begin{definition} \label{def.testfunction}
    A map  $\Phi : [0,T] \times \sub \to \R$ is a \textit{smooth test function} if \\
(i)~the derivatives
        \begin{align*}
            \partial_t \Phi : [0,T] \times \sub \to \R, \quad \frac{\delta \Phi}{\delta m} : [0,T] \times \sub \times \Omega \to \R, \ \text{and} \  \frac{\delta^2 \Phi}{\delta m^2} : [0,T] \times \sub \times \Omega^2 \to \R, 
        \end{align*}
        exist and are continuous (with respect to $\bd$), and furthermore the maps
        \begin{align*}
        [0,T] \times \sub \ni (t,m) \mapsto \frac{\delta \Phi}{\delta m}(t,m,\cdot) \in C^2(\ov{\Omega}), \quad  [0,T] \times \sub \ni (t,m) \mapsto \frac{\delta^2 \Phi}{\delta m^2}(t,m,\cdot, \cdot) \in C^1(\ov{\Omega} \times \ov{\Omega})
        \end{align*}
    are continuous and bounded \footnote{To clarify, part of the assumption here is that $\frac{\delta \Phi}{\delta m}(t,m,\cdot)$ extends continuously to all of $\ov \Omega$, and likewise for $\frac{\delta^2 \Phi}{\delta m}(t,m,\cdot, \cdot)$}, and
    \newline 
(ii)~ 
    the first derivative $\frac{\delta \Phi}{\delta m}(t,m,\cdot)$ vanishes on $\partial \Omega$, i.e.
        \begin{align*}
            \frac{\delta \Phi}{\delta m}(t,m,x) = 0 \text{ for } x \in \partial \Omega.
        \end{align*}

\end{definition}

\begin{remark} \label{rmk.hminuss}
    It is useful to note that, for fixed $m_0 \in \sub$ and $s > d/2 + 2$, the function 
    \begin{align*}
        \Phi(t,m) = \|m - m_0\|_{-s}^2
    \end{align*}
    is a smooth test function, where $\|\cdot\|_{-s}$ is the $H^{-s}$ norm, obtained from duality with the standard Sobolev space $H_0^s$. This can be checked from the explicit formula $\frac{\delta \Phi}{\delta m} = 2(m - m_0)^*$, where $H^{-s} \ni q \mapsto q^* \in H_0^s$ is the duality map.
\end{remark}

We state next the definition of the viscosity solution to \eqref{hjbinf}.

\begin{definition} \label{defn.viscositysoln}
A  continuous  (with respect to $\bd$) function $V : [0,T] \times \sub \to \R$ is a (viscosity) subsolution of \eqref{hjbinf} if $V(T,m) \leq G(m)$ for each $m \in \sub$, and there exists a constant $C > 0$ such that the following holds: for any smooth test function $\Phi$, $\delta > 0$, and $(t_0,m_0) \in [0,T) \times (\sub \cap L^2(\Omega))$ such that
    \begin{align} \label{touchingfromabove}
       V(t_0,m_0) - \Phi(t_0,m_0) - \delta \|m_0\|_{2}^2 = \sup_{(t,m) \in [0,T] \times (\sub \cap L^2(\Omega))} \Big\{V(t,m) - \Phi(t,m) - \delta \|m\|_{2}^2 \Big\}, 
    \end{align}
    we have $m_0 \in H_0^1$, and 
    \begin{equation} \label{subsol.bound}
    \begin{split} 
      &  - \partial_t \Phi(t_0,m_0) + \delta \int_{\Omega} |Dm_0(x)|^2 dx - \int_{\Omega} \Delta_x \frac{\delta \Phi}{\delta m}(t_0,m_0,x) m_0(x) dx\\ 
       &\qquad  + \int_{\Omega} H\Big( x, D_x \frac{\delta \Phi}{\delta m}(t_0,m_0,x)\Big) m_0(x) dx \leq F(m_0) + C \delta \|m_0\|_{2}^2.
    \end{split}
    \end{equation}
A  continuous  (with respect to $\bd$) function $V : [0,T] \times \sub \to \R$ is a (viscosity) supersolution of \eqref{hjbinf} if $V(T,m) \geq G(m)$ for each $m \in \sub$, and there exists a constant $C > 0$ such that the following holds: for each smooth test function $\Phi$ and $\delta > 0$, any pair $(t_0,m_0) \in [0,T) \times (\sub \cap L^2(\Omega))$ such that
    \begin{align} \label{touchingfrombelow}
       V(t_0,m_0) - \Phi(t_0,m_0) + \delta \|m_0\|_{2}^2 = \inf_{(t,m) \in (0,T) \times (\sub \cap L^2(\Omega))} \Big\{V(t,m) - \Phi(t,m) + \delta \|m\|_{2}^2 \Big\}, 
    \end{align}
    we have $m_0 \in H_0^1$, and 
    \begin{equation} \label{supersol.boud}
     \begin{split} 
      &  - \partial_t \Phi(t_0,m_0) - \delta \int_{\Omega} |Dm_0(x)|^2 dx - \int_{\Omega} \Delta_x \frac{\delta \Phi}{\delta m}(t_0,m_0,x) m_0(x) dx
      \\
       &\qquad  + \int_{\Omega} H\Big( x, D_x \frac{\delta \Phi}{\delta m}(t_0,m_0,x)\Big) m_0(x) dx \geq F(m_0) - C \delta \|m_0\|_{2}^2.
    \end{split}
    \end{equation}
A viscosity solution is a function which is both a viscosity subsolution and a viscosity supersolution.
\end{definition}

\begin{remark}
   In the setting of ``standard" mean field control (so that the relevant HJB equation is set on e.g. the Wasserstein space $\cP_2(\R^d)$ or $\cP(\T^d)$), various definitions of viscosity solution have been proposed. For example, in \cite{SonerYanAMO, SonerYanSICON, BayraktarEkrenZhangCPDE, DJS2025} viscosity solutions are defined in terms of smooth test functions; this is possible because in their doubling of variables arguments, the authors penalize with a``smooth metric" which is inherited from a sufficiently negative Sobolev space. There are serious challenges in adapting this strategy to the present setting, related to the incompatibility of these negative Sobolev metrics with the boundary conditions. 
    
    On the other hand, another strategy is to use some sort of ``singular penalization" to force the optimizers in the doubling of variables argument to have some regularity. The hope is that this additional regularity helps to make sense of the differential of a more natural squared distance function in a doubling of variables argument. For example, this strategy is used in \cite{BertucciLionsSoug} and \cite{daudinseeger}, where the distance function is the 2-Wasserstein distance, and the ``singular perturbation" is the entropy functional. This leads to definitions of viscosity solutions similar to \eqref{defn.viscositysoln}. 

    In this paper, we instead use the $H^{-1}$ distance in our doubling of variables argument, because it is compatible with boundary conditions in the equation. The ``singular perturbation" $\| \cdot \|_2^2$ provides the regularity needed to make sense of the differential of the squared $H^{-1}$ distance in the doubling of variables argument below.
\end{remark}

\begin{remark} \label{rmk.viscosityintuition}
The boundary condition \eqref{boundary.cond} is encoded through the assumption that the touching point $(t_0,m_0)$ satisfies $m_0 \in H_0^1$, and, in particular,  $m_0$ vanishes on the boundary. Indeed, if $m_0$ has a density, then, formally, we  have 
    \begin{align} \label{l2deriv}
    \frac{\delta}{\delta m}  \| \cdot \|_2^2 (m_0,x) = 2m_0(x).  
    \end{align}
    Thus, if \eqref{touchingfromabove} holds and $V$ is smooth, then 
    \begin{align*}
        \frac{\delta V}{\delta m}(t_0,m_0,x) \leq \frac{\delta \Phi}{\delta m}(t_0,m_0,x) + 2\delta m_0(x), 
    \end{align*}
    with equality on the support of $m_0$. 
    Since $\Phi$ is a smooth test function, $\frac{\delta \Phi}{\delta m}(t_0,m_0,\cdot) =  0 \ \text{ on }  \partial \Omega,$ which means that $\frac{\delta V}{\delta m}(t_0,m_0,\cdot) \leq 2\delta m_0 = 0\ \text{ on}  \ \partial \Omega.$ The supersolution test enforces the other direction.
 
To motivate  the inequality  \eqref{subsol.bound}, we note that, because of \eqref{l2deriv} and the fact that $(t,m) \mapsto \Phi(t,m) + \delta \|m\|_2^2$ touches $V$ from above at $(t_0,m_0)$, we formally expect
    \begin{align} \label{formal.subsol}
        &- \partial_t \Phi(t_0,m_0) - \int_{\Omega} \Delta_x \Big( \frac{\delta \Phi}{\delta m}(t_0,m_0,x) + 2 \delta m_0(x) \Big) m_0(x) dx
       \nonumber  \\
        &\qquad + \int_{\Omega} H\Big(x, D_x \frac{\delta \Phi}{\delta m}(t_0,m_0,x) + 2\delta Dm_0(x) \Big) m_0(x) dx \leq F(m_0).
    \end{align}
    Since $m_0 \in H_0^1$, we can, again formally,  integrate by parts to get 
    \begin{align} \label{ibp}
        \int_{\Omega} \Delta_x m_0(x) m_0(x) dx = - \int_{\Omega} |Dm_0(x)|^2 dx.
    \end{align}
    If $H$ is Lipschitz, it follows that 
    \begin{align} \label{Hlipcomp}
        \int_{\Omega} &H\Big(x, D_x \frac{\delta \Phi}{\delta m}(t_0,m_0,x) + 2\delta Dm_0(x) \Big) m_0(x) dx 
       \nonumber  \\
        &\geq \int_{\Omega} H\Big(x, D_x \frac{\delta \Phi}{\delta m}(t_0,m_0,x)\Big)m_0(x) dx - C\delta \int_{\Omega} |Dm_0(x)| m_0(x) dx 
        \nonumber \\
        &\geq \int_{\Omega} H\Big(x, D_x \frac{\delta \Phi}{\delta m}(t_0,m_0,x)\Big)m_0(x) dx - \delta \int_{\Omega} |Dm_0(x)|^2 dx - C \delta \|m_0\|_2^2.
    \end{align}
    Combining \eqref{formal.subsol}, \eqref{ibp}, and \eqref{Hlipcomp}, we arrive at \eqref{subsol.bound}.    
\end{remark}


Definition \ref{defn.viscositysoln} uses  test functions which are the sum of a smooth part $\Phi$ and a singular part involving the $L^2$ norm. This is the definition which we will use directly when we verify that ``limit points" of $(V^{N,K})_{K= 1,...,N}$ solve \eqref{hjbinf}. To prove, however,  the  comparison principle for \eqref{hjbinf}, it is  more convenient to work with $\Phi \in C^{1,2}([0,T] \times H^{-1})$. This is the topic of the next proposition. 

\begin{proposition} \label{prop.equiv}
    Suppose that $V$ is a viscosity subsolution to \eqref{hjbinf}, and let $C > 0$ be the constant appearing in Definition \ref{defn.viscositysoln}. Then, for any $\Phi \in C^{1,2}([0,T] \times H^{-1})$, $\delta > 0$, and $(t_0,m_0) \in [0,T) \times (\sub \cap L^2)$ such that \begin{align} \label{touchingfromabove.frechet}
       V(t_0,m_0) - \Phi(t_0,m_0) - \delta \|m_0\|_{L^2}^2 = \sup_{(t,m) \in [0,T] \times (\sub \cap L^2(\Omega))} \Big\{V(t,m) - \Phi(t,m) - \delta \|m\|_{L^2}^2 \Big\}, 
    \end{align}
    we have $m_0 \in H_0^1$, and
    \begin{equation}\label{subsol.bound.frechet}
    \begin{split} 
      &  - \partial_t \Phi(t_0,m_0) + \delta \int_{\Omega} |Dm_0(x)|^2 dx + \int_{\Omega} D_x D_{H^{-1}} \Phi(t_0,m_0,x) \cdot Dm_0(x) dx 
        \\
       &\qquad  + \int_{\Omega} H\Big( x, D_x D_{H^{-1}} \Phi(t_0,m_0,x)\Big) m_0(x) dx \leq F(m_0) + C \delta \|m_0\|_{L^2}^2.
    \end{split}
    \end{equation}
    Similarly, if $V$ is a viscosity supersolution and $C$ is the constant appearing in Definition \ref{defn.viscositysoln}, then, for any $\Phi \in C^{1,2}([0,T] \times H^{-1})$, $\delta > 0$, and $(t_0,m_0) \in [0,T) \times (\sub \cap L^2)$ such that \eqref{touchingfromabove} holds, we have $m_0 \in H_0^1$, and 
    \begin{equation}\label{supersol.bound.frechet}
    \begin{split}
      &  - \partial_t \Phi(t_0,m_0) - \delta \int_{\Omega} |Dm_0(x)|^2 dx + \int_{\Omega} D_x D_{H^{-1}} \Phi(t_0,m_0,x) \cdot Dm_0(x) dx 
        \\
       &\qquad  + \int_{\Omega} H\Big( x, D_x D_{H^{-1}}\Phi(t_0,m_0,x)\Big) m_0(x) dx \geq F(m_0) - C \delta \|m_0\|_{L^2}^2.
    \end{split}
    \end{equation}
\end{proposition}

In other words, the proposition above  says that, if $V$ is a subsolution in the sense of Definition \ref{defn.viscositysoln}, then it also satisfies the subsolution test for test functions in the class $C^{1,2}([0,T] \times H^{-1})$. 

To prove Proposition \ref{prop.equiv}, we need a regularization procedure, which approximates a given $\Phi \in C^{1,2}([0,T] \times H^{-1})$ by smooth test functions. For this, we fix $\theta_0 > 0$ and a family $(f_{\theta})_{\theta \in (0,\theta_0)}$ of diffeomorphisms $f_{\theta} : \R^d \to \R^d$ such that 
\begin{equation*}
\begin{cases}
 f_{\theta} \in C^{\infty}, \text{and there is a constant $C_0$ which is independent of $\theta$ such that, for all  $\theta \in (0,\theta_0)$, }\\[1.2mm]
  
\hskip1.5in   \|Df_{\theta}\|_{\infty} + \|D (f_{\theta}^{-1}) \|_{\infty} \leq C_0,\\[1.5mm]
   
\text{$f_{\theta} \to \text{Id}$ \ and \ $Df_{\theta} \to I_{d \times d}$ uniformly on $\R^d$ as $\theta \to  0$,}\\[1.2mm]

\text{for each $\theta \in (0,\theta_0)$ \ $f_{\theta}$ maps $\cN_{\theta}$ into $\Omega^c$, and }
\end{cases}
\end{equation*}
where $\text{Id}$ indicates the identity function on $\R^d$.  We note that to build such a family of diffeomorphisms, it suffices to take $f_{\theta}$ of the form 
\begin{align*}
    f_{\theta}(x) = x + \theta V(x), 
\end{align*}
where $V : \R^d \to \R^d$ is any smooth, compactly supported vector field with 
\begin{align*}
     V(x) \cdot  D d^{\text{s}}_{\partial \Omega}(x) \leq - 2
\end{align*}
in a neighborhood of the boundary. Indeed, the only non-trivial point to check is that $f_{\theta}$ maps $\cN_{\theta}$ into $\Omega^c$ for small enough $\theta$, and this can be verified by taking a second-order Taylor expansion of the signed distance function $d^{\text{s}}_{\partial \Omega}$; for $\theta$ small enough and $x \in \cN_{\theta}$, 
\begin{align*}
    d^{\text{s}}_{\partial \Omega}(f_{\theta}(x)) &\leq d^{\text{s}}_{\partial \Omega}(x) + \theta V(x) \cdot D d^{\text{s}}_{\partial \Omega}(x) + C\theta^2 \leq d_{\partial \Omega}^s(x) - 2\theta + C \theta^2, 
\end{align*}
and so for $\theta$ small enough $f_{\theta}$ indeed maps $\cN_{\eps}$ into $\Omega^c$.

We also fix an even approximation to the identity $(\rho_{\eta})_{\eta > 0}$ with $\text{supp}(\rho_{\eta}) \subset B_{\eta} = B_{\eta}(0)$, that is,
$\rho_{\eta}(x) = \eta^{-d} \rho(x/\eta)$, with the 
 smooth function $\rho : \R^d \to \R$ satisfying $\rho(x) \geq 0, \ \rho(x) = \rho(-x),$ and $\int_{\R^d} \rho(x) dx = 1.$
%

\begin{lemma} \label{lem.linearoperation}
    For each $\theta \in (0,\theta_0)$, the map $\phi \mapsto \phi \circ f_{\theta}$
    is a bounded linear operator on $H_0^1(\Omega)$. Moreover, for $\theta \in (0,\theta_0)$ and $\eta \in (0,\theta/2)$, the map
    \begin{align} \label{T.theta.eta}
        T_{\theta, \eta} \phi = \rho_{\eta} * \big( \phi \circ f_{\theta} \big)
    \end{align}
    is a bounded linear operator on $H_0^1(\Omega)$.
\end{lemma}

\begin{proof}
 Given any $\phi \in H_0^1 = H_0^1(\Omega)$,  we extend it  (without changing the notation) by zero to all of $\R^d$ and note that this extension is in $H_0^1(\R^d)$. 

Since $f_{\theta}(x) \in \Omega^c$ for $x \in \cN_{\theta}$, we have $\phi \circ f_{\theta} = 0$ on $\cN_{\theta} \cup \Omega^c$, and, in view of the fact that, for $\eta < \theta/2$,  $\text{supp}(\rho_{\eta}) \subset B_{\eta}$, it follows that $\rho_{\eta} * (\phi \circ f_{\theta}) = 0$ for $x \in \cN_{\theta/2}$. Thus for any $\phi \in H_0^1(\Omega)$, we have $T_{\theta, \eta} \phi = 0$ on $\partial \Omega$. 
    
Using next the change of variables $f_{\theta}(x) = y$, we find 
    \begin{align*}
        \int_{\Omega} &|\phi \circ f_{\theta}(x)|^2 dx = \int_{\R^d} |\phi \circ f_{\theta}(x)|^2 dx 
        \\
        &= \int_{\R^d} |\phi(y)|^2 \Big(J(f_{\theta})\big(f_{\theta}^{-1}(y)\big) \Big)^{-1} dy \leq C \int_{\R^d} |\phi(y)|^2, 
    \end{align*}
    where $J(f_{\theta})$ is the Jacobian of $f_{\theta}$, and the last bound comes from the fact that $Df_{\theta}^{-1} \leq C_0 I_{d \times d}$. An analogous computation gives
    \begin{align*}
        \int_{\Omega} &\big| D \big( \phi \circ f_{\theta} \big) (x) \big|^2 dx  \leq C \int_{\R^d} |D \phi(y)|^2 dy. 
    \end{align*}
    It follows that $\| \phi \circ f_{\theta} \|_{H_0^1} \leq C \| \phi \|_{H_0^1}.$
    The fact that $T_{\theta, \eta}$ is bounded follows easily.  
 \end{proof}

Using the notation for push-forward of measures, next we denote by $(f_{\theta})_{\#}$ the bounded linear map $H^{-1} \to H^{-1}$ which is the adjoint of $\phi \mapsto \phi \circ f_{\theta}$, that is, we define, for each $ q \in H^{-1}$, 
\begin{align*}
    \langle (f_{\theta})_{\#} q, \phi \rangle_{-1,1} = \langle q, \phi \circ f_{\theta} \rangle_{-1,1}.
\end{align*}
Finally, we note that the adjoint $T_{\theta,\eta}^*$ of $T_{\theta, \eta}$ is the bounded linear map $H^{-1} \to H^{-1}$ given, for each $ q \in H^{-1}$,  by 
\begin{align*}
    S_{\theta, \eta} q = (f_{\theta})_{\#}\big(q * \rho_{\eta} \big). 
\end{align*}
Indeed, because $\rho_{\eta}$ is even, we have 
\begin{align*}
    \langle q, T_{\theta, \eta} \phi \rangle_{-1,1} = \langle q * \rho_{\eta}, \phi \circ f_{\theta} \rangle_{-1,1} = \langle (f_{\theta})_{\#}\big(q * \rho_{\eta} \big), \phi \rangle_{-1,1} = \langle S_{\theta, \eta} q, \phi \rangle_{-1,1}.
\end{align*}
We note also that if $m \in \sub$, then $\rho_{\eta} * m$ is smooth, and, in particular,  is in $H^{-1}$, and, thus, we can make sense of $S_{\theta, \eta} m$. In particular, we can set, for each $m \in \sub$,
\begin{align*}
     S_{\theta, \eta} m = (f_{\theta})_{\#}\big(m * \rho_{\eta} \big).  
\end{align*}
It is then straightforward to check that the map $\sub \ni m \mapsto S_{\theta, \eta} m \in (\sub \cap H^{-1})$ is continuous (in fact Lipschitz) with respect to the metric $\bd$, in the sense that 
\begin{align*}
  \norm{S_{\theta, \eta} m' - S_{\theta, \eta} m}_{-1} \leq C_{\theta, \eta} \bd(m',m). 
\end{align*}

We are finally ready to describe the regularization scheme.
\begin{lemma} \label{lem.Phi.theta.eta}
 For each $\theta \in (0,\theta_0)$, $\eta \in (0,\theta/2)$, and $\Phi \in C^{1,2}([0,T] \times H^{-1})$, and $(t,m) \in [0,T] \times \sub$ define 
\begin{align} \label{phi.theta.eta}
    \Phi_{\theta,\eta}(t,m) = \Phi\big(t, S_{\theta, \eta} m\big) = \Phi\Big(t, (f_{\theta})_{\#} (\rho_{\eta} * m) \Big).
\end{align}
Then $\Phi_{\theta, \eta}$ is a smooth test function in the sense of Definition \ref{def.testfunction}, and, moreover, 
\begin{align} \label{Phi.derivformulas1}
    &\frac{\delta \Phi_{\theta, \eta}}{\delta m}(t,m,x) = T_{\theta, \eta} D_{H^{-1}} \Phi\big(t,S_{\theta, \eta} m\big)
   = \rho_{\eta} * \Big[ D_{H^{-1}} \Phi \Big(t,(f_{\theta})_{\#} \big(\rho_{\eta} * m\big), f_{\theta}(\cdot) \Big) \Big](x), 
    \\[1.5mm] \label{Phi.derivformulas2}
    &\frac{\delta^2 \Phi_{\theta, \eta}}{ \delta m^2} (t,m,x,y) =  D_{H^{-1}}^2 \Phi\big( t, S_{\theta, \eta} \big)\big( S_{\theta, \eta} \delta_x, S_{\theta, \eta} \delta_y \big).
\end{align}
\end{lemma}

\begin{proof}
  The fact that $\frac{\delta \Phi_{\theta, \eta}}{\delta m}$ and $\frac{\delta^2 \Phi_{\theta, \eta}}{\delta m^2}$ exist and satisfy the formulas \eqref{Phi.derivformulas1} and \eqref{Phi.derivformulas2} is the consequence of a straightforward but tedious calculation, which we omit.
 To see that $(x,y) \mapsto \frac{\delta^2 \Phi_{\theta, \eta}}{\delta m^2}(t,m,x,y)$ is smooth, one can use the formula above for $\frac{\delta \Phi_{\theta, \eta}}{\delta m}$ to explicitly compute 
   \begin{align*}
       D^{\alpha}_x D^{\beta}_y \frac{\delta^2 \Phi_{\theta, \eta}}{\delta m^2} (t,m,x,y) = D_{H^{-1}}^2 \Phi\big(t, (f_{\theta})_{\#} (m * \rho_{\eta}) \big)\Big((f_{\theta})_{\#} (\rho_{\eta} * D^{\alpha} \delta_x, (f_{\theta})_{\#} ( \rho_{\eta} * D^{\beta} \delta_y ) \Big).
   \end{align*}
   for any multi-indices $\alpha$ and $\beta$. 
 Moreover, one can check,  using the representations \eqref{Phi.derivformulas1}, \eqref{Phi.derivformulas2} and the fact that $\sub \ni m \mapsto S_{\theta, \eta} m \in H^{-1}$ is continuous, that the maps
   \begin{align*}
    [0,T] \times \sub \ni (t,m) \mapsto \frac{\delta \Phi_{\theta,\eta}}{\delta m}(t,m,\cdot) \in C^2(\ov{\Omega}), 
    \end{align*}
and 
 \begin{align*}    
    [0,T] \times \sub \ni (t,m) \mapsto \frac{\delta^2 \Phi_{\theta, \eta}}{\delta m^2}(t,m,\cdot, \cdot) \in C^1(\ov{\Omega} \times \ov{\Omega})
    \end{align*}
    are continuous and bounded. 
  
    Finally, we note that, since  $f_{\theta}$ maps $\cN_{\theta}$ into $\Omega^c$, for each $x  \ \text{in} \  \cN_{\theta}$, we have 
    \begin{align*}
        D_{H^{-1}} \Phi\Big( t, (f_{\theta})_{\#} (\rho_{\eta} * m), f_{\theta}(x) \Big) = 0 \quad \text{for $x$ in } \cN_{\theta}, 
    \end{align*}
and,  because $\eta < \theta/2$, we see from \eqref{Phi.derivformulas1} that $\frac{\delta \Phi_{\theta,\eta}}{\delta m}(t,x,m) = 0$ in a neighborhood of $\partial \Omega$.


 \end{proof}

We proceed now with the proof of Proposition \ref{prop.equiv}.

\begin{proof}{Proof of Proposition \ref{prop.equiv}.}
    We only prove the result for subsolutions, the corresponding result for supersolutions being analogous. Moreover, all limits in the proof are as $\theta, \eta\to 0$, a fact that we will not keep repeating. 
    
 Let $V$ be a viscosity subsolution,  and assume that  $(t_0,m_0) \in [0,T) \times (\sub \cap L^2)$ is the unique optimizer of the problem 
    \begin{align} \label{subolcond.2}
        \sup_{t \in [0,T], \, m \in \sub \cap L^2} \Big\{ V(t,m) - \Phi(t,m) - \delta \|m\|_2^2 \Big\}, 
    \end{align}
    with $\Phi \in C^{1,2}([0,T] \times H^{-1})$ and  $\delta > 0$. We wish to show that $m_0 \in H_0^1(\Omega)$, and that the inequality \eqref{subsol.bound} holds. For each $\theta \in (0,\theta_0)$ and $\eta \in (0,\theta_0/2)$,  we consider the optimization problems 
  \begin{align} \label{thetaproblem}
        \sup_{t \in [0,T], \, m \in \sub \cap L^2} \Big\{ V(t,m) - \Phi_{\theta}(t,m) - \delta \|m\|_2^2 \Big\}, 
    \end{align}
  and 
\begin{align} \label{theta.eta.problem}
        \sup_{t \in [0,T], \, m \in \sub \cap L^2} \Big\{ V(t,m) - \Phi_{\theta, \eta}(t,m) - \delta \|m\|_2^2 \Big\}, 
    \end{align}  
with    $ \Phi_{\theta}(t,m) = \Phi(t,(f_{\theta})_{\#}m)$ and $\Phi_{\theta, \eta}$ as in \eqref{phi.theta.eta}. Since the rest of the proof is rather long, we divide in several steps.
\newline \newline 
\noindent 
   \textit{Step 1 - convergence of the optimizers.} We  recall that the map $q \mapsto (f_{\theta})_{\#} q$ is a bounded linear map on $H^{-1}$ and as a consequence
    \begin{align*}
        [0,T] \times (\sub \cap L^2) \ni (t,m) \mapsto V(t,m) - \Phi_{\theta}(t,m) - \delta\|m\|_2^2
    \end{align*}
    is upper semi-continuous with respect to the weak $L^2$ topology for $m$. It then follows  that the problem \eqref{thetaproblem} admits at least one optimizer.

    Next, we claim that $\Phi_{\theta} \to \Phi$ uniformly on subsets of $L^2 \cap \sub$ which are bounded in $L^2$. Indeed, for $m \in L^2 \cap \sub$, and $\phi \in H_0^1$ with $\|\phi\|_{H_0^1} \leq 1$, we have 
    \begin{align*}
        \int_{\Omega} \phi \, d \big( (f_{\theta})_{\#} m - m \big) = \int_{\Omega} \Big( \phi(f_{\theta}(x)) - \phi(x) \Big) m(x) dx \leq \|\phi \circ f_{\theta} - \phi \|_{2} \|m\|_{2}, 
    \end{align*}
    and, for $\theta$ small enough, 
    \begin{align*}
       \| \phi \circ f_{\theta} - \phi\|_2^2 &= \int_{\R^d} \Big| \int_0^1 D\phi\big(x + t(f_{\theta}(x) - x) \big) \cdot (f_{\theta}(x) - x) dt \Big|^2 dx
       \\
       &\leq \|f_{\theta} - \text{Id}\|_{\infty}^2 \int_{0}^1 \int_{\R^d} \big| \nabla \phi\big( x + t(f_{\theta}(x) - x) \big) \big|^2 dx dt
       \\
       &\leq C \|f_{\theta} - \text{Id}\|_{\infty}^2. 
    \end{align*}
 The last bound above is coming from the fact that $\|\phi\|_{H_0^1} \leq 1$ together with a change of variables and the fact that, since $Df_{\theta} \to I_{d \times d}$ uniformly,  the Jacobian of $x \mapsto x + t(f_{\theta}(x) - x)$ is bounded from below uniformly in $t$ and $\theta$, for all $\theta$ small enough. 
 
 Thus we have found that there is a constant $C$ such that, for all $\theta$ small enough, 
    \begin{align*}
        \|(f_{\theta})_{\#} m - m\|_{H^{-1}} \leq C \|f_{\theta} - \text{Id}\|_{\infty} \|m\|_{2}.
    \end{align*}
  Since $\Phi$ is Lipschitz continuous  uniformly on sets which are bounded in $L^2$, it follows  that $\Phi_{\theta} \to \Phi$ uniformly on bounded, with respect to the $L^2$ norm,  subsets of $\sub \cap L^2$. Since $(t_0,m_0)$ is the unique optimizer for \eqref{subolcond.2}, this uniform convergence allows us to conclude that if $(t_{\theta}, m_{\theta})$ is any optimizer for \eqref{thetaproblem}, then we must have $(t_{\theta}, m_{\theta}) \to (t_0,m_0)$, 
    with the convergence of $m_{\theta}$ with respect to the weak topology of $L^2$. In particular, since $t_0 < T$, we see that there is a constant $\theta_0 > 0$ such that for all $\theta < \theta_0$, any optimizer $(t_{\theta}, m_{\theta})$ for \eqref{thetaproblem} satisfies 
    \begin{align} \label{theta.noterminal}
        t_{\theta} < T. 
    \end{align}

    We now turn to the problem \eqref{theta.eta.problem}. Again, using the continuity of $m \mapsto (f_{\theta})_{\#} (\rho_{\eta} * m)$ with respect to $H^{-1}$, and hence with respect to the weak topology on $L^2$, we conclude that, for each $\theta, \eta > 0$, there exists at least one optimizer for the problem \eqref{theta.eta.problem}.

    Moreover, for fixed $\theta$  and $\phi \in H_0^1$, we have
    \begin{align} \label{f.theta.eta.conv}
        \int_{\Omega} \phi \Big( (f_{\theta})_{\#}(m * \rho_{\eta}) - (f_{\theta})_{\#} m \Big) &= \int_{\R^d} \Big( (\phi \circ f_{\theta}) * \rho_{\eta} (x) - (\phi \circ f_{\theta})\Big)  m(x) dx 
        \nonumber \\
        &\leq \norm{(\phi \circ f_{\theta}) * \rho_{\eta} - (\phi \circ f_{\theta})}_{2} \|m\|_2.
    \end{align}
    Now, since $Df_{\theta} \to I_{d \times d}$ uniformly, a change of variables shows that, for $\theta$ small enough, $\| \phi \circ f_{\theta}\|_{H_0^1} \leq 2$, so that, setting for simplicity $\psi = \phi \circ f_{\theta}$, 
    \begin{align*}
        \|\psi * \rho_{\eta} - \psi\|_2^2 &= \int_{\R^d} \Big| \int_{\R^d} \rho_{\eta}(y) \big(\psi(x) - \psi(x-y) \big) dt \Big|^2 dx
        \\
        &= \int_{\R^d} \Big| \int_{\R^d} \rho_{\eta}(y) \Big( \int_0^1 D \psi(x-y -ty) \cdot y dt \Big) \Big|^2 dx 
        \\
        &\leq \eta^2 \int_0^1 \int_{\R^d} \int_{\R^d} \rho_{\eta}(y) |D \psi(x -y + ty)|^2 dx dy dt \leq 4\eta^2.
    \end{align*}
    In particular, combining this with \eqref{f.theta.eta.conv}, we see that for each $\theta$ small enough, $(f_{\theta})_{\#}(m * \rho_{\eta}) \to (f_{\theta})_{\#} m$,
    uniformly on bounded (with respect to $L^2$) subsets of $\sub \cap L^2$. Thus $\Phi_{\theta, \eta} \to \Phi_{\theta}$ uniformly on bounded subsets as $\eta \to 0$.
    
     We conclude that, for each $\theta > 0$ small enough, and any optimizers $(t_{\theta, \eta}, m_{\theta, \eta})$ for the problem \eqref{theta.eta.problem}, $m_{\theta, \eta}$ is bounded in $L^2$ and any weak limit point (as $\eta \downarrow 0$) $(t_{\theta}, m_{\theta})$ is an optimizer for \eqref{thetaproblem}. In particular, from \eqref{theta.noterminal}, we see that for all $\theta$ small enough, there exists $\eta_0$ depending on $\theta$ such that for $\eta < \eta_0$, any optimizer $(t_{\theta, \eta},m_{\theta, \eta})$ must satisfy
    \begin{align} \label{theta.eta.noterminal}
        t_{\theta,\eta} < T. 
    \end{align}
    \textit{Step 2 - Applying the subsolution test with $\theta, \eta$ fixed:} We now fix $\theta$ and then choose $\eta$ small enough so  that \eqref{theta.eta.noterminal} holds. Then Lemma \ref{lem.Phi.theta.eta} and Definition \ref{defn.viscositysoln} allow us to conclude that, for any optimizer $(t_{\theta,\eta},m_{\theta,\eta})$ for \eqref{theta.eta.problem}, we have
    \begin{align} \label{theta.eta.subsol}
    &- \partial_t \Phi_{\theta, \eta_j}\big(t_{\theta,\eta}, m_{\theta,\eta} \big) + \int_{\R^d} Dm_{\theta, \eta} \cdot \Big(\rho_{\eta} \ast \Big[ \big( Df_{\theta}(x) \big)^T D_x D_{H^{-1}} \Phi\Big(t_{\theta,\eta},(f_{\theta})_{\#} \big(\rho_{\eta} * m_{\theta,\eta}\big), f_{\theta}(x) \Big) \Big] \Big) dx
    \nonumber \\
    &\qquad + \delta \int_{\R^d} |D m_{\theta,\eta}|^2 dx  + \int_{\Omega} H\Bigl( 
\rho_{\eta} * \Big[ \big( Df_{\theta}(x) \big)^T D_x D_{H^{-1}} \Phi\Big((f_{\theta})_{\#} \big(\rho_{\eta} * m_{\theta,\eta}\big), f_{\theta}(\cdot) \Big) \Big](x) \Big)  m_{\theta,\eta}(dx)  
\nonumber \\
&\qquad \qquad \leq C \delta \int_{\R^d} \big| m_{\theta,\eta} \big|^2 dx.
\end{align}
\textit{Step 3 - Sending $\eta \to 0$.}
By Step 1, we can choose a sequence $\eta_j \to {0}$ such that $(t_{\theta, \eta_j}, m_{\theta, \eta_j})$ converges weakly in $L^2$ to some optimizer $(t_{\theta},m_{\theta})$ for the problem \eqref{theta.eta.problem}. For simplicity, we set $(t_j, m_j) = (t_{\theta, \eta_j}, m_{\theta, \eta_j})$. 

The optimality of the $(t_j,m_j)$'s implies that $m_j$'s are bounded in $L^2$, uniformly in $j$, and   \eqref{theta.eta.subsol} yields that  $m_j$'s  are  bounded, uniformly in $j$,  in $H_0^1$. Thus, as $j\to \infty$,  
\begin{align*}
    m_j  \to m_\theta \ {\text{weakly in $H_0^1$ and  strongly in $L^2$}}.
\end{align*}
It then follows easily that, as $j\to \infty$, 
\[
   \rho_{\eta_j} * m_j  \to m_{\theta} \ \text{strongly in $H^{-1}$},\]
which implies that, as $j\to \infty$, 
   \[(f_{\theta})_{\#} \big(\rho_{\eta_j} * m_j\big) \to (f_{\theta})_{\#} m_{\theta} \ \text{strongly in $H^{-1}$} ,\] 
and, then, 
\[ \rho_{\eta_j} * D_x D_{H^{-1}} \Phi\Big((f_{\theta})_{\#} \big(\rho_{\eta_j} * m_j\big), f_{\theta}(\cdot) \Big) \to 
D_x D_{H^{-1}} \Phi\Big((f_{\theta})_{\#} \big(m_{\theta} \big), f_{\theta}(\cdot) \Big) \ \text{strongly in $L^2(\Omega)$}.\]
%
The last claim  together with the strong convergence of $m_j$ in $L^2$ and the weak convergence of $Dm_j$ in $L^2$ allow  to pass to the limit in \eqref{theta.eta.subsol} and find that $(t_{\theta}, m_{\theta})$ satisfies
\begin{equation} \label{theta.subsol}
\begin{split}
 &  - \partial_t \Phi_{\theta}\big(t_{\theta}, m_{\theta} \big) + \int_{\R^d} Dm_{\theta} \cdot \Big(  \big( Df_{\theta}(x) \big)^T D_x D_{H^{-1}} \Phi\Big(t_{\theta},(f_{\theta})_{\#} m_{\theta}, f_{\theta}(x) \Big)  \Big) dx\\
    &\qquad + \delta \int_{\R^d} |D m_{\theta}|^2 dx  + \int_{\Omega} H\Bigl( 
\big( Df_{\theta}(x) \big)^T D_x D_{H^{-1}} \Phi\Big((f_{\theta})_{\#} m_{\theta}, f_{\theta}(\cdot) \Big)(x) \Big)  m_{\theta}(x) dx  \\
&\qquad \qquad \leq C \delta \int_{\R^d} \big| m_{\theta} \big|^2 dx.
\end{split}
\end{equation}
%
\textit{Step 4 - Sending $\theta \to 0$.} All limits in this step are as $\theta\to 0$, a fact that we will not be repeating. Again, we start by noting that  the $m_{\theta}$'s  are  bounded in $L^2$ and, hence,  by \eqref{theta.subsol}  also in $H_0^1$. Since Step 1 shows that $m_{\theta} \to m_0$ weakly in $L^2$, we in fact have that
\begin{align} \label{mtheta.conv}
t_{\theta} \to t_0 \ \text{and} \   m_{\theta} \to m_0 \ \text{weakly in $H_0^1(\R^d)$ \ and \  strongly in $L^2$}.  
\end{align}
We now claim that we also have that
\begin{align} \label{fthetal2}
    (f_{\theta})_{\#} m_{\theta} \to m_0 \ \text{weakly in $L^2(\R^d)$}.
\end{align}
Indeed, using  the fact that the density of $(f_{\theta})_{\#} m_{\theta}$ is $m_{\theta}(f_{\theta}^{-1}(x)) \cdot |J(f_{\theta}^{-1})|$, with $J(f_{\theta}^{-1})$ being the Jacobian of the map $f_{\theta}^{-1}$, we find that 
\begin{align*}
  &\| (f_{\theta})_{\#} m_{\theta} \|_2^2  dx = \int_{\R^d} |m_{\theta}\big(f_{\theta}^{-1}(x)\big)|^2 |J(f_{\theta}^{-1})(x)|^2 dx 
  \\
  &\quad = \int_{\R^d} |m_{\theta}(y)|^2 |J(f_{\theta}^{-1})(f_{\theta}(y))| dx \leq \| J(f_{\theta}^{-1})\|_{\infty} \|m_{\theta}\|_2^2. 
\end{align*}
Now since $Df_{\theta} \to I_{d \times d}$ uniformly, we see that $J(f_{\theta}^{-1})$ is bounded uniformly in $\theta$, so that the $(f_{\theta})_{\#} m_{\theta}$'s are  uniformly bounded in $L^2$. Since  $m_{\theta} \to m_0$ weakly in $H_0^1(\R^d)$ and $f_{\theta} \to \text{Id}$ uniformly, it is immediate  that in fact every limit point of the $(f_{\theta})_{\#} m_{\theta}$'s with respect to the weak topology on $L^2(\R^d)$ must be equal to $m_0$, which proves \eqref{fthetal2}. 

In light of \eqref{mtheta.conv}, to pass in  the limit in \eqref{theta.subsol}, it suffices to show that
\begin{align} \label{sufficient.strong}
    \big(Df_{\theta}(x) \big)^T 
    D_x D_{H^{-1}} \Phi\Big((f_{\theta})_{\#} m_{\theta}, f_{\theta}\Big) 
     \to  D_x D_{H^{-1}} \Phi\Big(m_0, \cdot \Big) \ \text{strongly in} \  L^2.
\end{align}
Using the triangular inequality we find
\begin{equation}\label{threeterms}
\begin{split} 
    &\norm{\big(Df_{\theta}(x) \big)^T D_x D_{H^{-1}} \Phi\Big((f_{\theta})_{\#} m_{\theta}, f_{\theta}(\cdot) \Big) - D_x D_{H^{-1}} \Phi\Big(m, \cdot \Big)}_{L^2} \\
   &\quad \leq \norm{ \Big(\big(Df_{\theta}(x)  \big)^T - I_{d \times d} \Big) D_x D_{H^{-1}} \Phi\Big((f_{\theta})_{\#} m_{\theta}, f_{\theta}(\cdot) \Big) }_{L^2} \\
  &\qquad + \norm{D_x D_{H^{-1}} \Phi\Big((f_{\theta})_{\#} m_{\theta}, f_{\theta}(\cdot) \Big) - D_x D_{H^{-1}} \Phi\Big(m, f_{\theta}(\cdot) \Big) }_{L^2}\\
&\qquad  + \norm{D_x D_{H^{-1}} \Phi\Big(m, f_{\theta}(\cdot) \Big) - D_x D_{H^{-1}} \Phi\Big(m, \cdot \Big)}_{L^2}.
\end{split}
\end{equation}

The first term  in \eqref{threeterms} converges to $0$ since $Df_{\theta} \to I_{d \times d}$ in $L^{\infty}$. For the second term, we use the fact that, since $(f_{\theta})_{\#}m_{\theta} \to m_0$ weakly in $L^2$ and, hence,  in $H^{-1}$, we have 
\begin{align*}
    D_x D_{H^{-1}} \Phi\big( (f_{\theta})_{\#} m_{\theta}, \cdot \big) \to D_x D_{H^{-1}} \Phi\big(  m_{\theta}, \cdot \big) \ \text{strongly in $L^2$}, 
\end{align*}
and then note that 
\begin{align*}
    \int_{\R^d}& \Big| D_x D_{H^{-1}} \Phi\Big((f_{\theta})_{\#} m_{\theta}, f_{\theta}(x) \Big) - D_x D_{H^{-1}} \Phi\Big(m, f_{\theta}(x) \Big) \Big|^2 dx 
    \\
    &= \int_{\R^d} \Big| D_x D_{H^{-1}} \Phi\Big((f_{\theta})_{\#} m_{\theta}, y \Big) - D_x D_{H^{-1}} \Phi\Big(m, y \Big) \Big|^2 |J(f_{\theta}^{-1}(y)|  dx 
    \\
    &\leq \|J(f^{-1})\|_{\infty} \norm{ D_x D_{H^{-1}} \Phi\Big((f_{\theta})_{\#} m_{\theta}, y \Big) - D_x D_{H^{-1}} \Phi\Big(m, y \Big)}_2^2, 
\end{align*}
which shows that the second term in \eqref{threeterms} tends to zero. 
 
For the third term, we need to show that, for fixed $\phi \in L^2(\R^d)$, we have $\phi \circ f_{\theta} \to f_{\theta}$ in $L^2$. For this, we fix $\eps > 0$, and choose $\phi_{\eps} \in C_c^{\infty}(\R^d)$ such that $\|\phi - \phi_{\eps}\|_2^2 < \eps$. Arguing as above, we have $\|\phi \circ f_{\theta} - \phi_{\eps} \circ f_{\theta}\|_2^2 \leq \|J(f_{\theta})^{-1}\|_{\infty} \eps$,  while  it is straightforward to check that $\phi_{\eps} \circ f_{\theta}  \to \phi_{\eps}$ in $L^2$. It follows that 
\begin{align*}
    \limsup_{\theta \to 0} \| \phi \circ f_{\theta} - \phi\|_2^2 \leq  C \eps +  \limsup_{\theta \to 0} \|\phi - \phi_{\eps}\|_2^2  \leq C \eps, 
\end{align*}
which shows that $\phi \circ f_{\theta} \to \phi$ in $L^2$, and, thus, that   the third term in \eqref{threeterms} converges to zero.

Since  \eqref{sufficient.strong} holds, it  enough to pass to the limit in \eqref{theta.subsol} to obtain
\begin{equation}
\begin{split} \label{subsol}
    &- \partial_t \Phi\big(t_{0}, m_{0} \big) + \int_{\R^d} Dm_{0} \cdot  D_x D_{H^{-1}} \Phi\big(t_{0},m_0, x  \big) dx\\
    &\qquad + \delta \int_{\R^d} |D m_{0}|^2 dx  + \int_{\Omega} H^R \Bigl( x, D_x D_{H^{-1}} \Phi\big(m_{\theta}, x \big) \Big)  m(dx)  \leq C \delta \int_{\R^d} \big| m_{\theta} \big|^2 dx, 
\end{split}
\end{equation}
which completes the proof.  \end{proof}

We end the section with another technical fact that is need for the proof of Theorem~ \ref{thm.comparison}.

\begin{lemma} \label{lem.linearderivcomp}
    Fix $n_0 \in \sub \cap H^{-1}$, and let $\Phi(m) = \frac{1}{2} \|m - n_0\|_{-1}^2$. Then $\Phi \in C^{1,2}(H^{-1})$, and, for $m_0 \in \sub \cap H^{-1}$,
    \begin{align*}
        D_{H^{-1}} \Phi(m_0,x) = f(x), 
    \end{align*}
    where $f = (m_0 - n_0)^* \in H_0^1$ is the dual element of $m_0 - n_0$, and  the unique solution in $H_0^1(\Omega)$ of the PDE 
    \begin{align} \label{dual.pde}
        f - \Delta f = (m_0 - n_0) \text{ in } \Omega, \quad f|_{\partial \Omega} = 0.
    \end{align}
\end{lemma}

\begin{proof}
   The fact that $D_{H^{-1}} \Phi(m_0,\cdot) = (m_0 - n_0)^*$ is standard. Notice that by integration by parts, the solution $f$ to \eqref{dual.pde} satisfies 
   \begin{align*}
       \langle f, g \rangle_{H_0^1} = \int_{\Omega} g d(m_0 - n_0) = \langle g, m_0 - n_0 \rangle_{1,-1}, 
   \end{align*}
   which by definition means that $f = (m_0 - n_0)^*$.  
   \end{proof}

\section{The comparison principle.} \label{sec.comparison}

We present here the proof of the comparison result.

\begin{proof}[Proof of Theorem \ref{thm.comparison}.]
    We suppose towards a contradiction that
    \begin{align*}
       \sup_{t \in [0,T], \, m \in \sub} \Big\{ V^-(t,m) - V^+(t,m) \Big\} > 0.
    \end{align*}
    Then, we can choose $\lambda > 0$ small enough that 
    \begin{align*}
        M_0 = \sup_{t \in [0,T], \, m \in \sub} \Big\{ V^-(t,m) - V^+(t,m) - \lambda(T-t) \Big\} > 0.
    \end{align*}
    Now, for $\delta > 0$, we consider the optimization problem
    \begin{align} \label{deltaproblem}
       M_{\delta} = \sup_{t \in [0,T], \, m \in \sub} \Big\{ V^-(t,m) - V^+(t,m) - \lambda(T-t) - 2\delta \|m\|_2^2 \Big\}.
    \end{align}
     We claim that $M_{\delta} \to M_0$. Indeed, clearly we have $M_{\delta} \leq M_0$ for each $\delta > 0$. Fix an optimizer $(t_0,m_0)$ in problem defining $M_0$. Then, define $m_{\kappa} = \big( m_0 * \rho_{\kappa} \big)|_{\Omega}$. It is straightforward to check that $m_{\kappa} \to m$ (with respect to $\bd$) as $\kappa \to 0$, and thus for any $\eta > 0$, we can choose $\kappa$ small enough that
    \begin{align*}
        M_0 \leq V^-(t_0,m_{\kappa}) - V^+(t_0,m_{\kappa}) - \lambda(T - t_0) + \eta, 
    \end{align*}
    and hence 
    \begin{align*}
        \limsup_{\delta \to 0} M_{\delta} &\geq \limsup_{\delta \to 0} \Big( V^-(t_0,m_{\kappa}) - V^+(t_0,m_{\kappa}) - \lambda(T - t_0) - 2 \delta \|m_{\kappa}\|_2^2 \Big) 
        \\
        &= V^-(t_0,m_{\kappa}) - V^+(t_0,m_{\kappa}) - \lambda(T - t_0) \geq M_0 - \eta.
    \end{align*}
    Thus we indeed have $M_{\delta} \to M_0$,  and so we can choose $\delta_0$ small enough so that 
    \begin{align} \label{mdelta.bound}
        M_{\delta} \geq M_0/2 > 0 \text{ for all } 0 < \delta < \delta_0.
    \end{align}
    We now double variables, and introduce the optimization problem
    \begin{align} \label{epsdeltaproblem}
     M_{\delta, \eps} = \sup_{t,s \in [0,T], \, m,n \in \sub \cap L^2(\Omega)} &\Big\{ V^-(t,m) - V^+(s,n) - \frac{1}{\eps} \Big( |t-s|^2 + \|m-n\|_{H^{-1}}^2 \Big) 
      \nonumber  \\
       &\qquad - \delta \|m\|_2^2 - \delta \|n\|_2^2 - \lambda(T-t) \Big\}.
    \end{align}
    To keep the argument clear, we divide the proof into separate steps. 
    \newline \newline 
    \textit{Step 1 - convergence of the optimizers.} The optimization problems   \eqref{deltaproblem} and \eqref{epsdeltaproblem} admit at least one optimizer $(t_{\delta}, m_{\delta})$ and $(t_{\delta, \eps}, s_{\delta, \eps}, m_{\delta, \eps}, n_{\delta, \eps})$ respectively.  This follows from the facts that any optimizing sequence is bounded in $L^2$ and the maps 
        \[\sub \cap L^2 \ni (t,m) \mapsto V^-(t,m) - V^+(t,m)  - 2\delta \|m\|_2^2  \]
  and
     \[\sub \cap L^2 \times \sub \cap L^2 \mapsto V^-(t,m) - V^+(t,m) - \frac{1}{\eps}  \|m-n\|_{H^{-1}}^2  - \delta \|m\|_2^2 - \delta \|n\|_2^2\]
 are upper semi-continuous with respect to $\bd$ and hence with respect to the weak topology on $L^2$.
    
    Since  
$        V^-(T,\cdot) \leq G \leq V^+(T,\cdot),$
    and, by assumption,  $M_{\delta} > 0$ for $\delta < \delta_0$, we deduce that, for each $\delta < \delta_0$, any optimizer $(t_{\delta},m_{\delta})$ of  \eqref{deltaproblem} must satisfy 
    \begin{align}
        t_{\delta} < T. 
    \end{align}
%
      Following standard arguments, one can check that if $(t_{\eps,\delta},s_{\delta,\eps}, m_{\delta, \eps}, n_{\delta,\eps})$ is any optimizer of \eqref{epsdeltaproblem}, then $m_{\delta, \eps}$ and $n_{\delta, \eps}$ are bounded in $L^2$ uniformly in $\eps$, and if $(t_{\delta}, s_{\delta}, m_{\delta}, n_{\delta})$ is any weak-$L^2$ limit point as $\eps \to 0$, then $t_{\delta} = s_{\delta}$, $m_{\delta} = n_{\delta}$, and $(t_{\delta}, m_{\delta})$ is an optimizer for \eqref{deltaproblem}. 
%
    

  %
       If $\delta < \delta_0$, we must therefore have $t_{\delta} < T$, and so we conclude that for each each $\delta < \delta_0$, there exists $\eps_0 = \eps_0(\delta) > 0$ such that for $\eps < \eps_0$ and for any optimizer $(t_{\delta,\eps}, s_{\delta, \eps}, m_{\delta, \eps}, n_{\delta, \eps})$ for \eqref{epsdeltaproblem}, we have
    \begin{align} \label{noterminaltime}
        t_{\delta, \eps} < T, \quad s_{\delta, \eps} < T.
    \end{align}
    \textit{Step 2 - using the equation.} We now fix $\delta < \delta_0$ and $\eps<\eps_0$ so that any optimizer $(t_{\delta,\eps}, s_{\delta, \eps}, m_{\delta, \eps}, n_{\delta, \eps})$ for \eqref{epsdeltaproblem} satisfies \eqref{noterminaltime}. 
    
    We apply Lemma \ref{lem.linearderivcomp}
    and the definition of viscosity subsolution to deduce that $m_{\delta, \eps}, n_{\delta, \eps} \in H_0^1$, and there is a constant $C$ which is independent of $\eps,\delta$ such that
    \begin{align*}
        &\lambda - \frac{(t_{\delta,\eps} - s_{\delta,\eps})}{\eps} + \delta \int_{\Omega} |Dm_{\delta, \eps}|^2 dx + \int_{\Omega} D_x f_{\delta,\eps} \cdot Dm_{\delta,\eps} dx  + \int_{\Omega} H \big( x, D_x f_{\delta,\eps} \big) m_{\delta,\eps}(x) dx 
        \\
        &\qquad \qquad \leq F(m_{\delta,\eps}) + C \delta \|m_{\delta,\eps}\|_2^2, 
    \end{align*}
    and, likewise,
    \begin{align*}
        &- \frac{(t_{\delta,\eps} - s_{\delta,\eps})}{\eps} - \delta \int_{\Omega} |Dn_{\delta, \eps}|^2 dx + \int_{\Omega}  D_x f_{\delta,\eps} \cdot Dm_{\delta,\eps} dx + \int_{\Omega} H \big( x, D_x f_{\delta,\eps} \big) n_{\delta,\eps}(x) dx 
        \\
        &\qquad \qquad \geq F(m_{\delta,\eps}) - C \delta \|n_{\delta,\eps}\|_2^2, 
    \end{align*} 
    where $f_{\delta,\eps} \in H_0^1$ is the unique solution of 
    \begin{align*}
        f_{\delta,\eps} - \Delta f_{\delta,\eps} = \frac{m_{\delta,\eps} - n_{\delta,\eps}}{\eps} \ \text{in} \ \Omega \ \text{and} \  f_{\delta,\eps}=0 \ \text{on} \ \partial \Omega. 
    \end{align*}
    Subtracting these two inequalities, we deduce that 
    \begin{align*}
        \lambda + \delta \int_{\Omega} \big(|Dm_{\delta, \eps}|^2 + |Dn_{\delta, \eps}|^2 \big) dx \leq I + II + III + IV, 
    \end{align*}
    where 
    \begin{align*}
        &I = -\int_{\Omega} D_x f_{\delta,\eps} \cdot D \big(m_{\delta,\eps} - n_{\eps, \delta} \big) dx, \qquad II = -\int_{\Omega} H\Big( x,D_x f_{\eps, \delta} \Big)(m_{\delta,\eps} - n_{\delta,\eps})(x) dx, 
        \\
        &III = F(m_{\eps, \delta}) - F(n_{\delta,\eps}), \qquad  \,\, \qquad  \qquad IV = C\delta \big(\|m_{\delta,\eps}\|_2^2 + \|n_{\delta,\eps}\|_2^2\big).
    \end{align*}
    Next, notice that 
    \begin{align*}
        I = \int_{\Omega} \Delta_x & f_{\delta,\eps} (m_{\delta,\eps} - n_{\delta,\eps}) dx = \int_{\Omega} f_{\delta,\eps} (m_{\delta,\eps} - n_{\delta,\eps}) dx - \frac{1}{\eps} \|m_{\delta,\eps}-n_{\delta,\eps}\|_2^2
        \\
        &\leq \|f_{\delta,\eps}\|_{H_0^1} \|m_{\delta,\eps} - n_{\delta,\eps}\|_{H^{-1}} - \frac{1}{\eps} \|m_{\delta,\eps} - n_{\delta,\eps}\|_2^2
        \\
        &\leq \frac{C}{\eps} \norm{m_{\delta,\eps}-  n_{\delta,\eps}}_{H^{-1}}^2- \frac{1}{\eps} \|m_{\delta,\eps} - n_{\delta,\eps}\|_2^2,
    \end{align*}
    while, by the linear growth of $H$,
    \begin{align*}
        II &\leq \int_{\Omega} C\big(1 + |D_x f_{\delta,\eps}| \big) |m_{\delta,\eps} - n_{\delta,\eps} | dx \leq C\big(1 + \|f_{\delta,\eps}\|_{H_0^1} \big) \|m_{\delta,\eps} - n_{\delta,\eps}\|_2
        \\
        &\leq \frac{1}{2\eps} \|m_{\delta,\eps} - n_{\delta,\eps}\|_2^2 + C \eps \Big(1 + \|f_{\delta,\eps}\|_{H_0^1}^2 \Big) \leq \frac{1}{2\eps} \|m_{\delta,\eps} - n_{\delta,\eps}\|_2^2 + C \Big(\eps + \frac{1}{\eps} \|m_{\delta,\eps} - n_{\eps, \delta}\|_{H^{-1}}^2\Big).
    \end{align*}
    Since $III \leq C d(m_{\delta,\eps}, n_{\delta,\eps}) \leq C \|m_{\delta,\eps} - n_{\delta,\eps}\|_{H^{-1}}$, combining the bounds above we get 
    \begin{equation}\label{lambdabound}
    \begin{split} 
        \lambda +& \delta \int_{\Omega} \big(|Dm_{\delta, \eps}|^2 + |Dn_{\delta, \eps}|^2 \big) dx \leq  C \bigg( \frac{1}{\eps} \norm{m_{\delta,\eps}- n_{\delta,\eps}}_{H^{-1}}^2 
       +  \eps + \|m_{\eps, \delta} 
        - n_{\delta,\eps}\|_{H^{-1}}
        + \delta \big(\|m_{\delta,\eps}\|_2^2 + \|n_{\delta,\eps}\|_2^2  \big) \bigg) 
    \end{split}
    \end{equation}
    Notice that, for fixed $\delta$, the $m_{\delta,\eps}$'s and $n_{\delta, \eps}$'s are bounded in $L^2$ independently of $\eps$, and so, in view of  \eqref{lambdabound}, they are in fact bounded in $H_0^1$ independently of $\eps$. Arguing as in Step 1, we can thus find a sequence $\eps_k \to 0$ such that, as $k\to \infty$, 
    \begin{align*}
       t_{\delta, \eps_k}, s_{\delta, \eps_k} \to t_{\delta}, \quad m_{\delta_k, \eps_k}, n_{\delta, \eps_k} \to m_{\delta}
     \ \text{weakly in $H_0^1$ \ and \ strongly in $L^2$},
    \end{align*}
where $(t_{\delta}, m_{\delta})$ is an optimizer for the problem \eqref{deltaproblem}.

    Note that one can easily check that \begin{align} \label{epstozero}
   M_{\delta, \eps} \to  M_{\delta}  \ \text{and}\  \frac{1}{\eps} \Big(|t_{\delta, \eps} - s_{\delta, \eps}|^2 + \|m_{\delta, \eps}-n_{\delta, \eps}\|_{H^{-1}}^2 \Big) \to 0.
   \end{align}
    Combined with \eqref{lambdabound}, this shows that, for each $\delta > 0$, 
    \begin{align*}
         \lambda \leq C \delta \|m_{\delta}\|_2^2. 
    \end{align*}
    But it is easy to check that, as $\delta \to 0$,  $M_{\delta} \to  M_0$, and, hence,  $\delta \|m_{\delta}\|_2^2 \to 0$. We thus find that $\lambda \leq 0$, which is a contradiction. This completes the proof.  
    \end{proof}

\section{The properties of $V^{N,K}$.} \label{sec.vnk}

We prove  in this section the Lipschitz continuity of the $V^{N,K}$'s and then study the limit as $N \to \infty$. To make the arguments more readable we split the 
discussion into two subsections.

\subsection{Uniform in $N$ Lipschitz bounds.} \label{subsec.vnkLip}

To extract a ``limit point" of the sequence  $(V^{N,K})_{K= 1,...,N}$,  we  view them  
as maps on $[0,T] \times \sub^N$, where 
\begin{align} \label{def.pnsub}
    \sub^N \coloneqq \big\{ m_{\bx}^{N,K} : K =1,...,N, \, \bx \in \Omega^K \big\}.
\end{align}
The  goal is  to show that the functions
\begin{align} \label{def.tildev}
  \wt{V}^{N,K} : [0,T] \times \sub^N \to \R, \quad \wt{V}^{N,K}(t,m_{\bx}^{N,K}) =   V^{N,K}(t,\bx),
\end{align}
which are well-defined in view of the symmetry of $(V^{N,K})_{K = 1,...,N}$,
satisfy appropriate estimates, uniformly in $N$. This is the subject of the following Theorem.

\begin{theorem} \label{thm.vnk.lip}
There is a constant $C$ such that, for each $N \in \N$, $t,s \in [0,T]$, $K,M \in \{1,...,N\}$, and $\bx \in \Omega^K$, $\by \in \Omega^M$, 
\begin{align*}
   |V^{N,K}(t,\bx) - V^{N,M}(s,\by)| \leq C \Big(|t-s|^{1/2} + \bd\big( m_{\bx}^{N,K},m_{\by}^{N,M} \big) \Big).
\end{align*}
\end{theorem}

Before presenting  the proof of Theorem~\ref{thm.vnk.lip}, we need a number of preliminary facts and estimates which we formulate as separate  lemmata.

\begin{lemma} \label{lem.crudebound}
    There is a constant $C$ which is independent of $N$ such that, for each $N$, $K=1,\ldots, N$  and  $\bx \in \Omega^K$, 
    \begin{align*}
       | V^{N,K}(t,\bx) - V^{N,K-1}(t,\bx^{-i}) | \leq \frac{C}{N}, \quad \text{where $\bx^{-i} = (x^1,...,x^{i-1},x^{i+1},...,x^K) \in \Omega^{K-1}$.}
    \end{align*}
\end{lemma}

\begin{proof}
For $K = 1,...,N$, and $i \in \{1,...,K\}$, we introduce the function  $\hat{V}^{N,K-1}: [0,T] \times \Omega^K \to \R$ given by $\hat{V}^{N,K-1}(t,\bx) = V^{N,K-1}(t,\bx^{-i})$,
which solves the PDE 
\begin{align*}
    \begin{cases}
  \ds       - \partial_t \hat{V}^{N,K-1} - \sum_{i = 1}^K \Delta_{x^i} \hat{V}^{N,K-1} + \frac{1}{N} \sum_{i 
= 1}^K H(x^i, N D_{x^i} \hat{V}^{N,K-1}) 
\\ \ds 
\qquad \qquad\qquad= F(m_{\bx^{-i}}^{N,K-1}) + \frac{1}{N} H(x^i, 0) \ \text{in} \ [0,T] \times (\T^d)^K,   \vspace{.2cm}
\\
\ds \hat{V}^{N,K-1}(T,\bx) = G(m_{\bx^{-i}}^{N,K-1}) \ \text{and} \  \hat{V}^{N,K-1} = V^{N,K-1} \ \text{  on } \ \partial(\Omega^K).
    \end{cases}
\end{align*}
Note that, if $K = 1$ in the formula above, we have  $\hat{V}^{N,0}(t,x) = V^{N,0}(t) = G(\bzero) + (T-t) F(\bzero)$. Since $H(x^i,0)$ is bounded and
\begin{align*}
    |G(m_{\bx^{-i}}^{N,K-1}) - G(m_{\bx}^{N,K})| \leq C \bd\big(m_{\bx^{-i}}^{N,K-1}, m_{\bx}^{N,K} \big) \leq C/N,  
\end{align*}
we conclude by the comparison principle.  \end{proof}

We next need to address the Lipschitz regularity of $V^{N,K}$ near the lateral boundary of its domain.  We recall from Section \ref{sec.prelim} that $d_{\partial \Omega}$ is the distance to the boundary, $\cN_{\eps}$ is an open collar around $\partial \Omega$,  that we can choose $\eps_0 > 0$ small enough so that, for each $0 < \eps < \eps_0$, $d_{\partial \Omega}$ is smooth on $\cN_{\eps}$ and satisfies $|D d_{\Omega}| = 1$. 
It is also clear that for $\eps$ small enough, $\cN_{\eps}$ is an open set with a $C^2$ boundary of the form
\begin{align*}
    \partial \cN_{\eps} = \partial \Omega \sqcup \partial^+ \cN_{\eps}, \quad \partial^+ \cN_{\eps} = \{ x \in \Omega : d_{\partial \Omega}(x) = \eps \}. 
\end{align*}

\begin{lemma} \label{lem.barrier}
    For any constants $C> 0$, there exists   $\eps \in (0,\eps_0)$ and smooth functions $\phi^+$, $\phi^- : \ov{\cN_{\eps}} \to \R$ such that 
\begin{equation}  \begin{split}
& \phi^+ = \phi^- = 0 \ \text{ on } \ \partial \Omega, \\[1.2mm]
& \Delta \phi^+ \leq -C \big(1 + |D \phi^+|^2 \big) \ \  \text{and} \ \ \Delta \phi^- \geq C \big(1 + |D \phi^-|^2\big) \ \text{ on} \ \cN_{\eps},\\[1.2mm] 
& \phi^+ \geq C \ \text{and} \ \phi^- \leq -C \ \text{on} \  \partial^+ \cN_{\eps},\\[1.2mm]
& \phi^+ \geq C d_{\partial \Omega},  \ \phi^- \leq - C d_{\partial \Omega} \ \text{on}  \ \cN_{\eps}. 
 \end{split}
\end{equation}

\end{lemma}
\begin{proof}
    Fix  $\eps \in (0,\eps_0)$, a smooth function $\psi : [0,\eps] \to \R$ , and let 
    \begin{align*}
        \phi^+(x) = \psi(d_{\partial \Omega}(x)).
    \end{align*}
    Then, 
    \begin{align*}
        D \phi^+ = \psi'(d_{\partial \Omega}(x)) Dd_{\partial \Omega}(x), 
        \quad \Delta \phi^+(x) = \phi^{''}(d_{\partial \Omega}(x)) |Dd_{\partial \Omega}(x)|^2 + \psi'(x) \Delta d_{\partial \Omega}(x).
    \end{align*}
   Recalling  that, in  $\cN_{\eps}$, $|Dd_{\partial \Omega}| = 1$ and that $\Delta d_{\partial \Omega}$ is bounded, we find, for some $C'$ depending on the bound on $\Delta d$, 
    \begin{align*}
        \Delta \phi^+(x) + C \big(1 + |D \phi^+|^2\big) \leq \psi^{''}(d_{\partial \Omega}(x)) + C' \big(1 + |\psi'(d_{\partial \Omega}(x))|^2).
    \end{align*}
    Thus, to have $\Delta \phi^+ \leq - C \big( 1 + |D \phi^+|^2 \big)$, it suffices to choose $\psi$ satisfying 
    \begin{align*}
        \psi^{''} = - C'\big(1 + |\psi'|^2 \big).
    \end{align*}
    The solutions to this ODE, with the initial condition $\psi(0) = 0$, are given by the  1-parameter family 
    \begin{align*}
        \psi(x) = \int_0^x \tan\big( - C' y + \arctan(s) \big) dy \ \text{for} \ s \in \R. 
    \end{align*}
It turns out that, if we choose $\eps$ small enough and  $s$ large enough, then this $\psi$ will satisfy $\psi'(x) > C$ on $[0,\eps]$, $\psi(\eps) > C$, so $\phi^+$ will have the desired properties. The construction for $\phi^-$ is similar.  \end{proof}

We note that under Assumption \ref{assump.main}, we can choose $C$ large enough so  that 
\begin{align} \label{C1}
    \Big|\frac{1}{N} H(x^i,Np^i) \Big| + |F(m_{\bx}^{N,K}) - F(m_{\bx^{-i}}^{N,K-1})| \leq \frac{C}{N} \big(1 + |Np^i|^2\big), 
\end{align}
\begin{align} \label{C2}
    |V^{N,K}(t,\bx) - V^{N,K-1}(t,\bx^{-i})| \leq C/N, \text{ and } 
\end{align}
\begin{align} \label{C3}
    |G(m_{\bx}^{N,K}) - G(m_{\bx^{-i}}^{N,K-1})| \leq \frac{C}{N} d(x^i,\partial \Omega). 
\end{align}

\begin{lemma} \label{lem.induction.gen}
    Let $C$  be a  positive constant such that \eqref{C1}, \eqref{C2} and  \eqref{C3} hold. If $\eps$, $\phi^+$ and  $\phi^-$ are  as in the statement of Lemma~\ref{lem.barrier},  then, for each $N$, $K\in\{1,\ldots,N\}$ and $\bx \in \Omega^K$ with $x^i \in \cN_{\eps}$,
    \begin{align*}
      V^{N,K-1}(t,\bx^{-i}) + \frac{1}{N} \phi^-(x^i) \leq  V^{N,K}(t,\bx) \leq V^{N,K-1}(t,\bx^{-i}) + \frac{1}{N} \phi^+(x^i).
    \end{align*}
\end{lemma}

\begin{proof}
    Since the proof  of the two inequalities are similar, we  only prove the second bound.

    Recall that when $K = 1$,  $V^{N,K-1}=V^{N,0}=G(\bzero) + (T-t) F(\bzero)$, and the first step of the induction is to prove that
    \begin{align*}
      V^{N,1}(t,x^1) \leq G(\bzero) + (T-t) F(\bzero) + \frac{1}{N} \phi^+(x^1).
    \end{align*}
    Notice that by the choice of $C$, the map $(t,x^1) \mapsto G(\bzero) + \frac{1}{N} \phi^+(x^1)$ is a supersolution of the equation for $V^{N,1}$ on the domain $[0,T] \times \cN_{\eps}$, thus  it suffices to show the bound on the terminal and lateral boundaries. At time $T$, we have 
    \begin{align*}
        V^{N,1}(T,x^1) = G(\frac{1}{N} \delta_{x^1}) \leq G(\bzero) + C \bd(\frac{1}{N} \delta_{x^1}, \bzero) \leq G(\bzero) + \frac{1}{N} \phi^+(x^1). 
    \end{align*}
    For $x^1 \in \partial \Omega$, clearly 
    \begin{align*}
        V^{N,1}(t,x^1) = G(\bzero) + (T-t) F(\bzero) = G(\bzero) + (T-t) F(\bzero) + \frac{1}{N} \phi^+(x^1).
    \end{align*}
    Finally, for $x^1$ on the inner boundary of $\cN$, we use Lemma \ref{lem.crudebound} to get, using that $\phi^+ \geq C$ on the inner boundary
    \begin{align*}
        V^{N,1}(t,x^1) \leq G(\bzero) + (T-t) F(\bzero) + \frac{C}{N} \leq G(\bzero) + (T-t) F(\bzero) + \frac{1}{N} \phi^+(x^1).
    \end{align*}
        
    Now suppose that, for some fixed $N$ and $K \in \{1,...,N\}$, the bound holds for $V^{N,K-1}$. Since  $V^{N,K-1}(t,\bx^{-i}) + \frac{1}{N} \phi^+(x^i)$ is 
    is a supersolution to the equation satisfied by $V^{N,K}$ on the domain $[0,T] \times \big(\Omega^K \cap \{x^i \in \cN_{\eps}\} \big)$,  we need to show that the bound holds on the parabolic boundary of this domain. At time $T$, we have 
    \begin{align*}
        V^{N,K}(T,\bx) &= G(m_{\bx}^{N,K}) \leq G(m_{\bx^{-i}}^{N,K-1}) + C \bd(m_{\bx}^{N,K}, m_{\bx^{-i}}^{N,K-1}) 
        \\
        & \leq G(m_{\bx^{-i}}^{N,K-1}) + \frac{C}{N} d(x^i,\partial \Omega) \leq G(m_{\bx^{-i}}^{N,K-1}) + \frac{1}{N} \phi^+(x^i).
    \end{align*}
    For $x^i \in \partial \Omega$, the two functions agree and,  for $x^i\in \partial^+ \cN_{\eps}$, we again use Lemma \ref{lem.crudebound} to get 
    \begin{align*}
        V^{N,K}(t,\bx) \leq V^{N,K-1}(t,\bx^{-i}) + \frac{C}{N} \leq V^{N,K-1}(t,\bx^{-i}) + \frac{1}{N} \phi^+(x^i). 
    \end{align*}
    Finally, if  $x^j \in \partial \Omega$ for some $j \neq i$, we use the inductive hypothesis to conclude that 
    \begin{align*}
        V^{N,K}(t,\bx) = V^{N,K-1}(t,\bx^{-j}) \leq V^{N,K-2}(t,\bx^{-j,-i}) + \frac{1}{N} \phi^+(x^i) = V^{N,K-1}(t,\bx^{-i}) + \frac{1}{N} \phi^+(x^i).
    \end{align*}
   The proof is now complete.  
   \end{proof}

\begin{proposition} \label{prop.DiVNK}
    There is a positive constant $C$ such that, for each $N \in \N$,  $K \in \{1,...,N\}$, $t \in [0,T]$ and $\bx,\by \in \Omega^K$, we have 
    \begin{align*}
        |V^{N,K}(t,\bx) - V^{N,K}(t,\by)| \leq \frac{C}{N} \sum_{i  = 1}^K |x^i - y^i|.
    \end{align*}
\end{proposition}

\begin{proof}
Combining Lemma \ref{lem.crudebound} and Lemma \ref{lem.induction.gen}, we see that there exists a constant $C'$ such that, for each $N \in \N$, $K \in \{1,...,N\}$, $i \in \{1,...,K\}$, and $\bx \in \Omega^K$, we have 
   \begin{align} \label{removeparticle}
       |V^{N,K}(t,\bx) - V^{N,K-1}(t,\bx^{-i})| < \frac{C'}{N} d_{\partial \Omega}(x^i).
   \end{align}
   Moreover, from the Lipschitz continuity of $G$, we can choose $C'$ larger if necessary so that 
   \begin{align} \label{glipbound}
       |G(m_{\bx}^{N,K}) - G(m_{\by}^{N,K})| \leq \frac{C'}{N} \sum_{i = 1}^K |x^i - y^i|.
   \end{align}
   We now define 
   \begin{align*}
       \lambda(t) = C' \exp\big( \hat C(T-t) \big), 
   \end{align*}
   where $\hat C = C_F+C_H$, with $C_F$ and $C_H$ chosen large enough that
   \begin{align*}
       |F(m_{\bx}^{N,K}) - F(m_{\by}^{N,K})| \leq \frac{C_F}{N} \sum_{i = 1}^K |x^i - y^i|, \quad 
       |D_x H(x,p)| \leq C_H(1 + |p|).
   \end{align*}
   For $\eps > 0$, we consider the optimization problem 
   \begin{align} \label{meps}
      M_{\eps} =  \max_{K = 1,...,N} \sup_{t \in [0,T], \bx, \by \in \Omega^K} \Big\{ V^{N,K}(t,\bx) - V^{N,K}(t,\by) - \lambda(t) \frac{1}{N} \sum_{i = 1}^K \big(|x^i - y^i|^2 + \eps^2 \big)^{1/2} \Big\}.
   \end{align}
 and aim to show that  $M_{\eps} \leq 0$ for all $\eps > 0$, which will clearly imply the result. 

  In view of \eqref{glipbound}, it suffices to show that any optimizer $(K_0, t_0,\bx_0,\by_0)$ for \eqref{meps} satisfies $t_0 = T$.
 So, arguing by contradiction, we assume  that there is an optimizer $(K_0, t_0,\bx_0,\by_0)$ for \eqref{meps} with $t_0 < T$. We next want to rule out the possibility that $\bx_0 \in \partial(\Omega^K)$ or $\by_0 \in \partial(\Omega^K)$. Indeed, if 
    $\bx_0^j \in \partial \Omega$ for some $i = 1,...,K_0$, then,  by optimality, we have 
   \begin{align*}
       &V^{N,K_0}(t_0,\bx_0) - V^{N,K_0}(t_0,\by_0) - \lambda(t_0) \frac{1}{N} \sum_{j = 1}^{K_0} \big(|x^j_0 - y_0^j|^2 + \eps^2 \big)^{1/2}
       \\
       &\quad \geq V^{N,K_0-1}(t_0,\bx_0^{-i}) - V^{N,K_0-1}(t,\by_0^{-i}) - \lambda(t_0) \frac{1}{N} \sum_{\substack{j = 1,...,K_0 \\ j \neq i }} \big(|x_0^j - y_0^j|^2 + \eps^2 \big)^{1/2}.
   \end{align*}
   Rearranging the terms in the last inequality  and recalling that, for $x_0^i \in \partial \Omega$,  $V^{N,K_0-1}(t_0,\bx_0^{-i}) = V^{N,K_0}(t_0,\bx_0)$, we find  
   \begin{align*}
      C' |x_0^i - y_0^i| <  \lambda(t_0) \big(|x_0^i - y_0^i|^2 + \eps^2 \big)^{1/2} \leq V^{N,K_0-1}(t_0,\by_0^{-i}) - V^{N,K_0}(t_0,\by_0) \leq C' d_{\partial \Omega}(y_0^i) \leq C' |x_0^i - y_0^i|, 
   \end{align*}
 which  is a contradiction.  The same argument shows that $y_0^i \in \Omega$ for each $i$. 

   Since $\by_0$ and $\bx_0$ are in the interior of $\Omega^{K_0}$, $V^{N,K_0}$ is smooth in a neighborhood of $(t_0,\bx_0)$ and in a neighborhood of $(t_0,\by_0)$. Thus the optimality conditions for \eqref{meps} show that 
   \begin{align*}
    &\partial_t V^{N,K_0}(t_0,\bx_0) - \partial_t V^{N,K_0}(t_0,\by_0) = \lambda'(t_0) \frac{1}{N} \sum_{j = 1}^K \big(|x^j_0 - y_0^j|^2 + \eps^2 \big)^{1/2}, 
    \\
    &D_{x^i} V^{N,K_0}(t_0,\bx_0) = D_{x^i} V^{N,K_0}(t_0,\by_0) = \frac{\lambda(t_0)}{N} \frac{(x_0^i - y_0^i)}{(|x_0^i - y_0^i|^2 + \eps^2)^{1/2}},
    \\
    &D_{x^ix^i} V^{N,K_0}(t_0,\bx_0) \leq  D_{x^ix^i} V^{N,K_0}(t_0,\by_0). 
   \end{align*}
    Using the equation for $V^{N,K_0}$, we deduce that 
    \begin{align*}
        - \partial_t V^{N,K_0}(t_0,\bx_0) - \sum_{i = 1}^{K_0} \Delta_{x^i} V^{N,K_0}(t_0,\bx_0) + \frac{1}{N} \sum_{i = 1}^{K_0} H\Big( x_0^i,  \lambda(t_0) \frac{(x_0^i - y_0^i)}{(|x_0^i - y_0^i|^2 + \eps^2)^{1/2}} \Big) = F(m_{\bx_0}^{N,K_0}), 
    \end{align*}
    and 
    \begin{align*}
        - \partial_t V^{N,K_0}(t_0,\bx_0) - \sum_{i = 1}^{K_0} \Delta_{x^i} V^N(t_0,\bx_0) + \frac{1}{N} \sum_{i = 1}^{K_0} H\Big( y_0^i,  \lambda(t_0) \frac{(x_0^i - y_0^i)}{(|x_0^i - y_0^i|^2 + \eps^2)^{1/2}} \Big) = F(m_{\bx_0}^{N,K_0}), 
    \end{align*}
    so that 
    \begin{align*}
        \hat C \lambda(t_0)& \frac{1}{N} \sum_{j = 1}^{K_0} \big(|x^j_0 - y_0^j|^2 + \eps^2 \big)^{1/2} = - \lambda'(t_0) \frac{1}{N} \sum_{j = 1}^{K_0} \big(|x^j_0 - y_0^j|^2 + \eps^2 \big)^{1/2}
        \\
        &= - \partial_t V^{N,K_0}(t_0,\bx_0) + \partial_t V^{N,K_0}(t_0,\by_0) 
        \\
        &= \sum_{i = 1}^{K_0} \Big( \Delta_{x^i} V^{N,K_0}(t_0,\bx_0) - \Delta_{x^i} V^{N,K_0}(t_0,\by_0) \Big) + F(m_{\bx_0}^{N,K_0}) - F(m_{\by_0}^{N,K_0})
        \\
        &\quad - \frac{1}{N} \sum_{i = 1}^{K_0} \bigg( H\Big(x_0^i,  \lambda(t_0) \frac{(x_0^i - y_0^i)}{(|x_0^i - y_0^i|^2 + \eps^2)^{1/2}} \Big) - H\Big(x_0^i,  \lambda(t_0) \frac{(x_0^i - y_0^i)}{(|x_0^i - y_0^i|^2 + \eps^2)^{1/2}} \Big) \bigg)
        \\
        &\leq \frac{C_H \lambda(t_0)}{N} \sum_{i = 1}^{K_0} \frac{|x_0^i - y_0^i|^2}{(|x_0^i - y_0^i|^2 + \eps^2)^{1/2}} + \frac{C_F}{N} \sum_{i = 1}^{K_0} |x_0^i - y_0^i|
        \\
        &< \frac{(C_H + C_F)\lambda(t_0)}{N} \sum_{j = 1}^{K_0} \big(|x^j_0 - y_0^j|^2 + \eps^2 \big)^{1/2},
    \end{align*}
   again a contradiction in view of the choice of $\hat C$. It follows that $t_0 = T$ and, hence,   for each $\eps > 0$, 
    \begin{align*}
        M_{\eps} &= V^{N,K_0}(T,\bx_0) - V^{N,K_0}(T,\by_0) - \lambda(T) \frac{1}{N} \sum_{i = 1}^{K_0} \big( |x_0^i - y_0^i|^2 + \eps^2 \big)^{1/2} 
        \\
        &\leq G(m_{\bx_0}^{N,K_0}) -G(m_{\by_0}^{N,K_0}) - \frac{C'}{N} \sum_{i = 1}^{K_0} \big( |x_0^i - y_0^i|^2 + \eps^2 \big)^{1/2} \leq 0, 
    \end{align*}
    which completes the proof. 
    \end{proof}

 To get from Proposition \ref{prop.DiVNK} to Theorem \ref{thm.vnk.lip}, we are going to use the Monge-Kantorovich duality. The basic step is that we need to relate the metric $\bd$ on $\sub$ to an optimal transport problem.  
 
 For this 
 we introduce a function $\rho : \ov{\Omega} \times \ov{\Omega} \to \R$ given by 
 \begin{align*}
     \rho(x,y) = \sup_{\phi \in E} \phi(x) - \phi(y), 
 \end{align*}
 where the supremum is taken over the set $E$ of $1$-Lipschitz functions $\phi : \ov{\Omega} \to \R$ vanishing on $\partial \Omega$.
 \begin{lemma} \label{lem.rho}
     The function $\rho$ defines a  pseudo-metric on $\ov \Omega$, which restricts to a metric on $\Omega$.  Moreover, $\rho$ satisfies the estimate
    \begin{equation}\label{estidelta}
\dfrac12 \min\{|x-y|, d_{\partial\Omega}(x)+ d_{\partial\Omega}(y) \} \leq \rho(x,y) \leq \min\{|x-y|, d_{\partial\Omega}(x)+ d_{\partial\Omega}(y) \} . 
\end{equation}
 \end{lemma}
 \begin{proof}
     It is easy to check that $\rho$ is a pseudo metric on $\ov{\Omega}$, and that its restriction to $\Omega$ is a metric. For the bound \eqref{estidelta}, we first fix $x,y \in \Omega$, and consider the  functions $\phi,\psi \in E$ given by
     \begin{align*}
         \phi(a) = \min\{|a-y|, d_{\partial {\Omega}}(a)\}, \quad \psi(a) = \min\{|a-x|, d_{\partial \Omega}(a)\}.
     \end{align*}
     Then 
     \begin{align*}
       \text{min}\{|x-y|, d_{\partial \Omega}(x)\} =   \phi(x) - \phi(y) \leq \rho(x,y), \text{  and  }
           \text{min}\{|x-y|, d_{\partial \Omega}(y)\} =    \psi(y) - \psi(x) \leq \rho(x,y). 
     \end{align*}
     It follows that
     \begin{align*}
           \text{min}\{|x-y|, d_{\partial \Omega}(x) + d_{\partial \Omega}(y)\} \leq   \text{min}\{|x-y|, d_{\partial \Omega}(x)\} +   \text{min}\{|x-y|, d_{\partial \Omega}(y)\} \leq 2 \rho(x,y), 
     \end{align*}
     which gives the lower bound in \eqref{estidelta}. On the other hand, if $\phi\in E$, then
$\phi(x)-\phi(y)\leq |x-y|$ and, as $\phi$ vanishes on the boundary of $\Omega$, $|\phi(x)|\leq d_{\partial\Omega}(x)$ and  $|\phi(y)|\leq d_{\partial\Omega}(y)$. Thus 
$$
\phi(x)-\phi(y)\leq \min\{ |x-y| , d_{\partial\Omega}(x)+d_{\partial\Omega}(y) \}  , 
$$
which proves the upper bound in \eqref{estidelta}.  
\end{proof}

We note that we can also view $\rho$ as a metric on $\ov{\Omega}/ \sim$, where $\sim$ is the equivalence relation which identifies all boundary points. Let us denote by $\bd_{\rho}$ the Monge-Kantorovitch metric on $\mathcal P(\bar \Omega/ \sim)$ associated with the pseudometric $\rho$, that is, for $\mu,\nu \in \cP(\ov{\Omega}/ \sim)$, 
\begin{align*}
    \bd_{\rho}(\mu,\nu) = \inf_{\pi \in \Pi(\mu,\nu)} \int_{(\ov{\Omega}/ \sim) \times (\ov{\Omega}/ \sim)} \rho(x,y) \pi(x,y),
\end{align*}
with the infimum taken over the set $\Pi(\mu,\nu)$ of couplings of $\mu, \nu$,     that is,  the set of $\pi \in \cP\big( (\ov{\Omega}/ \sim) \times (\ov{\Omega}/ \sim) )$ with first marginal is $\mu$ and second marginal is $\nu$. 

We recall that by Monge-Kantorovich duality
\begin{align*}
    \bd_{\rho}(\mu,\nu) = \sup_{f} \Big( \int_{\ov{\Omega}/\sim} f d \mu - \int_{\ov{\Omega}/\sim} f d \nu \Big), 
\end{align*}
where the supremum is over all functions $f : (\ov{\Omega}/\sim) \to \R$ which are $1$-Lipschitz with respect to $\rho$. 

We note that we can view $\ov{\Omega} / \sim$ as $(\ov{\Omega} / \sim) = \Omega \cup \{[\partial \Omega]\}$ in an obvious way, and so a probability measure $\mu$ on $\ov{\Omega} / \sim$ has an obvious restriction $\mu|_{\Omega}$, which is a sub-probability measure on $\Omega$. The equivalence between between $ \bd_{\rho}$ and $\bd$ is the subject of the next lemma. 

\begin{lemma} \label{lem.comparemetrics}
    For any $\mu, \nu \in \cP(\ov{\Omega}/ \sim)$, we have 
    \begin{align} \label{drho.d}
     \frac{1}{2} \bd\big( \mu|_{\Omega}, \nu|_{\Omega} \big) \leq \bd_{\rho}(\mu,\nu) \leq \bd\big( \mu|_{\Omega}, \nu|_{\Omega} \big)
    \end{align}
\end{lemma}

\begin{proof}
    On the one hand, if $\phi$ is $1$-Lipschitz with respect to $\rho$, then $\phi$ is constant on the boundary, so that, setting $\tilde{\phi} = \phi - \phi(\partial \Omega)$, we have $\tilde{\phi} \in E$, and thus
    \begin{align*}
        \int_{\ov{\Omega}} \phi d(\mu - \nu) = \int_{\ov{\Omega}} \tilde{\phi} d(\mu - \nu) = \int_{\Omega} \tilde{\phi} d( \mu|_{\Omega} - \nu|_{\Omega}) \leq \bd(\mu|_{\Omega}, \nu|_{\Omega}), 
    \end{align*}
    and so taking a supremum over $\phi$ gives the second bound in \eqref{drho.d}.
    On the other hand, if $\phi \in E$, then from Lemma \ref{lem.rho}, we see that $\phi/2$ is $1$-Lipschitz with respect to $\rho$, so that 
    \begin{align*}
         \int_{\Omega} \phi d( \mu|_{\Omega} - \nu|_{\Omega})= 2 \int_{\ov{\Omega}/ \sim} \frac{\phi}{2} d(\mu - \nu) \leq 2\bd_{\rho}(\mu,\nu), 
    \end{align*}
    and so taking a supremum over $\phi$ gives the first upper bound in \eqref{drho.d}. 
 \end{proof}
For $\bx \in \ov{\Omega}^N$, we can view $m_{\bx}^N$ as a probability measure on $\ov{\Omega}/ \sim$ in an obvious way, i.e. we abuse notation by identifying $m_{\bx}^N$ with the measure $r_{\#} m_{\bx}^N$, with $r : \ov{\Omega} \to \ov{\Omega}/ \sim$ the quotient map. With this notation in place, we have the following standard result which we state without proof.

\begin{lemma} \label{lem.drho.disctete}
    For any $\bx, \by \in \ov{\Omega}^N$, we have 
    \begin{align*}
        \bd_{\rho}(m_{\bx}^N,m_{\by}^N) = \inf_{\sigma} \frac{1}{N} \sum_{i= 1}^N \rho(x^i,y^{\sigma(i)}), 
    \end{align*}
    with the infimum taken over all permutations $\sigma : \{1,...,N\} \to \{1,...,N\}$.
\end{lemma}

An important step  towards   establishing  the Lipschitz continuity of the $V^{N,K}$'s is described and proved next.

\begin{proposition} \label{prop.lip} There is a constant $C_0$ which is independent of $N$ such that, for any $t\in [0,T]$ and $\bx,\by\in \bar \Omega^N$, 
$$
|V^{N,N}(t,\bx)-V^{N,N}(t,\by)| \leq C_0 \bd_{\rho}( m^{N,N}_{\bx},m^{N,N}_{\by}).
$$ 
\end{proposition}

\begin{proof} Given  Lemma~\ref{lem.drho.disctete}, we only  need to check the existence of a constant $C_0$ such that, for all $t\in [0,T]$ and $\bx,\by\in \bar \Omega^N$,  
\[
|V^{N,N}(t,\bx)-V^{N,N}(t,\by)|\leq \frac{C_0}{N} \sum_{i=1}^N \rho(x^i ,y^i). 
\]
In view of the continuity of $V^{N,N}$, we can also reduce the proof to the case where $\bx,\by\in  \Omega^N$. \\

We fix  $\bx,\by\in  \Omega^N$ and  we first assume  that $x^i = y^i$ for any $i\leq N-1$. Then, we apply Proposition \ref{prop.DiVNK} to find a constant $C_1$ which is independent of $N$, such that
\[
\begin{split}
|V^{N,N}(t,\bx) &-V^{N,N}(t,\by)| \\
&= |V^{N,N}(t,x^1,...,x^{N-1},x^N ) -V^{N,N}(t,x^1,...,x^{N-1},y^N)| \leq \frac{C_1}{N}  |x^{N}-y^{N}|.
\end{split}
\]

In addition, recalling Lemma~\ref{lem.crudebound} and Lemma~\ref{lem.induction.gen}, there is a constant $C_2$ which is independent of $N$ such that
$$
|V^{N,N}(t,\bx) - V^{N,N-1}(t,\bx^{-N})| \leq \frac{C_2}{N} d_{\partial \Omega}(x^N). 
$$
Recalling that  $V^{N,N-1}(t,\bx^{-N})= V^{N,N-1}(t,\by^{-N})$, we also obtain the estimate 
\begin{align*}
|V^{N,N}(t,\bx)-V^{N,N}(t,\by)| &\leq |V^{N,N}(t,\bx)-V^{N,N-1}(t,\bx^{-N})| 
\\
&\qquad +|V^{N,N}(t,\by) -V^{N,N-1}(t,\bx^{-N})| \\ 
& \leq  \frac{C_2}{N} (d_{\partial \Omega}(x^{N})+ d_{\partial \Omega}(y^{N})).  
\end{align*}
Combining the two estimates above  and recalling \eqref{estidelta}, we have derived that 
$$
|V^{N,N}(t,\bx)-V^{N,N}(t,\by)| \leq \frac{C_0}{N} \rho(x^N,y^N), \quad C_0 = 2\max\{C_1,C_2\}.
$$
For the general case we note that

\begin{align*}
&|V^{N,N}(t,\bx)-V^{N,N}(t,\by)|\\
& \leq \sum_{i=0}^{N-1} | V^{N,N}( y^1, \dots, y^i, x^{i+1}, \dots, x^N)-V^{N,N}(y^1,\dots, y^{i+1},  x^{i+2}, \dots, x^N)| 
\leq \frac{C_0}{N} \sum_{i=1}^{N}  \rho(x^{i},y^{i}),  
\end{align*} 
where the second inequality comes from the previous discussion and the symmetry of $V^{N,N}$. Finally, the symmetry of $V^{N,N}$ and Lemma \ref{lem.drho.disctete} complete the proof.  \end{proof}

We are now ready for the proof of Theorem \ref{thm.vnk.lip}. 

\begin{proof}[Proof of Theorem \ref{thm.vnk.lip}]
    Fix $t \in [0,T]$, $\bx \in \Omega^K$, $\by \in \Omega^M$, for some $K,M \in \{1,...,N\}$, and let $z$ be any point in $\partial \Omega$. Then, using Proposition~\ref{prop.lip} and Lemma~\ref{lem.comparemetrics}, we find
    \begin{align*}
        |V^{N,K}(t,\bx) - V^{N,M}(t,\by)| &= |V^{N,N}(t,x^1,...,x^K,z,...,z) - V^{N,N}(t,y^1,...,y^M,z,...,z)| 
        \\
        &\leq C \bd_{\rho}\Big(m_{\bx}^{N,K} + \frac{N-K}{N} \delta_z, m_{\by}^{N,M} + \frac{N-K}{N} \delta_z\Big)
        \leq C \bd(m_{\bx}^{N,K}, m_{\by}^{N,K}). 
    \end{align*}
    
    For the time regularity, we fix $t_0 \in [0,T]$, $\bx_0 \in \ov{\Omega}^N$, and $N$ independent Brownian motions $W^1,...,W^N$, and we set
    \begin{align*}
        X_t^i = x_0^i + \sqrt{2}(W^i_t - W^i_{t_0}), \quad t_0 \leq t \leq T. 
    \end{align*}
    Let  $\tau^i$ be the first time $X_t^i$ hits $\partial \Omega$, that is, 
    \begin{align*}
        \tau^i = \inf\{t \geq t_0 : X_t^i \in \partial \Omega\}. 
    \end{align*}
    Recalling the equation for $V^{N,K}$, the symmetry of $V^{N,K}$, and the fact that $V^{N,N}(t,\bx) = V^{N,K}(t,x^1,...,x^K)$ if $x^{K+1},...,x^{N} \in \partial \Omega$, we find using It\^o's formula that 
    \begin{align} \label{vnn.ito}
        &d V^{N,N}(t, X_{t \wedge \tau^1}^1,...,X_{t \wedge \tau^N}^N) = \frac{1}{N} \sum_{i = 1}^N H\Big(X^i_{t \wedge \tau^i}, N D_{x^i} V^{N,N}(t,X_{t \wedge \tau^1}^1,...,X_{t \wedge \tau^N}^N) \Big) 1_{t < \tau^i}\; dt
       \nonumber \\
        &\qquad + \sum_{i = 1}^N \sqrt{2} D_{x^i} V^{N,N}(t,X_{t \wedge \tau^1}^1,...,X_{t \wedge \tau^N}^N)1_{t < \tau^i} dW_t^i.
    \end{align}
    We note that to obtain \eqref{vnn.ito}, we are applying It\^o's formula over a sequence of stochastic intervals; from time $t_0$ until the first particle exits, from the time of the first exit until the second exit, etc. 
    
    Integrating in time and taking  expectations we get
    \begin{align*}
      V^{N,N}(t_0,\bx_0) &= \E\bigg[V^{N,N}(t_0 + h, X_{(t_0 + h) \wedge \tau^1}^1,...,X_{(t_0 + h) \wedge \tau^N}^N) 
      \\
      &\qquad - \int_{t_0}^{t_0 + h} \frac{1}{N} \sum_{i = 1}^N H\Big(X^i_{t \wedge \tau^i}, N D_{x^i} V^{N,N}(t,X_{t \wedge \tau^1}^1,...,X_{t \wedge \tau^N}^N) \Big) 1_{t < \tau^i} dt \bigg] 
    \end{align*}
    In particular, since Proposition~\ref{prop.DiVNK} yields  that $|D_{x^i} V^{N,N}| \leq C/N,$
    it follows  that
    \begin{align*}
        \Big|   V^{N,N}(t_0,\bx_0) - \E\Big[V^{N,N}(t_0 + h, X_{(t_0 + h) \wedge \tau^1}^1,...,X_{(t_0 + h) \wedge \tau^N}^N) \Big] \Big| \leq Ch . 
    \end{align*}
    Applying the spatial Lipschitz bound we get
    \begin{align*}
        \big| V^{N,N}(t_0,\bx_0)& - V^{N,N}(t_0+h,\bx_0) \big| 
        \\
        &\leq Ch + \Big| V^{N,N}(t_0+h,\bx_0) - \E\Big[V^{N,N}(t_0 + h, X_{(t_0 + h) \wedge \tau^1}^1,...,X_{(t_0 + h) \wedge \tau^N}^N) \Big] \Big| 
        \\
        &\leq Ch + C\E\Big[ \frac{1}{N} \sum_{i = 1}^N \rho\big(x_0^i, X_{(t_0 + h) \wedge \tau^i}^i\big) \Big]
        \\
        &\leq Ch + C\E\Big[ \frac{1}{N} \sum_{i = 1}^N |x_0^i - X_{(t_0 + h) \wedge \tau^i}^i| \Big] \leq Ch + C \sqrt{h}, 
    \end{align*}
   the last bound following from the fact that, for each $i$,  $\langle X^i \rangle_t = 1_{t < \tau^i}$, so that
    \begin{align*}
        \E\big[ |x_0^i - X_{(t_0 + h) \wedge \tau^i}^i| \big] \leq \E\big[ |x_0^i - X_{(t_0 + h) \wedge \tau^i}^i|^2 \big]^{1/2} = \E\big[ (t_0 + h) \wedge \tau^i - t_0 \big]^{1/2} \leq h^{1/2}.
    \end{align*}  
    \end{proof}

\subsection{Limit points of the $V^{N,K}$.} The objective in this subsection is to show that every limit point is a viscosity solution of \eqref{hjbinf}. Together with Theorem \ref{thm.vnk.lip}, this will be enough to show that the convergence of $(V^{N,K})_{K = 1,...,N}$ towards the unique viscosity solution of \eqref{hjbinf}. We begin  making precise the  meaning  of a  limit point of the hierarchy \eqref{hjbnk}.

\begin{definition} \label{defn.limitpoint}
    
A continuous function
    $V: [0,T] \times \sub \to \R$
is a limit point of the hierarchy \eqref{hjbnk} 
if there exists  a subsequence $(N_j)_{j \in \N}$, such that $ \underset{{j \to \infty}} \lim  N_j=\infty$ and 
\begin{align} \label{def.limitpoint}
    \lim_{j \to \infty} \max_{K = 1,...,N_j} \sup_{(t,\bx) \in [0,T] \times \ov{\Omega}^K} \big|V(t,m_{\bx}^{N_j,K}) - V^{N_j,K}(t,\bx) \big| = 0. 
\end{align}
\end{definition}

An important first step is to show that all limit points of the hierarchy \eqref{hjbnk} are actually  solutions  to  \eqref{hjbinfR} for all large enough $R$.

\begin{proposition} \label{prop.subsol}
   There exists $R_0 > 0$ such that,  for any $R > R_0$ and any limit point $V$ of the hierarchy \eqref{hjbnk}, $V$ is a viscosity solution of \eqref{hjbinfR}.
\end{proposition}

\begin{proof}
The Lipschitz bound of Proposition~\ref{prop.lip} gives an  $R_0 > 0$ such that, for $R > R_0$, $(V^{N,K})_{K = 1,...,N}$ satisfies
\begin{align} \label{hjbnk.R} \tag{$\text{HJB}_{N,K,R}$}
    - \partial_t V^{N,K} - \sum_{i = 1}^K \Delta_{x^i} V^{N,K} + \frac{1}{N} \sum_{i = 1}^K H^R\big(x^i, ND_{x^i} V^{N,K} \big) =F\big(m_{\bx}^{N,K}\big).
\end{align}
    For notational simplicity, we do not relabel the subsequence along which $(V^{N,K})_{K= 1,...,N}$ converges to $V$, that is, instead of \eqref{def.limitpoint}, we write 
    \begin{align}  \label{uniformconv}
        \lim_{N \to \infty} \max_{K = 1,...,N} \sup_{(t,\bx) \in [0,T] \times \ov{\Omega}^K} \big|V(t,m_{\bx}^{N,K}) - V^{N,K}(t,\bx) \big| = 0. \ \ 
\end{align}
    Since the arguments are very similar, we only prove the subsolution property.
    
Let  $\Phi$ be  a smooth test function in the sense of Definition~\ref{def.testfunction}, and assume that $(t_0,m_0) \in [0,T) \times (\sub \cap L^2)$ is a maximizer in  \eqref{touchingfromabove}.  By replacing $\Phi$ with 
\begin{align*}
    \Phi(t,m) + \Big( |t-t_0|^2 + \|m - m_0\|_{-s}^2 \Big)
\end{align*}
for some $s > d/2 + 2$ (see Remark \ref{rmk.hminuss}), we see that it suffices to assume that $(t_0,m_0)$ is a strict maximizer. The  aim is to prove that $m_0 \in H_0^1$ and it satisfies  \eqref{subsol.bound} with $C$ a constant which is independent of $\Phi$, $t_0$, and $m_0$. 
\newline \newline 
    \textit{Step 1 - regularizing  the singular term.}    
    We  regularize the singular penalization term involving $\|m\|_2^2$ using   mollification (convolution) with the family of kernels $ \rho_{\kappa}(x) = \kappa^{-d}\rho(\frac{x}{\kappa}),$ with a parameter $\kappa>0$ and 
   $\rho : \R^d \to \R$ an even and smooth function  with $\rho \geq 0$, $\int_{\R^d} \rho = 1$, and $\rho = 0$ on $B_1^c$.

    For each  $\kappa > 0$, we introduce the optimization problem 
    \begin{align} \label{kappaproblem}
        \sup_{(t,m) \in [0,T] \times  \sub} \Big\{ V(t,m) - \Phi(t,m) - \delta \|m * \rho_{\kappa} \|_2^2 \Big\}.
    \end{align}
     We note that in \eqref{kappaproblem}, $\|m * \rho_{\kappa}\|_2^2 = \|m * \rho_{\kappa}\|_{L^2(\R^d)}^2 = \int_{\R^d} |m * \rho_{\kappa} (x)|^2 dx$, i.e. we are viewing $m * \rho_{\kappa}$ as a smooth function defined on $\R^d$ and computing its $L^2$ norm on the whole space. We now claim that for each fixed $\kappa > 0$,
    \begin{align*}
       \sub \ni m \mapsto \|m * \rho_{\kappa} \|_2^2
    \end{align*}
    is lower semi-continuous with respect to $\bd$. Indeed, suppose $(m_n)_{n \in \N}$ is a sequence in $\sub$ converging to $m \in \sub$ (with respect to $\bd$). Choose a subsequence $(m_{n_k})$ such that 
    \begin{align*}
        \liminf_{n \to \infty} \| m_n * \rho_{\kappa}\|_2^2 = \lim_{k \to \infty} \| m_{n_k} * \rho_{\kappa}\|_2^2.
    \end{align*}
    Then, passing to a further subsequence if necessary (which we do not relabel for simplicity), we can find $\ov m \in \sub(\ov \Omega)$ such that $m_{n_k} \to \ov m$ weakly in $\sub(\ov \Omega)$. It follows that $\ov m|_{\Omega} = m$, and hence 
    \begin{align*}
        \lim_{k \to \infty} \|m_{n_k} * \rho_{\kappa}\|_2^2 = \| \ov m * \rho_{\kappa}\|_2^2 \geq \| m * \rho_{\kappa}\|_2^2.
    \end{align*}
    This proves that $m \mapsto \|m * \rho_{\kappa}\|_2^2$ is lower semi-continuous with respect to $\bd$, and hence the objective appearing in \eqref{kappaproblem} is upper semi-continuous. 
   Since  $(\sub, \bd)$ is a compact metric space, it follows that for each $\kappa>0$,  there exists at least  one optimizer $(t_{\kappa},m_{\kappa})$ for the problem \eqref{kappaproblem}. 
   
   %
    
    %
    It is straightforward to check (using compactness arguments similar to those appearing in the proof of Proposition \ref{prop.equiv}) that any limit point $(\hat{t},\hat{m}) \in [0,T] \times \sub(\ov{\Omega})$ of $(t_{\kappa},m_{\kappa})$  with respect to weak-$*$ convergence is in fact in $[0,T] \times (\sub(\Omega) \cap L^2(\Omega))$, and satisfies
    \begin{align*}
        \|\hat{m}\|_2^2 \leq \liminf_{j \to \infty} \|m_j * \rho_{\kappa_j}\|_2^2, 
    \end{align*}
    where $\kappa_j \downarrow 0$ is a sequence such that $(t_{\kappa_j}, m_{\kappa_j}) \to (\hat{t},\hat{m})$. 
    Together with the optimality of $(t_j,m_j)$ for \eqref{kappaproblem} (with $\kappa = \kappa_j$) these two observations are  enough to ensure that $(\hat{t}, \hat{m})$ is an optimizer for the problem which defines $(t_0,m_0)$. Since $(t_0,m_0)$ was assumed to be the unique minimizer  \eqref{touchingfromabove}, we conclude that $(\hat{t},\hat{m}) = (t_0,m_0)$, hence,  $(t_{\kappa_j},m_{\kappa_j}) \to (t_0,m_0)$ in weak-$*$. 
    
    Finally, since this holds for any sequence $\kappa_j \to 0$, it follows  that any limit point of $t_{\kappa}$ must be $t_0$, which implies  that, for sufficiently small $\kappa$,
    
    \begin{align} \label{tkappa.terminal}
        t_{\kappa} < T.
    \end{align}
    \noindent     
    \textit{Step 2 - sending $N \to \infty$ with $\kappa$ fixed. }
    We next note that we can view $V$ and $\Phi$ as maps on $[0,T] \times \sub(\ov{\Omega})$ by setting, for $(t,m) \in [0,T] \times \sub(\ov{\Omega})$,
    \begin{align*}
        \ov{V}(t,m) = V(t,m|_{\Omega}) \text{ and} \  \ov{\Phi}(t,m) = \Phi(t,m|_{\Omega}),
    \end{align*}
   which implies  $\ov{V}$ and $\ov{\Phi}$ are continuous with respect to the weak-$*$ topology. 
   
   Next, for $N \in \N$ and $K\in \{0,1,\ldots, N\}$,  let $(K_N, t_N,\bx_N)$ be an optimizer of
    \begin{align*}
        \max_{K = 0,...,N} \sup_{(t,\bx) \in \ov{\Omega}^K} \Big\{ V^{N,K}(t,\bx)- \ov{\Phi}(t,m_{\bx}^{N,K}) - \delta \|m_{\bx}^{N,K} * \rho_{\kappa}\|_2^2 \Big\}.
    \end{align*}
     If  $m_N = m_{\bx_N}^{N,K_N}$,  by compactness $(t_N,m_N)$ converges in weak-$*$ along a subsequence, which we do not relabel for simplicity, to some $(t_{\kappa},m_{\kappa}) \in [0,T] \times \sub(\ov{\Omega})$. 
    
    It is now easy to check from the uniform convergence of the $V^{N,K}$'s that $(t_{\kappa}, m_{\kappa})$ must be an optimizer for the problem
    \begin{align}
        \label{kappaproblem2}
         \sup_{(t,m) \in [0,T] \times  \sub(\ov{\Omega})} \Big\{ \ov{V}(t,m) - \ov{\Phi}(t,m) - \delta \|m * \rho_{\kappa} \|_2^2 \Big\}, 
    \end{align}
    It is also clear that any optimizer for $\eqref{kappaproblem2}$ in fact satisfies $m_{\kappa} \in \sub(\Omega)$, that is, $m_{\kappa}(\partial \Omega) = 0$, and is also an optimizer for \eqref{kappaproblem}.

    In particular, since the $(t_N,m_N)$'s are  converging,  along a subsequence which we do not relabel,  to some minimizer $(t_{\kappa},m_{\kappa})$ of \eqref{kappaproblem}, we obtain from \eqref{tkappa.terminal} that, for all $N$ large enough, $t_N < T$.

    In addition, we note that, for each $N \in \N$, we must have $\bx_N^i \in \Omega$ for each $i = 1,...,K_N$. Indeed, if $x_N^i \in \partial \Omega$ for some $i$, then
    \begin{align*}
        &V^{N,K_N - 1}(t_N,\bx_N^{-i}) = V^{N, K_N}(t_N,\bx_N), \quad \ov{\Phi}(t,m_{\bx_N^{-i}}^{N,K_N - 1}) = \ov{\Phi}(t,m_{\bx_N}^{N,K_N}), \text{  and  } \\
        &\|m_{\bx_N^{-i}}^{N,K_N-1} * \rho_{\kappa}\|_2^2  < \|m_{\bx_N}^{N,K_N} * \rho_{\kappa}\|_2^2, 
    \end{align*}
   a fact that implies that 
    \begin{align*}
        V^{N,K_N-1}(t_N,\bx_N) - \ov{\Phi}(t_N,m_{\bx_N}^{N,K_N - 1}) - \delta \|m_{\bx}^{N,K_N - 1} * \rho_{\kappa}\|_2^2 > V^{N,K_N}(t,\bx)- \ov{\Phi}(t,m_{\bx}^{N,K_N}) - \delta \|m_{\bx}^{N,K_N} * \rho_{\kappa}\|_2^2, 
    \end{align*}
    and hence a contradiction to the optimality of $(K_N,t_N,\bx_N)$. 
    
    Since $t_N < T$, we can now apply the subsolution property of $V^{N,K_N}$ for $K_N \in\{1,\ldots, N\}$ or the \\
    fact that $V^{N,0}(t) = G(\bzero) - (T-t) F(\bzero)$ if $K_N = 0$ to find that, for all $N$ large enough, we have
    \begin{align} \label{subsol.f}
        - \partial_t f(t_N,\bx_N) - \sum_{i = 1}^{K_N} \Delta_{x^i} f(t,\bx_N) + \frac{1}{N} \sum_{i = 1}^{K_N} H^R\big( x_N^i, ND_{x^i} f(t_N,\bx_N) \big) \leq F(m_{\bx_N}^{N,K_N}), 
    \end{align}
    where $f : [0,T] \times \Omega^{K_N} \to \R$ is given by $f(t,\bx) = \Phi(t,m_{\bx}^{N,K_N}) - \delta \| m_{\bx}^{N,K} * \rho_{\kappa}\|_2^2.$
    Using Lemma \ref{lem.projection}, which is stated and proved after  the end of the ongoing proof, to compute the derivatives of $f$, we obtain from \eqref{subsol.f} the inequality
    \begin{equation}\label{discrete.subsoltest}
    \begin{split}
         &- \partial_t \Phi(t_N,m_N) - \int_{\Omega} \Delta_x \frac{\delta \Phi}{\delta m}(t_N,m_N, y) dm_{N}(y) + 2 \delta  \int_{\R^d} |D \rho_k * m_N|^2 dy 
        \\ 
        &\qquad + \int_{\Omega} H^R\Big(y, D_x \frac{\delta \Phi}{\delta m}(t_N, m_N, y) + 2\delta \big( D \rho_{\kappa} * \rho_{\kappa} * m_N \big)(y) \Big) dm_N(y) \leq F(m_N) + \frac{C_{\kappa}}{N}, 
    \end{split}
    \end{equation}
    with $C_{\kappa}$ a constant which can depend on $\kappa$ but not on $N$. 
     
    Now, we recall that we have, again up to a subsequence which we have not relabeled, that, as $N\to \infty$  and weakly-$*$,
    \begin{align*}
        (t_N,m_N) \rightharpoonup (t_{\kappa},m_{\kappa}). 
    \end{align*}
    It also follows from  the definition of the  smooth test functions, that,  as $N\to \infty$ and uniformly, 
    \begin{align*}
       &D_x \frac{\delta \Phi}{\delta m}(t_{N}, m_{N}, \cdot) \to D_x \frac{\delta \Phi}{\delta m}(t_{\kappa}, m_{\kappa}, \cdot) \ \text{and} \  \Delta_x \frac{\delta \Phi}{\delta m}(t_{N}, m_{N}, \cdot) \to \Delta_x \frac{\delta \Phi}{\delta m}(t_{\kappa}, m_{\kappa}, \cdot).
    \end{align*}
    Finally, it is clear that, as $N\to \infty$ and uniformly, 
    \begin{align*}
       &\qquad \qquad \qquad \qquad D \rho_k * m_{N_j} \to  D \rho_{\kappa} * m_{\kappa}.
    \end{align*}
   Thus we can pass to the $N\to$ limit in \eqref{discrete.subsoltest} to find that $(t_{\kappa}, m_{\kappa}) \in [0,T] \times \sub$ satisfies 
    
    \begin{equation} \label{subsoltest.eta}
    \begin{split}
        &- \partial_t \Phi(t_{\kappa},m_{\kappa}) - \int_{\Omega} \Delta_x \frac{\delta \Phi}{\delta m}(t_{\kappa},m_{\kappa}, y) dm_{N}(y) + 2\delta \int_{\R^d} |D \rho_k * m_{\kappa}|^2 dy 
        \\ 
        &\qquad + \int_{\Omega} H^R\Big(y, D_x \frac{\delta \Phi}{\delta m}(t_{\kappa}, m_{\kappa}, y) + 2\delta \big( D \rho_{\kappa} * \rho_{\kappa} * m_{\kappa} \big)(y) \Big) dm_{\kappa}(y) \leq F(m_{\kappa}).
    \end{split}
    \end{equation}
    
    Using  the Lipschitz continuity of $H^R$ with a constant $C_R$, we have
    \begin{align*}
        \int_{\Omega} &H^R\Big(y, D_x \frac{\delta \Phi}{\delta m}(t_{\kappa}, m_{\kappa}, y) + 2 \delta \big( D \rho_{\kappa} * \rho_{\kappa} * m_{\kappa} \big)(y) \Big) dm_{\kappa}(y) 
        \\
        & \leq \int_{\Omega} H^R\Big(y, D_x \frac{\delta \Phi}{\delta m}(t_{\kappa}, m_{\kappa}, y) \big)(y) \Big) dm_{\kappa}(y) + 2C_R  \delta \int_{\Omega} \Big|  \rho_{\kappa} * D \rho_{\kappa} * m_{\kappa}  \Big| dm_{\kappa}(y)
        \\
        &\leq \int_{\Omega} H^R\Big(y, D_x \frac{\delta \Phi}{\delta m}(t_{\kappa}, m_{\kappa}, y) \big)(y) \Big) dm_{\kappa}(y) + 2C_R \delta  \int_{\Omega} \rho_{\kappa} * \Big|   D \rho_{\kappa} * m_{\kappa}  \Big| dm_{\kappa}(y)
        \\
        &= \int_{\Omega} H^R\Big(y, D_x \frac{\delta \Phi}{\delta m}(t_{\kappa}, m_{\kappa}, y) \big)(y) \Big) dm_{\kappa}(y) + 2C_R \delta \int_{\Omega} \Big|   D \rho_{\kappa} * m_{\kappa}  \Big| \rho_{\kappa} * m_{\kappa}(y) dy
        \\
        &\leq \int_{\Omega} H^R\Big(y, D_x \frac{\delta \Phi}{\delta m}(t_{\kappa}, m_{\kappa}, y) \big)(y) \Big) dm_{\kappa}(y) + \delta  \int_{\Omega} \big| D \rho_{\kappa} * m_{\kappa}(y) \big|^2 dy + C_R^2 \delta  \int_{\Omega} |\rho_{\kappa} * m_{\kappa}(y)|^2 dy.
    \end{align*}
    
    Coming back to \eqref{subsoltest.eta}, we find that, in fact, $(t_{\kappa}, m_{\kappa})$ satisfies 
    \begin{equation} \label{subsoltest.eta.2}
    \begin{split}
         &- \partial_t \Phi(t_{\kappa},m_{\kappa}) - \int_{\Omega} \Delta_x \frac{\delta \Phi}{\delta m}(t_{\kappa},m_{\kappa}, y) dm_{\kappa}(y) + \delta \int_{\R^d} |D \rho_k * m_{\kappa}|^2 dy 
        \\ 
        &\qquad + \int_{\Omega} H^R\Big(y, D_x \frac{\delta \Phi}{\delta m}(t_{\kappa}, m_{\kappa}, y) \Big) dm_{\kappa}(y) \leq F(m_{\kappa}) + C_R^2 \delta \|m* \rho_{\kappa}\|_2^2.
   \end{split}
    \end{equation}
   We have thus shown that for $\kappa > 0$ small enough, there exists a pair $(t_{\kappa}, m_{\kappa}) \in [0,T) \times \sub$ which is an optimizer  for \eqref{kappaproblem}, and such that \eqref{subsoltest.eta.2} holds. 
\newline \newline     
    \textit{Step 3 - sending $\kappa \to 0$.}     
    From Step 1, we know that $(t_{\kappa},m_{\kappa}) \to (t_0,m_0)$ in weak-$*$. Moreover, because $(t_{\kappa},m_{\kappa})$ is an optimizer for \eqref{kappaproblem}, we have that the $\rho_{\kappa} * m_{\kappa}$'s are bounded in $L^2$, and then, because of \eqref{subsoltest.eta.2}, we see that in fact the $\rho_{\kappa} * m_{\kappa}$'s are  bounded in $H_0^1(\R^d)$. The key observation here is that both bounds are  uniform in $\kappa$. 
    
    Arguing as in Step 1, we see that any limit point of $\rho_{\kappa} * m_{\kappa}$ with respect to the weak topology on $H_0^1(\R^d)$ must coincide with $m_0$, and so in fact we have that, in the limit $\kappa \to 0$, 
 \begin{equation}\label{eta.conv}
  t_{\kappa} \to t_0, \ \ m_{\kappa} \to  m_0 \ \text{weakly-$*$  in  $\sub(\ov{\Omega})$ }, \ \  \rho_{\kappa} * m_{\kappa} \to m_0 \ \text{weakly in } \  H_0^1(\R^d)
  \end{equation}
    Next, we note that,  because $m_0 \in H_0^1(\R^n)$ and $m_0 = 0$ on $\Omega^c$, we have $m_0 \in H_0^1(\Omega)$ and by the weak convergence in the limit $\kappa\to 0$,
    \begin{align*}
        \int_{\Omega} |D m_0|^2 dx \leq \liminf_{\kappa \to 0} \int_{\R^d} |D \rho_{\kappa} * m_{\eta}|^2 dx. 
    \end{align*}
    Now using \eqref{eta.conv}, we can pass to the limit in \eqref{subsoltest.eta.2} to find that 
   \begin{equation} \label{subsoltest}
   \begin{split}
         &- \partial_t \Phi(t_0,m_0) - \int_{\Omega} \Delta_x \frac{\delta \Phi}{\delta m}(t_0,m_0, y)m_0(y) dy + \int_{\R^d} |D m_{0}|^2 dy 
        \\ 
        &\qquad + \int_{\Omega} H^R\Big(y, D_x \frac{\delta \Phi}{\delta m}(t_0,m_0, y) \Big) m_{0}(y) dy \leq F(m_{0}) + C_R^2 \delta \|m_0\|_2^2, 
    \end{split}
    \end{equation}
    Finally, using the fact that $m_0 \in H_0^1(\Omega)$, we can integrate by parts to find that 
    \begin{align*}
       \int_{\Omega} \Delta_x \frac{\delta \Phi}{\delta m}(t_0,m_0, y)m_0(y) dy = - \int_{\Omega} D_x \frac{\delta \Phi}{\delta m}(t_0,m_0, y) \cdot Dm_0(y) dy, 
    \end{align*}
    which completes the proof.  
    \end{proof}

We complete the section with some  technical facts that were used in the previous proof.

\begin{lemma}
\label{lem.projection}
Let $\Phi$ be a smooth test function, $\delta, \kappa > 0$, $N \in \N$, $K \in \{1,...,N\}$, and define 
\begin{align*}
   \Phi^{N,K} : [0,T] \times \Omega^K \to \R, \quad \Psi^{N,K} : \Omega^K \to \R
\end{align*}
by 
\begin{align*}
    \Phi^{N,K}(t,\bx) = \Phi(t,m_{\bx}^{N,K}), \quad \Psi^{N,K}(t,\bx) = \|m_{\bx}^{N,K} * \rho_{\kappa} \|_2^2.
\end{align*}
Then $\Phi^{N,K}, \Psi^{N,K} \in C^{1,2}([0,T] \times \Omega^K)$, and satisfies 
\begin{align*}
    &\partial_t \Phi^{N,K}(t,\bx) = \partial_t \Phi(t,m_{\bx}^{N,K}), 
    \\[1.5mm]
    &D_{x^i} \Phi^{N,K}(t,\bx) = \frac{1}{N} D_x \frac{\delta \Phi}{\delta m}(t,m_{\bx}^{N,K}, x^i), 
    \\[1.5mm]
    &\Delta_{x^i} \Phi^{N,K}(t,\bx) = \frac{1}{N} \Delta_x \frac{\delta \Phi}{\delta m}(t,m_{\bx}{N,K},x^i) + \frac{1}{N^2} \tr\big( D_{x,y} \frac{\delta \Phi}{\delta m} (t,m_{\bx}^{N,K}, x^i,x^i) \big).
\end{align*}
and 
    \begin{align*}
        &D_{x^i} \Psi^{N,K}(\bx) = \frac{2}{N} \Big( D\rho_{\kappa} * \rho_{\kappa} * m_{\bx}^{N,K} \Big) (x^i), 
        \\[1.5mm]
        &\Delta_{x^i} \Psi^{N,K}(t,\bx) =  \frac{2}{N} \Big( D\rho_{\kappa} * D\rho_{\kappa} * m_{\bx}^{N,K} \Big) (x^i) + \frac{2}{N^2} \|D \rho_{\kappa}\|_2^2, 
        \\[1.5mm]
        &\sum_{i = 1}^N \Delta_{x^i} \Psi^{N,K}(t,\bx) = - 2 \| D \rho_{\kappa} * m_{\bx}^{N,K} \|_2^2 + \frac{2}{N} \|D \rho_{\kappa}\|_2^2.
    \end{align*}

\end{lemma}
\begin{proof}
    The expressions for the derivatives of $\Phi^{N,K}$ can be deduced from, for example,  \cite[Proposition 6.30]{CarmonaDelarue_book_II}, and the expressions for the derivatives of $\Psi^{N,K}$ can be obtained by explicit computation. We omit the details. 
 \end{proof}

\section{The properties of $U$.} \label{sec.u}

For technical reasons, for each $R > 0$ we introduce the value function $U^R : [0,T] \times \sub \to \R$, defined exactly like $U$, except with the additional constraint that $|\alpha(t,x)| \leq R$, that is, for each 
\begin{equation}\label{Rvaluefunction}
U^R= \inf_{(m,\alpha)\in \mathcal A^R(t_0,m_0)} \bigg\{ \int_{t_0}^T \Big( \int_{\Omega} L\big( x, \alpha(t,x) \big)m_t(dx) + F(m_t) \Big) dt + G(m_T) \bigg\},
\end{equation}
where $ \mathcal A^R(t_0,m_0)= \{(m,\alpha)\in \mathcal A^R(t_0,m_0) : | \alpha(t,x)| \leq R \text{ for all } (t,x)\}$.

\subsection{The Fokker-Planck equation and its dual} 
In order to investigate the properties of $U$ and $U^R$, it will be useful to collect a few fairly straightforward facts about the Fokker-Planck-type initial boundary bound problem 
\begin{align} \label{eq.FP}
    \partial_t m = \Delta m - \text{div}(m \alpha) \ \text{ in}\  (t_0,T) \times \Omega, \quad m_{t_0} = m_0, \quad m_t|_{\partial \Omega} = 0.
\end{align}
Because we will be making arguments about \eqref{eq.FP} via duality, it will also be useful to study the backwards problem
\begin{align} \label{dualeqn}
    \partial_t \phi + \Delta \phi + \alpha \cdot D \phi = \beta, \text{  in  } (t_1,t_2) \times \Omega, \quad \phi(t_2,x) = f(x), \quad \phi|_{[t_1,t_2] \times \partial \Omega} = 0.
\end{align}
Here $0 \leq t_1 < t_2 \leq T$, and $\alpha : [t_1,t_2] \times \Omega \to \R^d$, $\beta : [t_1,t_2] \times \Omega \to \R$, and $f : \Omega \to \R$ are given. We fix $p > d + 2$ throughout this section. We begin with a lemma about \eqref{dualeqn}. 
\begin{lemma} \label{lem.duality}
    Suppose that $f \in C_c^{\infty}(\Omega)$ and $\alpha, \beta \in L^{\infty}$. Then there is a unique strong solution to \eqref{dualeqn} in the parabolic Sobolev space $W_{t,x}^{2,p}([t_1,t_2] \times \Omega)$, which satisfies the bound 
    \begin{align} \label{dual.apriori}
        |\phi(t,x) - \phi(s,y)| \leq C\big(|x-y| + |t-s|^{1/2}\big), \quad C = C(\|\alpha\|_{\infty}, \|\beta\|_{\infty}, \|D f\|_{\infty}). 
    \end{align}
    Moreover, if $m \in C([t_0,T] ; \sub(\Omega))$ is any weak solution to \eqref{dualeqn} for $0 \leq t_0 \leq t_1$, then we have 
    \begin{align} \label{duality.nonsmooth}
        \int_{\Omega} f(x) m_{t_2}(dx) = \int_{\Omega} \phi(t_1,x) m_{t_1}(dx) + \int_{t_1}^{t_2} \beta(t,x)  m_t(dx) dt.
    \end{align}
\end{lemma}
\begin{proof}
We start by assuming that we have a solution in to \eqref{dualeqn}, with $\alpha$ and $\beta$ bounded. Let $(P_t)_{t \geq 0}$ be the heat semigroup on $\Omega$ with Dirichlet boundary conditions. Then, Duhamel's  formula gives 
   \begin{align*}
       \phi(s, \cdot) = P_{t_2-s} f + \int_s^t P_{r-s} \big( \alpha(r,\cdot) \cdot D\phi(r,\cdot) \big) dr - \int_s^t P_{r-s} \beta(r,\cdot) dr.
   \end{align*}
Using the smoothing effect of the heat semigroup  (see, for example, \cite[Theorem 1.1]{FurutoIwabuchi} for the relevant result on a bounded domain with $C^3$ boundary), we get 
   \begin{align*}
       \|D \phi(s,\cdot)\|_{\infty} \leq C \Big( \|D f \|_{\infty} + \|\alpha \|_{\infty} \int_s^t \frac{1}{\sqrt{r-s}} \|D \phi (r,\cdot)\|_{\infty}dr + \|\beta\|_{\infty} \int_s^t \frac{1}{\sqrt{r-s}} dr \Big), 
   \end{align*}
   which implies  (by an extension of Gr\"onwall's lemma from \cite{YE20071075}, see \cite[Lemma A.1]{DJS2025} for the precise version we need)  that
    \begin{align} \label{dphi.est}
        \|D\phi \|_{\infty} \leq C(\|\alpha\|_{\infty}, \|\beta\|_{\infty}) \|Df \|_{\infty}. 
    \end{align} 
    The H\"older regularity in time can also be easily inferred from the Duhamel formula, which leads to \eqref{dual.apriori}. With the bound on $\|D \phi\|_{\infty}$ in hand, $\phi$ solves the heat equation with a bounded forcing term; thus the a-priori estimate 
    \begin{align*}
        \| \phi \|_{W_{t,x}^{2, p}} = \| \phi\|_{L_{t,x}^p} + \|\partial_t \phi \|_{L_{t,x}^p} + \|D \phi \|_{L_{t,x}^p} + \|D^2 \phi\|_{L_{t,x}^p} \leq C = C\big(\|\alpha\|_{\infty}, \|\beta\|_{\infty}, \|f \|_{C^2} \big)
    \end{align*}
    follows from the Calder\'on Zygmund estimates for the heat equation (see e.g. \cite{waterswang}). Existence and uniqueness for \eqref{dualeqn} can both be obtained from this a-priori estimate. 

    The duality relationship \eqref{duality.nonsmooth} would follow immediately from the definitions, except that because $\alpha, \beta$ are just bounded and measurable, the solution of \eqref{dualeqn} may not be $C^{1,2}$. To circumvent this issue, we mollify $\alpha$ and $\beta$ to produce smooth functions $\alpha^{\eps}, \beta^{\eps}$ such  that $\alpha^{\eps}$ and $\beta^{\eps}$ are bounded uniformly in $\eps$ and $\alpha^{\eps} \to \alpha$, $\beta^{\eps} \to \beta$ almost everywhere. We define $\phi^{\eps}$ exactly like $\phi$, except with $\alpha^{\eps}$ replacing $\alpha$ and $\beta^{\eps}$ replacing $\eta$. Then, since $\partial \Omega$ is $C^3$, $\alpha^{\eps}$ and $\beta^{\eps}$ are smooth, and $f \in C_c^{\infty}(\Omega)$, $\phi^{\eps}$ is smooth enough to act as a test function in the equation for $m$. We thus get 
\begin{align*}
    \int f d{m_{t_2}} - \int \phi^{\eps}(t_1,x) m_{t_2}(dx) = \int_{t_1}^{t_2} \int_{\Omega} (\alpha - \alpha^{\eps}) D\phi^{\eps} m_t(dx) dt + \int_{t_1}^{t_2} \int_{\Omega} \beta^{\eps} m_t(dx) dt.
\end{align*}
We know from \eqref{dphi.est} that $D \phi^{\eps}$ is bounded independently of $\eps$. Moreover, the measure $m_t(dx) dt$ has a density with respect to the Lebesgue measure (see e.g. \cite[Corollary 6.3.2]{bogachevkrylovrocker}). Thus, we can pass to the limit by the dominated convergence theorem to deduce \eqref{duality.nonsmooth}. 
\end{proof}

\begin{proposition} \label{prop.FP}
    Suppose that $t_0 \in [0,T)$, $m_0 \in \sub$, and $\alpha : [t_0,T] \times \R^d \to \R^d$ is a measurable and bounded function. Then   \eqref{eq.FP} has a unique distributional solution in the space $C([t_0,T] ; \sub)$ with continuity understood with respect to $\bd$. Moreover, for any smooth test function $\Phi$ in the sense of Definition \ref{def.testfunction}, and for any $h \in (0,T-t_0]$,
    \begin{equation} \label{Phi.expansion}
    \begin{split}
        \Phi(&t_0 + h, m_{t_0 + h}) - \Phi(t_0,m_0) = \int_{t_0}^{t_0+h} \partial_t \Phi(t,m_t) dt 
     \\
        &\int_{t_0}^{t_0+h}\int_{\Omega} \Big( \Delta_x \frac{\delta \Phi}{\delta m}(m_t,x)m_t(x) + D_x \frac{\delta \Phi}{\delta m}(m_t,x) \cdot \alpha(x) \Big) m_t(dx) dt
    \end{split}
    \end{equation}
    
    Finally, if $m_0 \in L^2 \cap \sub$, then we further have $m \in L_t^{\infty} L_x^2 \cap L_t^2 H_0^1$, and, for any $h \in (0,T-t_0]$, 
    \begin{align} \label{l2.decrease}
      \|m_{t_0+h}\|_2^2 - \|m_{0}\|_2^2 = -2\int_{t_0}^{t_0 + h} \int_{\Omega} |Dm_t|^2 dx dt + 2 \int_{t_0}^{t_0 + h} \int_{\Omega} Dm_t \cdot \alpha \, m_t dx dt.
    \end{align}
\end{proposition}
\begin{proof}

Uniqueness follows directly from Lemma \ref{lem.duality} (taking $t_1 = t_0$ in \eqref{duality.nonsmooth}). Moreover, the H\"older regularity of $\phi$ in Lemma \ref{lem.duality} together wit the duality result \eqref{duality.nonsmooth} leads to an a-priori estimate: any solution $m \in C([t_0,T] \times \sub)$ to \eqref{eq.FP} satisfies
\begin{align*}
    \bd( m_{t_1},m_{t_2}) \leq C |t-s|^{1/2}.
\end{align*}
This a-priori estimate in turn leads existence by a compactness argument; we first mollify $m_0$ and $\alpha$, solve the resulting equations (which have classical solutions), and extract a subsequential limit in the pace $C([t_1,t_2]; \sub)$ via the Arzel\`a-Ascoli compactness criterion. The expansion \eqref{Phi.expansion}; meanwhile, is a straightforward extension of \cite[Theorem 5.99]{CarmonaDelarue_book_I}.

A more subtle point is the equality \eqref{l2.decrease} when $m_0 \in L^2$, which we now check. The strategy is to  approximate \eqref{eq.FP} by a sequence of smoother problem whose solutions satisfy \eqref{l2.decrease} and have the necessary compactness properties to pass in the limit to get \eqref{l2.decrease} for 
$m_0 \in L^2$.

We  choose a sequence $(m_0^{\eps})_{\eps >0}$  in $\sub \cap C_c^{\infty}(\Omega)$ such that, when $\eps \to 0$, $m_0^{\eps} \to m_0$ in $L^2$, and a sequence of smooth maps $\alpha^{\eps} : [t_0,T] \times \Omega \to \R^d$ which are  bounded in $L^{\infty}$, uniformly in $\eps$,  and, when $\eps\to 0$,  $\alpha^{\eps} \to \alpha$ in $L^p$ for any $p < \infty$. 

The solution  $m^{\eps}$ of the initial boundary problem
    \begin{equation} \label{eq.FPeps}
    \begin{split}
    \partial_t m^{\eps} = \Delta m^{\eps} - \text{div}(m^{\eps} \alpha) \ \text{in} \ [t_0,T] \times \Omega, \quad m^{\eps}_{t_0} = m^{\eps}_0, \quad   m^{\eps}_t|_{\partial \Omega} = 0
\end{split}
\end{equation}
is smooth, and,  satisfies the identity 
 \begin{equation} \label{l2.equality}
      \|m^{\eps}_{t_0+h}\|_2^2 - \|m^{\eps}_{0}\|_2^2 = -2\int_{t_0}^{t_0 + h} \int_{\Omega} |Dm^{\eps}_t|^2 dx dt + 2 \int_{t_0}^{t_0 + h} \int_{\Omega} Dm_t \cdot \alpha^{\eps} \, m^{\eps}_t dx dt.
    \end{equation}
    
 Elementary manipulations (Cauchy-Schwarz inequality and integration by parts)  and  Sobolev imbeddings establish that  the $m^{\eps}$'s are  bounded in $L_t^{\infty} L_x^2 \cap L_t^2 H_0^1$ uniformly in $\eps$, and,  hence,  in $L_t^2 L_x^p$, with $p  > 2$, and, in particular, $p = 2d/(d-2)$  when $d > 2$, and any $2 < p < \infty$ if $d = 1,2$. 
     
    We  claim that,  for any $2 < q < 3 \wedge (p/2 + 1)$, we have $m \in L_{t,x}^q$. Indeed,  using.  that, if $q < p/2 + 1$, then $ 2(q-1) < p$ and, if  $q < 3$, then $q - 1 < 2$, we have, for some $C$ depending on the universal in $\eps$ bounds on the $m_0^{\eps}$'s and the $a_\eps$,  the inequalities 
    \begin{align*}
        \int_{t_0}^T \int_{\Omega} |m_t^{\eps}|^q dx dt &\leq \int_0^T \Big( \int_{\Omega} |m_t^{\eps}|^2 \Big)^{1/2} \Big( \int_{\Omega} |m_t^{\eps}|^{2(q-1)} \Big)^{1/2} dt
        \\[2mm]
        &\leq \|m^{\eps}\|_{L_t^{\infty} L_x^2} \int_0^T \|m_t^{\eps}\|_{L_x^{2(q-1)}}^{q-1} dt \leq C\|m^{\eps}\|_{L_t^{\infty} L_x^2} \int_0^T \|m_t^{\eps}\|_{L_x^{p}}^{q-1} dt
        \\[2mm]
        &\leq C\|m^{\eps}\|_{L_t^{\infty} L_x^2} \norm{m^{\eps}}_{L_t^2 L_x^p}^{(q-1)/2}. 
    \end{align*}
    
    Now for $\eps, \delta > 0$, we compute
    \begin{align*}
        \frac{d}{dt} \frac{1}{2} &\int_{\Omega} |m_t^{\eps} - m_t^{\delta}|^2 = - \int_{\Omega} |Dm_t^{\eps} - Dm_t^{\delta}|^2dx  + \int_{\Omega} \big( Dm_t^{\eps} - Dm_t^{\delta} \big) \cdot (\alpha^{\eps} m_t^{\eps} - \alpha^{\delta} m_t^{\delta}\big) dx
        \\[2mm]                                                                                                                                                                                                                                                                                                                                                                                                                                                                                                                                                                                                                                                                                                                                                                                  
        &\leq - \frac{1}{2} \int_{\Omega} |Dm_t^{\eps} - Dm_t^{\delta}|^2 dx + \frac{1}{2} \int_{\Omega} \big| \alpha^{\eps} m_t^{\eps} - \alpha^{\delta} m_t^{\delta} \big|^2 dx
        \\[2mm]
        &\leq  - \frac{1}{2} \int_{\Omega} |Dm_t^{\eps} - Dm_t^{\delta}|^2 dx +  \int_{\Omega} |\alpha^{\eps}|^2 \big| m_t^{\eps} -  m_t^{\delta} \big|^2 dx + \int_{\Omega} |m_t^{\delta}|^2 |\alpha^{\eps} - \alpha^{\delta}|^2 dx 
        \\
        &\leq  - \frac{1}{2} \int_{\Omega} |Dm_t^{\eps} - Dm_t^{\delta}|^2 dx +  C\int_{\Omega} \big| m_t^{\eps} -  m_t^{\delta} \big|^2 dx + \int_{\Omega} |m_t^{\delta}|^2 |\alpha^{\eps} - \alpha^{\delta}|^2 dx, 
    \end{align*}
    with $C$ independent of $\eps$ and $\delta$. 
   
   Then, using  Gronwall's inequality, we deduce that, if $r =  \frac{2p}{p-2}$, then 
    \begin{align*}
        \| m^{\eps} - m^{\delta}\|_{L_t^{\infty} L_x^2} &+ \| m^{\eps} - m^{\delta}\|_{L_t^{2} H_0^1} \leq C \Big( \int_{t_0}^T \int_{\Omega} |m_t^{\delta}|^2 |\alpha^{\eps}(t,x) - \alpha^{\delta}(t,x)|^2 dx dt\Big)^{1/2} 
        \\
        &\leq C \| m^{\delta} \|_{L_{t,x}^p} \|\alpha^{\eps} - \alpha^{\delta}\|_{L_{t,x}^r}. 
    \end{align*}
    It follows that $(m^{\eps})_{\eps > 0}$ is Cauchy in $L_t^{\infty} L_x^2 \cap L_t^2 H_0^1$, and  converges in this space towards the solution $m$ to \eqref{eq.FP}. This allows us to pass to the limit in \eqref{l2.equality} to obtain \eqref{l2.decrease}.  
    \end{proof}

\subsection{Dynamic programming and space-time regularity of the value function.} The value functions $U$ and $U^R$ satisfy the classical dynamic programming principle, which is a crucial step of the 
proofs of the uniform continuity of the value functions and the viscosid  $U$ is  viscosity solution. 

\begin{lemma} \label{lem.dpp}
    For each $(t_0,m_0) \in [0,T) \times \sub$ and $h \in (0,T-t_0)$, we have 
    \begin{align*}
        U(t_0,m_0) = \inf_{(m,\alpha) \in {\mathcal A}(t_0,m_0;h)} \bigg\{ \int_{t_0}^{t_0 + h}\bigg( \int_{\Omega} L\big( x, \alpha(t,x) \big)m_t(dx) + F(m_t) \bigg) dt + U(t_0 + h, m_{t_0 + h}) \bigg\}, 
    \end{align*}
   where ${\mathcal A}(t_0,m_0;h)$ is the collection of pairs $(m,\alpha)$ consisting of a curve $[t_0,t_0+h] \ni t \mapsto m_t \in \sub$ and a measurable map $\alpha : [t_0,t_0 + h] \times \Omega \to \R^d$ such that
\begin{align*}
    \int_{t_0}^{t_0+h} \int_{\Omega} |\alpha(t,x)|^2 m_t(dx) < \infty, 
\end{align*}
and $m$ satisfies in the sense of distributions 
\begin{align*}
    \partial_t m = \Delta m - \text{div}(m\alpha) \ \text{in} \ (t_0, t_0+h)\times \Omega, \   m_{t_0} = m_0 \text{ in} \ \Omega, \ m = 0 \text{ on } (t_0,T) \times \partial \Omega. 
\end{align*}
An analogous dynamic programming principle holds for $U^R$.
\end{lemma}
    Since this is a deterministic optimal control problem, the proof is standard and is omitted.
The next result is about the existence of minimizers for $U$ and $U^R$.

\begin{proposition} \label{prop.optconditions}
    For each $(t_0,m_0) \in [0,T) \times \sub$, there is at least one minimizer $(m,\alpha)$ in the problem definition $U(t_0,m_0)$. Moreover, any minimizer satisfies 
    \begin{align*}
        \alpha(t,x) = - D_p H(x, D u(t,x)), \quad dt \otimes m_t(x) dx \,\, a.e., 
    \end{align*}
    where $(m,u)$ is some solution to
    \begin{align} \label{mfg}
        \begin{cases} 
        \ds - \partial_t u - \Delta u + H(x, Du) = \frac{\delta F}{\delta m} (m_t,x),                                                                                                              \ \text{in} \ [t_0,T] \times \Omega, \vspace{.1cm}
       \ds  \\
        \partial_t m = \Delta m + \text{div}\big(m_t D_pH(x,Du(t,x)) \big), \quad (t,x) \in [t_0,T] \times \Omega, 
      \ds   \\
      \ds   m_{t_0} = m_0, \quad u(T,x) = \frac{\delta G}{\delta m}(m,x), \quad u = m = 0 \text{ on } [t_0,T] \times \partial \Omega.
        \end{cases}
    \end{align}
\end{proposition}
 The proof of Proposition~\ref{prop.optconditions} is a straightforward extension of the arguments in the proof of the existence of minimizers and theoptimality conditions for standard MFC problem, as obtained in, for example, \cite[Proposition 3.1]{BrianiCardaliaguet}, so we omit it.
 
The next result is about the important fact that the second component of any minimizer for $U(t_0,m_0)$ is bounded by a constant which does not depend on the initial state and time. 
\begin{proposition} \label{prop.alphabounded}
    There is a constant $C$ which is independent of $(t_0,m_0)$ such that any solution to \eqref{mfg} satisfies $\|Du\|_{\infty} \leq C$.
    \end{proposition}
  Before we present the proof we state as a corollary an immediate  consequence of Proposition~\ref{prop.alphabounded}.
  
 \begin{corollary}
There is a constant $C$ such that, for any $(t_0,m_0) \in [0,T) \times \sub$ and any 
optimal control $\alpha$ for the problem defining $U(t_0,m_0)$, we have $\|\alpha\|_{\infty} \leq C$.
\end{corollary} 
We continue with the proof of the last proposition.
\begin{proof}[Proof of Proposition \ref{prop.alphabounded}.]

    Fix a curve $t \mapsto m_t \in \sub$, and let $u$ be the unique viscosity solution of 
    \begin{align} \label{little.u}
        \begin{cases}
           & - \partial_t u - \Delta u + H(x,Du) = \frac{\delta F}{\delta m}(m_t,x) \  \text{in} \   [t_0,T) \times \Omega, 
            \\[2mm]
            & u(T,\cdot) = \frac{\delta G}{\delta m}(m,\cdot)  \ \text{on} \  \Omega, \qquad  u = 0 \ \text{ on } \  [t_0,T] \times \partial \Omega.
        \end{cases}
    \end{align}
   It follows from a straightforward  application of the maximum  principle and the assumptions on $H, F$ and $G$ that there exists a constant $C$ which is independent of $m$ and $t_0$ such that $\|u\|_{\infty} \leq C.$ 
    
   Next choose $C$ large enough that
    \begin{align*}
        \|H(\cdot,0)\| + \norm{\frac{\delta F}{\delta m}}_{\infty} \leq C, \quad \|u\|_{\infty} \leq C, \quad 
         |\frac{\delta G}{\delta m}(m_T,x )| \leq C d_{\partial \Omega}(x)
    \end{align*}
    Using Lemma \ref{lem.barrier}, we can find $\eps > 0$ and two smooth functions  $ \phi^+ , \phi^- : \overline \cN^{\eps}\to \R$ such that  
    \begin{equation*}
    \begin{split}
    & (t,x) \to \phi^+(x) \text{ is a supersolution to the equation satisfied by $u$ and}\\
   & \phi^+\geq u \ \text{on} \ \big([t_0,T] \times \partial^+ \cN^{\eps} \big) \cup \big(\{T\} \times \Omega\big) \big),
        \end{split}
        \end{equation*}
and
\begin{equation*}
    \begin{split}
& (t,x) \to  \phi^-(x) \text{ is a subsolution to the equation satisfied by $u$ and}\\
& \phi^-\leq u \ \text{on} \ \big([t_0,T] \times \partial^+ \cN^{\eps} \big) \cup \big(\{T\} \times \Omega\big) \big).
\end{split}
 \end{equation*}    
It follows from  the comparison principle, that $u$ satisfies 
    \begin{align*}
       \begin {split} 
       \phi^-(x) \leq u(t,x) \leq \phi^+(x) \  \text{for} \ (t, x) \in [t_0,T] \times  \cN^{\eps}.
       \end{split}
    \end{align*}
Recalling that $\phi^-(0) = 0 = \phi^+(0)$, we deduce that there exists a constant $C_0$, which is independent of $m$,  such that $u$ satisfies 
$$|Du| \leq C_0 \ \text{on} \ [t_0,T] \times \partial \Omega.$$

Propagating this gradient bound in the interior when $H$ satisfies the additional structure condition in \eqref{assump.main1} is a straightforward and  much simpler version of the proof of Proposition \ref{prop.DiVNK}, and so is omitted.  \end{proof}

The preceding proposition shows that optimal controls are bounded, uniformly in the initial condition. Thus, without loss of generality, we may assume that all controls are bounded. More precisely, we have the following corollary of Proposition \ref{prop.alphabounded}.  Since its proof is a mathematical repetition of the first sentence of this paragraph, it is omitted.

\begin{corollary} \label{cor.Ur}
    There is a constant $R_0$ such that, for all $R \geq R_0$, we have $U = U^R$.
\end{corollary}

The next result is about the continuity of $U$.

\begin{proposition}
    \label{prop.d1Lip}
    There is a constant $C$ such that 
    \begin{align}
       |U(t,m) - U(s,n)| \leq C \Big(|t-s|^{1/2} + \bd(m,n) \Big).
    \end{align}
\end{proposition}

\begin{proof}
   We first establish  the regularity in $m$.  Fix  $t_0 \in [0,T)$, $m_0,m_1 \in \sub$,  let  $\alpha_0$ be an optimal control started from $(t_0,m_0)$,  and 
 let $m_t^0$, $m_t^1$ be solutions to 
   \begin{align*}
       \partial_t m_t^i = \Delta m_t^i - \text{div}(m_t^i \alpha_0) \ \text{in} \ (t_0,T] \times \Omega, \quad 
         m^i_t=0 \ \text{in} \ (t_0,T] \times {\partial \Omega}, \quad m_{t_0}^i = m_i,  \quad i = 1,2. 
   \end{align*}
   
   It follows from  the optimality of $\alpha_0$ that 
   \begin{equation} \label{m0m1bound}
   \begin{split}
        U(t_0,m_1) - &U(t_0,m_0) \leq \int_{t_0}^{T} \int_{\Omega} L\big(x,\alpha_0(t,x)\big) d(m_t^1- m_t^0)\\[1.5mm]
        & + \int_{t_0}^T \big(F(m_t^1) - F(m_t^0) \big) dt + G(m_T^1) - G(m_T^0). 
   \end{split}
   \end{equation}
   Next, we claim that there is a constant $C$ independent of $t_0,m_0,m_1$ such that
   \begin{align} \label{fpstability}
        \sup_{t_0 \leq t \leq T} \bd(m_t^1, m_t^0) \leq C \bd(m_1,m_0). 
   \end{align}
   By Lemma \ref{lem.duality}, for any test function $f \in C_c^{\infty}(\Omega)$, 
   \begin{align*}
       \int f d(m_t^1 - m_t^2) = \int \phi(t_0,\cdot) d(m_1 - m_0), 
   \end{align*}
   where $\phi$ solves the dual equation on $[t_0,t_1] \times \Omega$ with terminal condition $f$, and satisfies the estimate
   \begin{align*}
       \|D \phi\|_{\infty} \leq C, \quad C = C(\| D\phi\|_{\infty}, \|\alpha\|_{\infty}). 
   \end{align*}
    The bound \eqref{fpstability} follows, and so 
    \begin{align} \label{fgbound}
       \big| F(m_t^1) - F(m_t^0) \big| \leq C \bd(m_0,m_1), \quad \big| G(m_t^1) - G(m_t^0)  \big| \leq C \bd(m_0,m_1). 
    \end{align}
    Next, we note that, again by duality, we can write
    \begin{align} \label{gduality}
        \int_{t_0}^T \int_{\Omega} L\big(x, \alpha_0(t,x) \big) d(m_t^1 - m_t^0) = \int_{\Omega} \psi(t_0,\cdot) d(m_t^1 - m_t^0), 
    \end{align}
    where $\psi : [t_0,t] \times \ov{\Omega} \to \R$ satisfies 
    \begin{align*}
        \partial_t \psi + \Delta \psi + \alpha_0 \cdot D\psi = L\big(x, \alpha_0(t,x) \big) \  \text{in} \                                                                     \ [t_0,t] \times \Omega, \quad \psi(t,\cdot) = 0, \quad      \psi=0 \ \text{on} \ [t_0,t]\times \partial \Omega.
    \end{align*}
    Again, Lemma \ref{lem.duality} allows us to deduce from  \eqref{gduality} that
    \begin{align} \label{lbound}
         \int_{t_0}^T \int_{\Omega} L\big(x, \alpha_0(t,x) \big) d(m_t^1 - m_t^0)  \leq C \bd(m_0,m_1). 
    \end{align}
    Combining \eqref{m0m1bound}, \eqref{fgbound} and \eqref{lbound}, we obtain the spatial Lipschitz estimate 
    \begin{align*}
      |U(t_0,m_0) - U(t_0,m_1)| \leq C \bd(m_0,m_1). 
    \end{align*}
    
   The regularity in time follows in a standard way from the regularity in space, via dynamic programming. We omit the details.  

\end{proof}

\subsection{The viscosity solution property and the proof of Theorem~\ref{thm.conv}.}
The subject of this subsection is to show that the value function is a viscosity solution to \eqref{hjbinf} and provide the proof of Theorem~\ref{thm.conv}.

For the first part, that is, the viscosity property, we work with the truncated version $U^R$ and then use Corollary~\ref{cor.Ur} to conclude for $U$. 

\begin{proposition} \label{prop.u.subsolution}
    For each $R > 0$, the value function $U^R$ is a viscosity subsolution to  \eqref{hjbinfR}.
\end{proposition}

\begin{proof}
    Recall that  the proof of Proposition~\ref{prop.d1Lip} implies that  $U^R$ is uniformly continuous for each $R > 0$.
    
    Next, let $\Phi$ be a smooth test function, $\delta > 0$, and suppose that $(t_0,m_0) \in [0,T) \times \big( \sub \cap L^2 \big)$ satisfies \eqref{touchingfromabove}, but with $U^R$ replacing $U$. Fix a measurable  function $\alpha:\Omega\to \R^d$  such that $|\alpha|\leq R$ and let $m$ be the solution to 
 \begin{align*}           \partial_t m = \Delta m - \text{div}(m\alpha) \ \text{in} \ (t_0,T]\times \Omega, \quad m_{t_0}=m_0.  \quad m=0 \ \text{on} \ [t_0,T] \times \partial \Omega.
    \end{align*}
    By dynamic programming (Lemma \ref{lem.dpp}), for each $h > 0$, we have 
    \begin{equation} \label{Phi.dpp}
    \begin{split}
        \Phi(t_0,m_0)& + \delta \|m_0\|_2^2 \leq \Phi(t_0 + h, m_{t_0 + h}) \\[1.5mm]
       & + \delta \|m_{t_0 +h}\|_2^2 + \int_{t_0}^{t_0 + h} \bigg( \int_{\Omega} L\big(x,\alpha(x) \big)m_t(dx) + F(m_t) \bigg) dt.
    \end{split}
    \end{equation}
    Combining \eqref{Phi.dpp} with Proposition \ref{prop.FP} we obtain
    \begin{equation} \label{almost.dpp0}
    \begin{split}
        &- \int_{t_0}^{t_0 + h} \partial_t \Phi(t, m_t) dt + 2\delta \int_{t_0}^{t_0 + h} \int_{\Omega} \big( |Dm_t(x)|^2 - Dm_t \cdot \alpha \, m_t \big) dx dt 
     \\[1.5mm]
        &+ \int_{t_0}^{t_0 + h} D_x \frac{\delta \Phi}{\delta m}(m_t,x) \cdot Dm_t(x) dx  dt \\[1.5mm]
        &+ \int_{t_0}^{t_0+h} \int_{\Omega} \Big(- L(x,\alpha(x)) - D_x \frac{\delta \Phi}{\delta m}(m_t,x) \cdot \alpha(x) \Big)m_t(x) dx dt \leq 0. 
    \end{split}
    \end{equation}
    Dividing by $h$, applying Young's inequality, and  using the fact that $\|m_t\|_2^2 \leq C(\|\alpha\|_{\infty}) \|m_0\|_2^2$, we find
    \begin{equation} \label{almost.dpp}
    \begin{split}
        &- \frac{1}{h} \int_{t_0}^{t_0 + h} \partial_t \Phi(t, m_t) dt +  \frac{\delta}{h} \int_{t_0}^{t_0 + h} \int_{\Omega} |Dm_t(x)|^2 dx dt \\[1.5mm]
        &+ \frac{1}{h} \int_{t_0}^{t_0 + h} \int_{\Omega} D_x \frac{\delta \Phi}{\delta m}(m_t,x) \cdot Dm_t(x) dx  dt 
       \\
        &\quad + \frac{1}{h} \int_{t_0}^{t_0+h} \int_{\Omega} \Big(- L(x,\alpha(x)) - D_x \frac{\delta \Phi}{\delta m}(m_t,x) \cdot \alpha(x) \Big)m_t(x) dx dt \leq C(\|\alpha\|_{\infty}) \delta \|m_0\|_2^2.
    \end{split}
    \end{equation}
    
     Next,  for each $h > 0$, set 
     \begin{align*}
         \mu_h = \frac{1}{h} \int_{t_0}^{t_0 + h} m_t.
     \end{align*}
     Combining the bound $\|m_t\|_2^2 \leq C(\|\alpha\|_{\infty}) \|m_0\|_2^2$ with \eqref{almost.dpp} yields, 
     for some constant $C$, 
     \begin{align} \label{muhbound}
         \int_{\Omega} |D\mu_h|^2 dx = \int_{\Omega} \bigg| \frac{1}{h} \int_{t_0}^{t_0 + h} Dm_t dt \bigg|^2 dx \leq \frac{1}{h} \int_{t_0}^{t_0 + h} \int_{\Omega} |Dm_t|^2 dx dt \leq C.
     \end{align}
    In view of the assumed  continuity of $t \mapsto m_t$ in $L^2$, clearly we have $\mu_h \to m_0$ in $L^2$, and so \eqref{muhbound} implies that in fact $\mu_h \rightharpoonup m_0$ weakly in $H_0^1$.  In particular, this proves that $m_0 \in H_0^1$.

The final goal is to pass to the limit in \eqref{almost.dpp}, for which we start by noting that, in view of the fact that, as $t\to 0$,   $m_t \to m_0$ strongly in $L^2$, we have, in the limit $h\to 0 $, 
     \begin{equation} \label{easyterms}
     \begin{split}
         &\frac{1}{h} \int_{t_0}^{t_0 + h} \partial_t \Phi(t,m_t) dt \xrightarrow{h \downarrow{0}} \partial_t \Phi(t_0,m_0),
    \\
         &\frac{1}{h} \int_{t_0}^{t_0+h} \int_{\Omega} \Big( L(x,\alpha(x)) + D_x \frac{\delta \Phi}{\delta m}(m_t,x) \cdot \alpha(x) \Big) m_t(x) dx dt 
      \\
         &\qquad \to  \int_{\Omega} \Big(L(x, \alpha(x)) + D_x \frac{\delta \Phi}{\delta m}(m_0,x) \Big) m_0(x) dx dt.
     \end{split}
     \end{equation}
     Meanwhile, since  $\mu_h \rightharpoonup m_0$ weakly in $H_0^1$, \eqref{muhbound} yields 
     \begin{align} \label{h1bound}
         \int_{\Omega} |Dm_0|^2 dx \leq \liminf_{h \to 0} \frac{1}{h} \int_{t_0}^{t_0+h} \int_{\Omega} |Dm_t(x)|^2dx dt. 
     \end{align}
     Next, we estimate
     \begin{align*}
        &\Big| \frac{1}{h} \int_{t_0}^{t_0+ h} \int_{\Omega} D_x \frac{\delta \Phi}{\delta m}(m_t,x) \cdot Dm_t(x) dx dt  - \int_{\Omega} D_x \frac{\delta \Phi}{\delta m}(m_0,x) \cdot Dm_0(x) dx dt\Big|
        \\[1.5mm]
        &\quad \leq \Big| \frac{1}{h} \int_{t_0}^{t_0+ h} \int_{\Omega} \Big(D_x \frac{\delta \Phi}{\delta m}(m_t,x) -  D_x \frac{\delta \Phi}{\delta m}(m_0,x) \Big) \cdot Dm_t(x) dx dt\Big|  
        \\[1.5mm]
        &\qquad +  \Big|\int_{\Omega} D_x \frac{\delta \Phi}{\delta m}(m_0,x) \cdot \big(D \mu_h - Dm_0\big) dx\Big|
        \\[1.5mm]
        &\quad \leq \Big( \frac{1}{h} \int_{t_0}^{t_0 + h} \big| D_x \frac{\delta \Phi}{\delta m}(m_t,\cdot) - D_x \frac{\delta \Phi}{\delta m}(m_0,\cdot)\big|^2 \Big)^{1/2} \Big( \frac{1}{h} \int_{t_0}^{t_0 +h} |Dm_t(x)|^2 dx dt  \Big)^{1/2}
        \\[1.5mm]
        &\qquad +  \Big|\int_{\Omega} D_x \frac{\delta \Phi}{\delta m}(m_0,x) \cdot \big(D \mu_h - Dm_0\big) dx\Big|
     \end{align*}
     The first term in the tight hand side of the last inequality vanishes as $h \to 0$ because $\frac{1}{h} \int_{t_0}^{t_0 +h} |Dm_t(x)|^2 dx dt$ is bounded and $D_x \frac{\delta \Phi}{ \delta m}(m_t,\cdot) \to D_x \frac{\delta \Phi}{ \delta m}(m_0,\cdot)$ strongly in $L^2$. The second term vanishes because of the weak  in $H^1_0$ convergence of the $\mu_h $'s to $m_0$. We thus have
     \begin{align} \label{hardterm}
   \underset{h\to 0} \lim \; \frac{1}{h} \int_{t_0}^{t_0+ h} \int_{\Omega} D_x \frac{\delta \Phi}{\delta m}(m_t,x) \cdot Dm_t(x) dx dt = \int_{\Omega} D_x \frac{\delta \Phi}{\delta m}(m_0,x) \cdot Dm_0(x) dx dt. 
     \end{align}
     Combining \eqref{almost.dpp}, \eqref{easyterms}, \eqref{h1bound}, and \eqref{hardterm}, we find that
     \begin{align} \label{subsol.dpp}
        &- \partial_t \Phi(t_0,m_0) dt + \delta \int_{\Omega} |Dm_0(x)|^2 dx + \int_{\Omega} D_x \frac{\delta \Phi}{\delta m}(m_0,x) \cdot Dm_0(x) dx  dt 
      \nonumber   \\
        &\quad + \int_{\Omega} \Big(- L(x,\alpha(x)) - D_x \frac{\delta \Phi}{\delta m}(m_0,x) \cdot \alpha(x) \Big)m_0(x) dx \leq C\delta \|m_0\|_2^2, 
    \end{align}
    and, finally, after  choosing $\alpha$ such that 
    \begin{align*}
        \alpha(x) \in \text{argmax} \Big\{ B_R \ni a \mapsto - L(x,a) - D_x \frac{\delta \Phi}{\delta m}(m_0,x) \cdot a \Big\}
    \end{align*}
    completes the proof.  
    \end{proof}

Now we deal with the supersolution property which is, typically, the harder case.

\begin{proposition} \label{prop.u.supersolution}
    For each $R> 0$, value function $U^R$ is a viscosity supersolution to \eqref{hjbinfR}.
\end{proposition}

\begin{proof}
     Let $\Phi \in C^{1,2}([0,T] \times H^{-1})$, $\delta > 0$, and suppose that $(t_0,m_0) \in [0,T) \times \big( \sub \cap L^2 \big)$ satisfies \eqref{touchingfrombelow}, but with $U^R$ replacing $U$. Let $(m,\alpha)$ be an optimizer for the problem defining $U^R(t_0,m_0)$, so that by dynamic programming (Lemma \ref{lem.dpp}), we have,  for any $h \in (0,T-t_0)$, 
     \begin{align*}
         U^R(t_0,m_0) = \int_{t_0}^{t_0 + h}\bigg( \int_{\Omega} L\big( x, \alpha(t,x) \big)m_t(dx) + F(m_t) \bigg) dt + U^R(t_0 + h, m_{t_0 + h}).
     \end{align*}
    It follows that
     \begin{align*}
         \Phi(t_0,m_0) + \delta \|m_0\|_2^2 \geq \int_{t_0}^{t_0 + h}\bigg( \int_{\Omega} L\big( x, \alpha(t,x) \big)m_t(dx) + F(m_t) \bigg) dt + \Phi(t_0,m_0) + \delta \|m_{t_0 + h}\|_2^2.
     \end{align*}
     Arguing as in the proof of Proposition \ref{prop.u.subsolution}, we deduce that 
     \begin{align} \label{almost.dpp2}
     & - \frac{1}{h} \int_{t_0}^{t_0 + h} \partial_t \Phi(t, m_t) dt - \frac{\delta}{h} \int_{t_0}^{t_0 + h} \int_{\Omega} |Dm_t(x)|^2 dx dt + \frac{1}{h} \int_{t_0}^{t_0 + h} \int_{\Omega} D_x \frac{\delta \Phi}{\delta m}(m_t,x) \cdot Dm_t(x) dx  dt 
      \nonumber   \\
        &\qquad  + \frac{1}{h} \int_{t_0}^{t_0+h} \int_{\Omega}H\Big(x, D_x \frac{\delta \Phi}{\delta m}(m_t,x) \Big)m_t(x) dx dt
      \nonumber   \\
        &\geq - \frac{1}{h} \int_{t_0}^{t_0 + h} \partial_t \Phi(t, m_t) dt - \frac{\delta}{h} \int_{t_0}^{t_0 + h} \int_{\Omega} |Dm_t(x)|^2 dx dt + \frac{1}{h} \int_{t_0}^{t_0 + h} \int_{\Omega} D_x \frac{\delta \Phi}{\delta m}(m_t,x) \cdot Dm_t(x) dx  dt 
      \nonumber   \\
        &\quad + \frac{1}{h} \int_{t_0}^{t_0+h} \int_{\Omega} \Big(- L(x,\alpha(x)) - D_x \frac{\delta \Phi}{\delta m}(m_t,x) \cdot \alpha(t,x) \Big)m_t(x) dx dt \geq -C(\|\alpha\|_{\infty}) \delta \|m_0\|_2^2.
    \end{align}
    Following again the proof of Proposition~\ref{prop.u.subsolution}, we find that $\mu_h = \frac{1}{h} \int_{t_0}^{t_0 + h} m_t dt$
    is bounded in $H_0^1$, and converges to $m_0$ as $h \to 0$, so that in particular $m_0 \in H_0^1$.

    We now pass to the limit in \eqref{almost.dpp2} exactly as in the proof of Proposition \ref{prop.subsol}. The only new term is the one containing the Hamiltonian. For this term, we note that the Lipschitz continuity of $H^R$ and the fact that, as $t\to 0$,  $D_x \frac{\delta \Phi}{\delta m} (m_{t},\cdot) \to D_x \frac{\delta \Phi}{\delta m} (m_{t_0},\cdot)$ strongly in $L^2$ imply that, as $t\to 0$ and strongly in $L^2$, 
    \begin{align*}
        H^R\Big(\cdot, D_x \frac{\delta \Phi}{\delta m} (m_{t},\cdot)\Big) \to  H^R\Big(\cdot, D_x \frac{\delta \Phi}{\delta m} (m_{t_0},\cdot) \Big).
    \end{align*}
    Since  $m_t \to m_0$ strongly in $L^2$ as $t\to t_0$, it follows that 
    \begin{align*}
        \frac{1}{h} \int_{t_0}^{t_0 + h} \int_{\Omega} H^R \Big(x, D_x \frac{\delta \Phi}{\delta m} (m_t,x) \Big) m_t dx dt \xrightarrow{h \downarrow{0}} \int_{\Omega} H^R\Big(x,D_x \frac{\delta \Phi}{\delta m} (m_0,x) \Big)m_0(x) dx.
    \end{align*}
    Hence,  we can indeed pass  in  the  $h \to 0$ limit in \eqref{almost.dpp2} to find that
     \begin{align*}
      &  - \partial_t \Phi(t_0,m_0) - \delta \int_{\Omega} |Dm_0(x)|^2 dx + \int_{\Omega} D_x \frac{\delta \Phi}{\delta m}(t_0,m_0,x) \cdot Dm_0(x)dx
        \\
       &\qquad  + \int_{\Omega} H^R\Big( x, D_x \frac{\delta \Phi}{\delta m}(t_0,m_0,x)\Big) m_0(x) dx \geq F(m_0) - C \delta \|m_0\|_{L^2}^2, 
    \end{align*}
    which completes the proof.  
    \end{proof}

We are now in position to give the proof of the second main result of the paper.

\begin{proof}[Proof of Theorem \ref{thm.value}.]
    It follows from a straightforward combination of Proposition~\ref{prop.d1Lip}, Corollary~\ref{cor.Ur}, and Proposition~ \ref{prop.u.subsolution} and Proposition~ \ref{prop.u.supersolution}  
    \end{proof}

As we already mentioned in the introduction, for the convergence problem, it is convenient to work in $N-$dimensional sub-probability empirical measures $\sub^N = \{m_{\bx}^{N,K} : K = 1,...,N, \, \bx \in \Omega^K\}$ introduced in \eqref{def.pnsub}
By Theorem~\ref{thm.vnk.lip}, there is a constant $C_0$ which is independent of $N$ such that, for all
$t,s \in [0,T] $ and $m,n \in \sub^N$, the map $\tilde{V}^N$ defined in \eqref{def.tildev} satisfies the bound 
\begin{align} \label{tildevn.lip}
    |\tilde{V}^{N}(t,m) - \tilde{V}^N(s,n)| \leq C_0\Big(|t-s|^{1/2} + \bd(m,n) \Big). 
\end{align}

Hence, we can just extend $\tilde{V}^N$ to a map $\hat{V}^N : [0,T] \times \sub \to \R$ such that,  for all
$t,s \in [0,T] $ and $m,n \in \sub$, 
\begin{align*}
    |\hat{V}^{N}(t,m) - \hat{V}^N(s,n)| \leq C_0\Big(|t-s|^{1/2} + \bd(m,n) \Big). 
\end{align*}
Because $(\sub, \bd)$ is compact, we can apply the Arzel\'a-Ascoli Theorem to the sequence $(\hat{V}^N)_{N \in \N}$, and obtain the following fact, which we record as a lemma.

\begin{lemma} \label{lem.compactness}
    For any sequence $(N_k)_{k \in \N}$, there exists a subsequence $(N_{k_j})_{j \in \N}$, and a function $V : [0,T] \times \sub \to \R$ satisfying, $t,s \in [0,T] $ and $m,n \in \sub$,
   \begin{align} \label{vlip}
    |V(t,m) - V(s,n)| \leq C_0\Big(|t-s|^{1/2} + \bd(m,n) \Big),  
\end{align}
    such that 
    \begin{align} \label{limitpoint}
        \lim_{j \to \infty} \max_{K = 1,...,N_{k_j}} \sup_{t \in [0,T], \, \bx \in \Omega^K} \Big|V(t,m_{\bx}^{N_{k_j},K}) - V^{N_{k_j},K}(t,\bx) \Big| = 0.
    \end{align}
\end{lemma}

We now complete the proof of Theorem \ref{thm.conv}. 

\begin{proof}[Proof of Theorem \ref{thm.conv}.]
    Pick $R$ large enough so that Proposition \ref{prop.subsol} holds. Combining  Theorem~\ref{thm.comparison} about the uniqueness of viscosity solution to \eqref{hjbinfR}  together with Proposition~\ref{prop.subsol}, Proposition~\ref{prop.u.subsolution} and Proposition~ \ref{prop.u.supersolution}, we find that for any sequence $(N_k)_{k \in \N}$, there exists a subsequence $(N_{k_j})_{j \in \N}$,
    such that 
    \begin{align*}
        \lim_{j \to \infty} \max_{K = 1,...,N_{k_j}} \sup_{t \in [0,T], \, \bx \in \Omega^K} \Big|U(t,m_{\bx}^{N_{k_j},K}) - V^{N_{k_j},K}(t,\bx) \Big| = 0.
    \end{align*}The result follows.  
    \end{proof}

\bibliographystyle{alpha}
\bibliography{boundary}

\end{document}